\documentclass[onecolumn,final]{elsart3p}

\usepackage{comment}
\usepackage{amsmath}       
\usepackage{amssymb}       
\usepackage{chapterbib}    
\usepackage{cite}
\usepackage{color}
\usepackage{comment}
\usepackage{bm,bbm}
\usepackage{stackengine}
\usepackage{latexsym}
\usepackage{overpic}
\usepackage{times}
\usepackage{epsfig}
\usepackage{graphicx,graphics,rotating}
\usepackage{subfigure}
\usepackage{multicol}
\usepackage{multirow}

\usepackage{ifthen}

\usepackage{titlesec}
\titlespacing*{\section}
{0pt}{3.8ex plus 0.75ex minus .2ex}{3.0ex plus .2ex}
\titlespacing*{\subsection}
{0pt}{3.0ex plus 0.75ex minus .2ex}{2.3ex plus .2ex}

\theoremstyle{plain}
\newtheorem{theorem}{Theorem}[section]
\newtheorem{lemma}[theorem]{Lemma}

\theoremstyle{definition}

\newtheorem{remark}[theorem]{Remark}

\theoremstyle{plain}
\newtheorem{assumption}[theorem]{Assumption}

\definecolor{MyDarkGreen}{rgb}{0,0.45,0}

\newcommand{\RED}[1]{{\color{red}#1}}

\def\trait #1 #2 #3 {\vrule width #1pt height #2pt depth #3pt}
\def\fin{\hfill
        \trait .3 5 0
        \trait 5 .3 0
        \kern-5pt
        \trait 5 5 -4.7
        \trait 0.3 5 0
\medskip}

\newenvironment{proof}{\textit{Proof.}}{\fin}

\newcommand{\TERM}[2]{\big(\textsf{#1}_{#2}\big)}
\newcommand{\DOFS}[2]{(\textbf{#1}-\textbf{#2})}
\newcommand{\MESH}[1]{\mathcal{M}#1}

\newcommand{\INTP}{\footnotesize{I}}

\newcommand{\REAL}{\mathbbm{R}}

\newcommand{\restrict}[2]{{#1}_{|{#2}}}
\newcommand{\STACKON}[2]{\stackon{$#1$}{\footnotesize{#2}}}

\newcommand{\PGRAPH}[1]{\medskip\noindent\textbf{#1}.}
\newcommand{\EOD}{\end{document}}

\renewcommand{\cv}{\mathbf{c}}

\newcommand{\fv}{\mathbf{f}}

\newcommand{\nv}{\mathbf{n}}

\newcommand{\qv}{\mathbf{q}}

\newcommand{\uv}{\mathbf{u}}
\newcommand{\vv}{\mathbf{v}}
\newcommand{\wv}{\mathbf{w}}
\newcommand{\xv}{\mathbf{x}}
\newcommand{\yv}{\mathbf{y}}

\newcommand{\Vv}{\mathbf{V}}

\newcommand{\as}{a}
\newcommand{\bs}{b}

\newcommand{\hs}{h}

\newcommand{\ks}{k}

\newcommand{\ns}{n}

\newcommand{\ps}{p}
\newcommand{\qs}{q}

\renewcommand{\ss}{s}
\newcommand{\ts}{t}
\newcommand{\us}{u}
\newcommand{\vs}{v}

\newcommand{\xs}{x}
\newcommand{\ys}{y}

\newcommand{\Cs}{C}

\newcommand{\Qs}{Q}

\newcommand{\Ss}{S}

\newcommand{\Vs}{V}

\newcommand{\calM}{\mathcal{M}}

\newcommand{\calO}{\mathcal{O}}

\newcommand{\HONE}  {H^1}
\newcommand{\HONEzr}{H^1_0}

\newcommand{\HTWO}  {H^2}

\newcommand{\LTWO}  {L^2}

\newcommand{\LTWOzr}{L^2_0}

\newcommand{\HS}[1] {H^{#1}}

\newcommand{\CS}[1] {C^{#1}}

\newcommand{\PS}[1] {\mathbbm{P}_{#1}}

\newcommand{\Vvhk}{\Vv^{h}_{k}}

\newcommand{\Vvh}[1]{\Vv_{#1}^{h}}

\newcommand{\Qshkk}{\Qs^{\hh}_{k-1}}

\newcommand{\Vshkp}  {\Vs^{\FO,h}_{k}}
\newcommand{\Vvhkp}  {\Vv^{\FO,h}_{k}}

\newcommand{\Vvhkpp}{\Vv^{\FT,h}_{k}}
\newcommand{\Vvhpp}[1]{\Vv^{\FT,h}_{#1}}

\newcommand{\FO}{\textit{F1}}
\newcommand{\FT}{\textit{F2}}

\newcommand{\VvhkFO}{\Vv^{\FO,h}_{k}}

\newcommand{\VvhkFT}{\Vv^{\FT,h}_{k}}

\renewcommand{\P} {\textrm{E}}            
\newcommand  {\E} {\textrm{e}}
\newcommand  {\V} {\textrm{v}}            

\newcommand{\hh}{h}
\newcommand{\Th}{\Omega_{\hh}}

\newcommand{\xvP}{\xv_{\P}}        

\newcommand{\dims}{2}              

\newcommand{\hP}{\hh_{\P}}

\newcommand{\hE}{\hh_{\E}}

\newcommand{\mP}{\ABS{\P}}

\newcommand{\mE}{\ABS{\E}}

\newcommand{\Eset}{\mathcal{E}}    
\newcommand{\Vset}{\mathcal{V}}    

\newcommand{\NMB}{N}

\newcommand{\NPV}{\NMB^{\Vset}_{\P}}      
\newcommand{\NPE}{\NMB^{\Eset}_{\P}}      

\newcommand{\dV}{\,d\xv}
\newcommand{\dS}{\,ds}
\newcommand{\ds}{\,ds} 

\newcommand{\DIV}  {\text{div}\,}

\newcommand{\norE} {\mathbf{n}_{\E}}

\newcommand{\norPE}{\mathbf{n}_{\P,\E}}

\newcommand{\norEx} {\ns_{\E,x}}
\newcommand{\norEy} {\ns_{\E,y}}

\newcommand{\tngE} {\mathbf{t}_{\E}}

\newcommand{\X}   {\mathbf{x}}

\newcommand{\xV}{\X_{\V}}

\newcommand{\bil}[2]{\langle#1,#2\rangle}

\newcommand{\scal}  [2]{(#1,#2)}

\newcommand{\abs}   [1]{|#1|}

\newcommand{\ABS}   [1]{\left|#1\right|}

\newcommand{\snorm}  [2]{|#1|_{#2}}

\newcommand{\norm}  [2]{||#1||_{#2}}

\newcommand{\Piz}[1]{\Pi^{0}_{#1}}
\newcommand{\PinP}[1]{\Pi^{\nabla,\P}_{#1}}
\newcommand{\PizP}[1]{\Pi^{0,\P}_{#1}}

\newcommand{\PiF} {\pi_{F}}

\newcommand{\Pio}{\pi_{1}}
\newcommand{\PioP}{\pi_{1}^{\P}}

\newcommand{\Pit}{\pi_{2}}
\newcommand{\PitP}{\pi_{2}^{\P}}

\newcommand{\SPh}{\Ss^{\P}_{\hh}}

\newcommand{\ash}{\as_{\hh}}
\newcommand{\bsh}{\bs_{\hh}}

\newcommand{\asPh}{\as^{\P}_{\hh}}
\newcommand{\bsPh}{\bs^{\P}_{\hh}}

\newcommand{\asP}{\as^{\P}}
\newcommand{\bsP}{\bs^{\P}}

\newcommand{\psh}{p_{\hh}}

\newcommand{\psI}{p_{\INTP}}

\newcommand{\qsh}{\qs_{\hh}}

\newcommand{\qsI} {\qs_{\INTP}}

\newcommand{\vsh} {\vs_{\hh}}

\newcommand{\uvh} {\uv_{\hh}}
\newcommand{\uvI} {\uv_{\INTP}}

\newcommand{\vvh} {\vv_{\hh}}
\newcommand{\vvI} {\vv_{\INTP}}

\newcommand{\vshx} {\vs_{\hh,x}}
\newcommand{\vshy} {\vs_{\hh,y}}

\newcommand{\wvh} {\wv_{\hh}}

\newcommand{\fvh} {\fv_{\hh}}

\newcommand{\qvh} {\qv_{\hh}}

\newcommand{\xVp} {\xV^{\prime}}
\newcommand{\xVpp}{\xV^{\prime\prime}}

\newcommand{\btauh}{\bm\tau_{\hh}}

\newcommand{\kb}{\bar{k}}

\newcommand{\dvh}{{\bm\delta}_{\hh}}
\newcommand{\ssh}{\sigma_{\hh}}

\newcommand{\psiv}{{\bm\Psi}}
\newcommand{\phis}{\varphi}

\newcommand{\psivI}{{\bm\Psi}_{\INTP}}
\newcommand{\phisI}{\varphi_{\INTP}}

\newif\ifARXIV
\ARXIVtrue

\usepackage{xcolor}

\begin{document}
%
\begin{frontmatter} 
  \title{Conforming virtual element approximations of the
    two-dimensional Stokes problem}

  \author[IMATI] {Gianmarco Manzini}
  \author[DICEA]{and Annamaria Mazzia}

  \address[IMATI]{
    Istituto di Matematica Applicata e Tecnologie Informatiche,
    Consiglio Nazionale \\delle Ricerche,
    via Ferrata 1,
    27100 Pavia,
  }

    
  \address[DICEA]{Dipartimento di Ingegneria Civile, Edile e Ambientale - ICEA,
    Universit\`a di Padova, 35131 Padova, Italy}
  
  \begin{abstract}
    The virtual element method (VEM) is a Galerkin approximation
    method that extends the finite element method to polytopal meshes.
    In this paper, we present two different conforming virtual element
    formulations for the numerical approximation of the Stokes problem
    that work on polygonal meshes.
    The velocity vector field is approximated in the virtual element
    spaces of the two formulations, while the pressure variable is
    approximated through discontinuous polynomials.
    Both formulations are inf-sup stable and convergent with optimal
    convergence rates in the $\LTWO$ and energy norm.
    We assess the effectiveness of these numerical approximations by
    investigating their behavior on a representative benchmark
    problem.
    The observed convergence rates are in accordance with the
    theoretical expectations and a weak form of the zero-divergence
    constraint is satisfied at the machine precision level.
  \end{abstract}
  

\end{frontmatter}

\renewcommand{\arraystretch}{1.}
\raggedbottom



\section{Introduction}
\label{sec:intro}

Many physical phenomena in physics and engineering can be modeled by
the Stokes flow~\cite{Galdi:2011}.
Noteworthy applications are, for example, Stokes flows in porous
media~\cite{Bang-Lukkassen:1999}, design and development of efficient
fibrous filters~\cite{Linden-Cheng-Wiegmann:2018} and micro-fluid
devices~\cite{Smith-Barbati-Santana-Gleghon-Kirby:2012}, dynamics of
droplets~\cite{Kitahata-Yoshinaga-Nagai-Sumino:2013}, bio-suspensions
and sedimentation~\cite{Hofer:2018}.
A very successful approach for the numerical treatment of the Stokes
equations in variational form is based on the finite element method
(FEM)~\cite{Cai-Tong-Vassilevski-Wang:2010,Crouzeix-Raviart:1973,Girault-Raviart:1986}.
The FEM normally uses triangular and quadrilateral meshes in the
two-dimensional (2D) case and tetrahedral and hexahedral meshes in the
three-dimensional case (3-D).
Furthermore, in the last two decades a great effort has been devoted
in the design of numerical methods for partial differential equations
(PDEs) suitable to polygonal and polyhedral
meshes~\cite{Wachspress:2015,
  Kuznetsov-Repin:2003,
  Sukumar-Tabarraei:2004,
  BeiraodaVeiga-Lipnikov-Manzini:2014}.
To this end, it is worth mentioning the mimetic finite the difference
(MFD)
method~\cite{Lipnikov-Manzini-Shashkov:2014,BeiraodaVeiga-Lipnikov-Manzini:2014}
and its variational reformulation that led to the virtual element
method
(VEM)~\cite{BeiraodaVeiga-Brezzi-Cangiani-Manzini-Marini-Russo:2013}.
The MFD was designed to preserve several fundamental properties of
PDEs, such as the maximum/minimum principle, the conservation of
fundamental quantities in physics (mass, momentum, energy) and the
solution symmetries.
The MFD method was successfully applied to the numerical approximation
on unstructured polygonal and polyhedral meshes of diffusion
problems~\cite{Brezzi-Lipnikov-Shashkov-Simoncini:2007,Brezzi-Lipnikov-Shashkov:2006},
convection–diffusion problems~\cite{Cangiani-Manzini-Russo:2009},
elasticity problems~\cite{Lipnikov-Morel-Shashkov:2004},
gas dynamic problems~\cite{Campbell-Shashkov:2001},
and electromagnetic problems~\cite{Hyman-Shashkov:2001}.
On the other hand, the VEM is a finite element method that does not
require the explicit knowledge of the basis functions and use of
quadrature formulas to compute the bilinear forms of the Galerkin
formulation.
Indeed, the VEM can handle the construction of the bilinear forms on
general polygonal and polyhedral elements through special polynomial
projections of the basis functions and their derivatives (gradients,
curl, divergence).
Such projections are computable from the degrees of freedom of the
virtual element functions and ensure the polynomial consistency of the
bilinear forms.
The connection between the VEM and the FEM on polygonal/polyhedral
meshes is thoroughly investigated
in~\cite{Manzini-Russo-Sukumar:2014,Cangiani-Manzini-Russo-Sukumar:2015,DiPietro-Droniou-Manzini:2018},
between VEM and discontinuous skeletal gradient discretizations
in~\cite{DiPietro-Droniou-Manzini:2018}, and between the VEM and the
BEM-based FEM method
in~\cite{Cangiani-Gyrya-Manzini-Sutton:2017:GBC:chbook}.

The VEM was originally formulated
in~\cite{BeiraodaVeiga-Brezzi-Cangiani-Manzini-Marini-Russo:2013} as a
conforming FEM for the Poisson problem.
Then, it was later extended to convection-reaction-diffusion problems
with variable coefficients
in~\cite{Ahmad-Alsaedi-Brezzi-Marini-Russo:2013,BeiraodaVeiga-Brezzi-Marini-Russo:2016b}.
Meanwhile, the nonconforming formulation for diffusion problems was
proposed in~\cite{AyusodeDios-Lipnikov-Manzini:2016} as the finite
element reformulation of~\cite{Lipnikov-Manzini:2014}.
Mixed VEM for elliptic problems were introduced
in~\cite{Brezzi-Falk-Marini:2014}, and later extended to meshes with
curved edges in~\cite{Dassi-Fumagalli-Losapio-Scialo-Scotti-Vacca:2020a}.
Implementation of mixed methods is discussed
in~\cite{Dassi-Vacca:2019,Dassi-Scacchi:2020,Dassi-Fumagalli-Losapio-Scialo-Scotti-Vacca:2020b}.

The connection with de~Rham diagrams and Nedelec elements and the
application to the electromagnetics has been explored
in~\cite{BeiraodaVeiga-Brezzi-Marini-Russo:2016a}.
A practical application of these concepts can be found in
\cite{BeiraodaVeiga-Dassi-Manzini-Mascotto:2021,NaranjoAlvarez-Bokil-Gyrya-Manzini:2020}.
Other significant applications of the VEM on general meshes are found,
for example, in~\cite{%
  Antonietti-Manzini-Verani:2018,%
  Antonietti-Manzini-Verani:2019:CAMWA:journal,%
  Brezzi-Marini:2013,%
  BeiraodaVeiga-Manzini:2014,%
  BeiraodaVeiga-Manzini:2015,%
  BeiraodaVeiga-Mora-Vacca:2019,%
  BeiraodaVeiga-Manzini-Mascotto:2019,%
  Benedetto-Berrone-Pieraccini-Scialo:2014,%
  Benedetto-Berrone-Borio-Pieraccini-Scialo:2016b,%
  Benedetto-Berrone-Scialo:2016,%
  Benvenuti-Chiozzi-Manzini-Sukumar:2019:CMAME:journal,%
  Berrone-Borio-Scialo:2016,%
  Berrone-Pieraccini-Scialo:2016,%
  Berrone-Borio-Manzini:2018:CMAME:journal,%
  Cangiani-Gyrya-Manzini:2016,%
  Cangiani-Georgoulis-Pryer-Sutton:2016,%
  Cangiani-Manzini-Sutton:2017,%
  Certik-Gardini-Manzini-Vacca:2018:ApplMath:journal,
  Certik-Gardini-Manzini-Mascotto-Vacca:2019:CAMWA:journal,%
  Gardini-Manzini-Vacca:2019:M2AN:journal,%
  Mora-Rivera-Rodriguez:2015,%
  Natarajan-Bordas-Ooi:2015,%
  Paulino-Gain:2015,%
  Perugia-Pietra-Russo:2016,%
  Wriggers-Rust-Reddy:2016,%
  Zhao-Chen-Zhang:2016%
}.

In this work, we consider two possible numerical formulations of the
VEM for the discretization of the two-dimensional (2D) Stokes
equation.
In both formulation, we approximate the two components of the velocity
vector separately by using a variant of the conforming virtual element
space originally proposed
in~\cite{BeiraodaVeiga-Brezzi-Cangiani-Manzini-Marini-Russo:2013} and
already considered in~\cite{Manzini-Mazzia:2021}.
In the first formulation we assume that the edge trace of each
component of the velocity is a polynomial of degree $k+1$, where $k$
is the maximum degree of the polynomials that are in the virtual
element space.
This definition of the scalar virtual element space is a special case
of the \emph{generalized local virtual element space} that is proposed
in~\cite[Section~3]{BeiraodaVeiga-Vacca:2020-arXiv}.
In the second formulation, we assume that only the trace of the normal
component of the velocity vector is a polynomial of degree $k+1$,
while the trace of the tangential component is a polynomial of degree
$k$.
For both formulations, we also consider the modified (``enhanced'')
definition of the virtual element
space~\cite{Ahmad-Alsaedi-Brezzi-Marini-Russo:2013}, which allows us
to construct the $\LTWO$ orthogonal projection onto the polynomials of
degree $k$.
In both formulations, the scalar unknown, e.g., the pressure, is
approximated by discontinuous polynomials on the mesh elements.
These two virtual element formulations satisfy the inf-sup stability
condition, which is crucial to prove the well-posedness of the method,
and can be proved to have an optimal convergence rate for the
approximation errors in the $\LTWO$ norm and in the $\HONE$-seminorm.
A similar approach for the incompressible Stokes equations led to the
low-order accurate MFD methods
in~\cite{BeiraodaVeiga-Gyrya-Lipnikov-Manzini:2009,BeiraodaVeiga-Lipnikov:2010},
that are equivalent to the formulations proposed in our work for
$k=1$.

All our numerical experiments confirm the expected optimal behavior of
these two formulations, whose accuracy is comparable, although the
second formulation requires less degrees of freedom than the first
one.
The zero divergence constraint is satisfied in a variational sense,
i.e., the projection of the divergence on the subset of polynomials
used in the scheme formulation is zero.
It is worth mentioning that other virtual element approaches were
recently proposed in the literature that approximate the Stokes
velocity in such a way that its divergence is a polynomial that is set
to zero in the scheme.
This strategy provides an approximation of the Stokes velocity that
satisfies the zero divergence constraint in a pointwise sense.
We refer the interested reader to the works of References~\cite{
  BeiraodaVeiga-Lovadina-Vacca:2017,
  BeiraodaVeiga-Lovadina-Vacca:2018,
  BeiraodaVeiga-Mora-Vacca:2019,
  BeiraodaVeiga-Dassi-Vacca:2020,
  Chernov-Marcati-Mascotto:2021}.
However, the polynomial projection of the velocity divergence in our
VEM is zero up to the machine precision, so if we consider such
projection as the virtual element approximation of the velocity
divergence, this approximation is identically zero almost everywhere
in the computational domain.

\subsection{Structure of the paper}
The outline of the paper is as follows.
In Section~\ref{sec:Stokes}, we introduce the Stokes problem.
In Section~\ref{sec:VEM}, we discuss two different virtual element
formulations for numerically solving this problem.
In Section~\ref{sec:convergence}, we investigate the convergence of
these formulations theoretically, and derive optimal convergence rates
in the energy and $\LTWO$ norms for the velocity approximation and in
the $\LTWO$ norm for the pressure approximation.
In Section~\ref{sec:numerical}, we assess the accuracy of
these virtual element approximations by investigating their behavior
on a  representative benchmark problem.
In Section~\ref{sec:conclusions}, we offer our final conclusions.


\subsection{Notation and technicalities}
\label{subsec:notation}
We use the standard definition and notation of Sobolev spaces, norms
and seminorms, cf.~\cite{Adams-Fournier:2003}.
Let $k$ be a nonnegative integer number.
The Sobolev space $\HS{k}(\omega)$ consists of all square integrable
functions with all square integrable weak derivatives up to order $k$
that are defined on the open, bounded, connected subset $\omega$ of
$\REAL^{2}$.
As usual, if $k=0$, we prefer the notation $\LTWO(\omega)$.
We will also use the subspace of $\LTWO(\Omega)$ denoted by
$\LTWOzr(\Omega)$ and defined on the computational domain $\Omega$ as
\begin{align}
  \LTWOzr(\Omega):=\bigg\{\,\qs\in\LTWO(\Omega)\,:\,\int_{\Omega}\qs\dV=0\,\bigg\}.
\end{align}
Norm and seminorm in $\HS{k}(\omega)$ are denoted by
$\norm{\cdot}{k,\omega}$ and $\snorm{\cdot}{k,\omega}$, respectively.
We use the integral notation to denote the $\LTWO$-inner product
between vector-valued fields, although for notation's conciseness, we
may prefer to use the notation ``$(\cdot,\cdot)$'' in a few situations.


\subsection{Mesh definition and regularity assumptions}
\label{subsec:mesh:regularity:assumptions}
For exposition's sake, we consider an open, bounded, polygonal domain
$\Omega$ and a family of mesh decompositions of $\Omega$ denoted by
$\mathcal{T}=\{\Th\}_{\hh}$.
Each mesh $\Th$ is a set of non-overlapping, bounded (closed) elements
$\P$ such that $\overline{\Omega}=\cup_{\P\in\Th}\P$, where
$\overline{\Omega}$ is the closure of $\Omega$ in $\REAL^2$.
The subindex $\hh$, which labels each mesh $\Th$, is the maximum of
the diameters $\hP=\sup_{\xv,\yv\in\P}\abs{\xv-\yv}$.
Each element $\P$ has a non-intersecting polygonal boundary
$\partial\P$ formed by $\NPE$ straight edges $\E$ connecting the
$\NPV$ ($=\NPE$) polygonal vertices.
The sequence of vertices forming $\partial\P$ is oriented in the
counter-clockwise direction and the vertex coordinates are denoted by
$\xV=(\xs_{\V},\ys_{\V})$.
We denote the measure of $\P$ by $\mP$, its barycenter (center of
gravity) by $\xvP:=(\xs_{\P},\ys_{\P})$, the unit normal vector to
each edge $\E\in\partial\P$ and pointing out of $\P$ by $\norPE$, and
the length of $\E$ by $\hE$.
Moreover, we assume that the orientation of the mesh edges in every
mesh is fixed \emph{once and for all}, so that we can unambiguously
introduce $\norE$, the unit normal vector to edge $\E$.
The orientation of this vector is independent of the element $\P$ to
which $\E$ belongs, and may differ from $\norPE$ only by the
multiplicative factor $-1$.

\PGRAPH{Mesh regularity assumptions}
%
In the definition of the admissible meshes, we first assume that the
elemental boundaries are ``polylines'', i.e., continuously connected
portions of straight lines.
Then, we need the following regularity assumptions on the family of
mesh decompositions $\{\Th\}_{\hh}$ in order to use the interpolation
and projection error estimates from the theory of polynomial
approximation of functions in Sobolev
spaces~\cite{Brenner-Scott:1994}.

\medskip
\begin{assumption}[Mesh regularity]~\\
  \vspace{-\baselineskip}
  \begin{itemize}
  \item There exists a positive constant $\varrho$
    independent of $\hh$ such that for every polygonal element $\P$ it
    holds that
    \begin{description}
    \item[]\textbf{(M1)}~~$\P$ is star-shaped with respect to a disk  with radius $\ge\varrho\hP$;
    \item[]\textbf{(M2)}~~for every edge $\E\in\partial\P$ it holds that
      $\hE\geq\varrho\hP$.
    \end{description}
  \end{itemize}
\end{assumption}

\medskip
\begin{remark}
  The star-shapedness property \textbf{(M1)} implies that all the mesh
  elements are \emph{simply connected} subsets of $\REAL^{2}$.
  The scaling property \textbf{(M2)} implies that the number of edges
  in all the elemental boundaries is uniformly bounded from above over
  the whole mesh family $\{\Th\}_{\hh}$.
\end{remark}

\medskip
These mesh assumptions are quite general and, as observed from the
very first publication on the VEM, see, for example,
\cite{BeiraodaVeiga-Brezzi-Cangiani-Manzini-Marini-Russo:2013}, allow
the method a great flexibility in the geometric shape of the mesh
elements.
For example, we can consider elements with hanging nodes as in the
adaptive mesh refinement (AMR) technique and elements with a
non-convex shape.
In this work we avoid elements with intersecting boundaries, elements
with ``holes'', and elements totally surrounding other elements.
However, elements with such more challenging shapes have already been
considered in the virtual element formulation to show the robustness
of the method~\cite{Paulino-Gain:2015}.
A recent review of the mesh regularity assumptions in the VEM
literature and a thorough investigation of the VEM performance on mesh
families with extreme characteristics can also be found
in~\cite{Sorgente-Biasotti-Manzini-Spagnuolo:2021:arXiv,Sorgente-VEM-Book:2021}.

\subsection{Polynomials}

Hereafter,
$\PS{\ell}(\P)$ denotes the linear space of polynomials of degree up
to $\ell$ defined on $\P$, with the useful convention that
$\PS{-1}(\P)=\{0\}$;
$\big[\PS{\ell}(\P)\big]^2$ denotes the space of two-dimensional
vector-valued fields of polynomials of degree up to $\ell$ on $\P$;
$\big[\PS{\ell}(\P)\big]^{2\times2}$ denotes the space of
$2\times2$-sized tensor-valued fields of polynomials of degree up to
$\ell$ on $\P$.
Similar definitions also hold for the space of univariate polynomials
defined on all mesh edges $\E$.
Then, we define the linear space of discontinuous scalar, vector and
tensor polynomial fields by collecting together the local definitions,
so that
\begin{align*}
  \PS{\ell}(\Th)&:=\Big\{
  \qs\in\LTWO(\Omega)\,:\,\restrict{\qs}{\P}\in\PS{\ell}(\P)\quad\forall\P\in\Th
  \Big\},\\[1em]
  \big[\PS{\ell}(\Th)\big]^2&:=
  \Big\{
  \qv\in\big[\LTWO(\Omega)\big]^2\,:\,\restrict{\qv}{\P}\in\big[\PS{\ell}(\P)\big]^2
  \quad\forall\P\in\Th
  \Big\},\\[1em]
  \big[\PS{\ell}(\Th)\big]^{2\times2}&:=
  \Big\{
  \bm{\kappa}\in\big[\LTWO(\Omega)\big]^{2\times2}\,:\,
  \restrict{{\bm\kappa}}{\P}\in\big[\PS{\ell}(\P)\big]^{2\times2}
  \quad\forall\P\in\Th  
  \Big\}.
\end{align*}
We will also use the norm and seminorm:
\begin{align}
  \norm{\vv}{1,\hh}^2=\norm{\vv}{0,\Omega}^2+\snorm{\vv}{1,\hh}^2
  \quad\textrm{with}\quad
  \snorm{\vv}{1,\hh}^2 =
  \sum_{\P\in\Th}\snorm{\vv}{1,\P}^2
  \label{eq:broken:seminorm}
\end{align}
for every function $\vv$ defined in the broken Sobolev space
\begin{align*}
  \big[\HONE(\Th)\big]^2 =
  \Big\{\vv\in\big[\LTWO(\Omega)\big]^2\,:\,\restrict{\vv}{\P}\in\big[\HONE(\P)\big]^2\quad\forall\P\in\Th\Big\},
\end{align*}
which is the space of square integrable vector-valued functions whose
restriction to every mesh element $\P$ is in $\big[\HONE(\P)\big]^2$.

Space $\PS{\ell}(\P)$ is the span of the finite set of \emph{scaled
monomials of degree up to $\ell$}, that are given by
\begin{align*}
  \calM_{\ell}(\P) =
  \bigg\{\,
  \left( \frac{\xv-\xvP}{\hP} \right)^{\alpha}
  \textrm{~with~}\abs{\alpha}\leq\ell
  \,\bigg\},
\end{align*}
where 
\begin{itemize}
\item $\xvP$ denotes the center of gravity of $\P$ and
  $\hP$ its characteristic length, as, for instance, the edge
  length or the cell diameter;
\item $\alpha=(\alpha_1,\alpha_2)$ is the two-dimensional multi-index
  of nonnegative integers $\alpha_i$ with degree
  $\abs{\alpha}=\alpha_1+\alpha_{2}\leq\ell$ and such that
  $\xv^{\alpha}=\xs_1^{\alpha_1}\xs_{2}^{\alpha_{2}}$ for any
  $\xv\in\REAL^{2}$ and
  $\partial^{\abs{\alpha}}\slash{\partial\xv^{\alpha}}=\partial^{\abs{\alpha}}\slash{\partial\xs_1^{\alpha_1}\partial\xs_2^{\alpha_2}}$.
\end{itemize}
The dimension of $\PS{\ell}(\P)$ equals
$N_{\ell}=(\ell+1)(\ell+2)/2$, the cardinality of the basis set
$\calM_{\ell}(\P)$.

Let $\vs$ and $\vv=(\vs_x,\vs_y)^T$ denote a (smooth enough) scalar
and vector-valued field.
Then, 
\begin{itemize}

\item the elliptic projection $\PinP{\ell}\vs\in\PS{\ell}(\P)$ is
  the solution of the variational problem
  \begin{align}
    \int_{\P}\nabla\big(\vs-\PinP{\ell}\vs\big)\cdot\nabla\qs\dV &= 0 \qquad\forall\qs\in\PS{\ell}(\P),\\[0.5em]
    \int_{\partial\P}\big(\vs-\PinP{\ell}\vs\big)\dS&=0;
    \label{eq:proj:H1:P:def}
  \end{align}

\item the orthogonal projection $\PizP{\ell}\vs\in\PS{\ell}(\P)$ is
  the solution of the variational problem
  \begin{align}
    \int_{\P}\big(\vs-\PizP{\ell}\vs\big)\qs\dV = 0
    \qquad\forall\qs\in\PS{\ell}(\P);
    \label{eq:scalar:proj:L2:P:def}
  \end{align}
  
\item the orthogonal projection of a vector-valued field
  $\vv=(\vs_x,\vs_y)^T$
  is the solution of the variational problem
  \begin{align}
    \int_{\P}\big(\vv-\PizP{\ell}\vv\big)\cdot\qv\dV = 0
    \qquad\forall\qv\in\big[\PS{\ell}(\P)\big]^2,
    \label{eq:vector:proj:L2:P:def}
  \end{align}
  and can be computed componentwisely, i.e.,
  $\PizP{\ell}\vv=(\PizP{\ell}\vs_x,\PizP{\ell}\vs_y)^T\in\big[\PS{\ell}(\P)\big]^2$,
  where $\PizP{\ell}\vs_x$ and $\PizP{\ell}\vs_y$ are the scalar
  orthogonal projections defined above;
  
\item the gradient of vector $\vv$ and its orthogonal projection
  $\PizP{\ell}\nabla\vv\in\big[\PS{\ell}(\P)\big]^{2\times2}$ onto the
  linear space of $2\times2$-sized matrix-valued polynomials of degree
  $\ell$, which are defined componentwisely as follows:
  \begin{align}
    \nabla\vv =
    \left(
    \begin{array}{cc}
      \frac{\partial\vs_x}{\partial\xs} & \quad\frac{\partial\vs_x}{\partial\ys}\\[1.0em]
      \frac{\partial\vs_y}{\partial\xs} & \quad\frac{\partial\vs_y}{\partial\ys}\\
    \end{array}
    \right)
    \qquad\textrm{and}\qquad
    \PizP{\ell}\nabla\vv =
    \left(
    \begin{array}{cc}
      \PizP{\ell}\frac{\partial\vs_x}{\partial\xs} & \quad\PizP{\ell}\frac{\partial\vs_x}{\partial\ys}\\[1.0em]
      \PizP{\ell}\frac{\partial\vs_y}{\partial\xs} & \quad\PizP{\ell}\frac{\partial\vs_y}{\partial\ys}\\
    \end{array}
    \right),
  \end{align}
  and this latter one is the solution of the variational problem:
  \begin{align}
    \int_{\P}\big(\nabla\vv-\PizP{\ell}\nabla\vv\big):\bm\kappa\dV = 0
    \qquad\forall{\bm\kappa}\in\big[\PS{\ell}(\P)\big]^{2\times2}.
    \label{eq:tensor:proj:L2:P:def}
  \end{align}
\end{itemize}

\section{The Stokes problem and the virtual element discretization}
\label{sec:Stokes}

The incompressible Stokes problem for the vector-valued field $\uv$
and the scalar field $\ps$ is governed by the system of equations:
\begin{align}
  - \Delta\uv +\nabla\ps &= \fv \phantom{0}    \quad\textrm{in~}\Omega,\label{eq:stokes:A}\\[0.2em]
  \DIV\uv                &= 0   \phantom{\fv}  \quad\textrm{in~}\Omega,\label{eq:stokes:B}\\[0.2em]
  \uv                    &= 0   \phantom{\fv}  \quad\textrm{on~}\Gamma\label{eq:stokes:C}
\end{align}
on the computational domain $\Omega$ with boundary $\Gamma$.
We refer to $\uv$ and $\ps$ as the \emph{Stokes velocity} and the
\emph{Stokes pressure}.
To ease the exposition, we consider only the case of homogeneous
Dirichlet boundary conditions, see~\eqref{eq:stokes:C}.
However, the extension to nonhomogeneous Dirichlet boundary conditions
is deemed straightforward and the general case is considered in the
section of numerical experiments.

\medskip
\noindent
The variational formulation of \eqref{eq:stokes:A}-\eqref{eq:stokes:C}
reads as: \emph{Find
  $(\uv,\ps)\in\big[\HONEzr(\Omega)\big]^2\times\LTWOzr(\Omega)$} such
that
\begin{align}
  \as(\uv,\vv) + \bs(\vv,\ps) &= (\fv,\vv) \phantom{0}        \qquad\forall\vv\in\big[\HONEzr(\Omega)\big]^2,\label{eq:stokes:var:A}\\[0.5em]
  \bs(\uv,\qs)                &= 0         \phantom{(\fv,\vv)}\qquad\forall\qs\in\LTWOzr(\Omega),            \label{eq:stokes:var:B}
\end{align}
where the bilinear forms
$\as(\cdot,\cdot):\big[\HONE(\Omega)\big]^2\times\big[\HONE(\Omega)\big]^2\to\REAL$ and
$\bs(\cdot,\cdot):\big[\HONE(\Omega)\big]^2\times\LTWO(\Omega)\to\REAL$ are
\begin{align}
  \as(\vv,\wv)&:=\int_{\Omega}\nabla\vv:\nabla\wv\dV
  \phantom{\int_{\Omega}\DIV\vv\,\qs\dV}\hspace{-1.25cm}
  \forall\vv,\wv\in\HONE(\Omega),
  \label{eq:as:def}\\[0.5em]
  \bs(\vv,\qs)&:=-\int_{\Omega}\qs\DIV\vv\dV
  \phantom{\int_{\Omega}\nabla\vv:\nabla\wv\dV}\hspace{-1.25cm}
  \forall\vv\in\HONE(\Omega),\,\qs\in\LTWO(\Omega).
  \label{eq:bs:def}
\end{align}
In the following section, it will be convenient to split these
bilinear forms on the mesh elements by rewriting them in the following
way:
\begin{align}
  \as(\vv,\wv)&=\sum_{\P\in\Th}\asP(\vv,\wv)\quad\textrm{with}\quad
  \asP(\vv,\wv)=\int_{\P}\nabla\vv:\nabla\wv\dV,
  \label{eq:asP:def}\\[0.5em]
  \bs(\vv,\qs)&=\sum_{\P\in\Th}\bsP(\vv,\qs)\quad\textrm{with}\quad
  \bsP(\vv,\qs)=-\int_{\P}\qs\DIV\vv\dV.
  \label{eq:bsP:def}
\end{align}

The bilinear form $\as(\cdot,\cdot)$ is continuous and coercive.
The bilinear form $\bs(\cdot,\cdot)$ is continuous and satisfies the
inf-sup condition:
\begin{align}
  \inf_{\qs\in\LTWOzr(\Omega)\backslash\{0\}}\sup_{\vv\in[\HONEzr(\Omega)\backslash\{0\}]^2}\frac{ \bs(\vv,\qs) }{ \norm{\vv}{1,\Omega}\,\norm{\qs}{0,\Omega} }\geq\beta,
  \label{eq:exact:inf-sup}
\end{align}
for some real, strictly positive constant $\beta$.
These properties imply the existence and uniqueness of the solution
pair $(\uv,\ps)$, and, so, the well-posedness of the variational
formulation~\eqref{eq:stokes:var:A}-\eqref{eq:stokes:var:B}, and the
stability inequality
\begin{align*}
  \norm{\uv}{1,\Omega} + \norm{\ps}{0,\Omega}
  \leq C\norm{\fv}{-1,\Omega},
\end{align*}
for a right-hand side forcing term $\fv\in\HS{-1}(\Omega)$, and a
constant $C$ that depends only on $\Omega$,
cf.~\cite{Boffi-Brezzi-Fortin:2013,Girault-Raviart:1986,Girault-Raviart:1979}.

\medskip
Let $k\geq1$ be a given integer number.
Our virtual element discretizations have the general abstract form:
\emph{Find
  $(\uvh,\psh)\in\Vvhk\times\Qshkk$}
\begin{align}
  \ash(\uvh,\vvh) + \bsh(\vvh,\psh) &= \bil{\fvh}{\vvh} \phantom{0}               \qquad\forall\vvh\in\Vvhk, \label{eq:stokes:vem:A}\\[0.5em]
  \bsh(\uvh,\qsh)                   &= 0                \phantom{\bil{\fvh}{\vvh}}\qquad\forall\qsh\in\Qshkk.\label{eq:stokes:vem:B}
\end{align}
Here, $\Vvhk$ is a finite-dimensional conforming subspace of
$\big[\HONEzr(\Omega)\big]^2$ and $\Qshkk$ a finite-dimensional
discontinuous subspace of $\LTWOzr(\Omega)$.
We use the integer $k$, which is a polynomial degree, to denote the
accuracy of the method.
The vector field $\uvh$ and the scalar field $\psh$ are the virtual
element approximation of $\uv$ and $\ps$, respectively.
The bilinear forms $\ash(\cdot,\cdot):\Vvhk\times\Vvhk\to\REAL$ and
$\bsh(\cdot,\cdot):\Vvhk\times\Qshkk\to\REAL$ are the virtual element
approximations to the corresponding bilinear forms $\as(\cdot,\cdot)$
and $\bs(\cdot,\cdot)$.
The linear functional $\bil{\fvh}{\cdot}$ is the virtual element
approximation of the right-hand side of~\eqref{eq:stokes:var:A}.
The definition of all these mathematical objects is discussed in the
next section, where we present, analyze and investigate numerically
two new virtual element formulations that are suitable to polygonal
meshes.

\section{Virtual element approximations of the Stokes problem}
\label{sec:VEM}

We present two different virtual element approximations of the 2-D
Stokes problem in variational form.
For both formulations, the Stokes pressure is approximated by a
piecewise polynomial function that belongs to the space
\begin{align}
  \Qshkk&:=\Big\{\qsh\in\LTWOzr(\Omega)\,:\,\restrict{\qsh}{\P}\in\PS{k-1}(\P)\quad\forall\P\in\Th\Big\}=\PS{k-1}(\Th)\cap\LTWOzr(\Omega),\label{eq:SV:scalar:space:def}
\end{align}
and the degrees of freedom are the polynomial moments in every element
against the polynomials of degree $k-1$.
The Stokes velocity field is approximated in the finite-dimensional
subspace of $\big[\HONEzr(\Omega)\big]^2$ given by
\begin{align}
  \Vvhk:=\Big\{\vvh\in\big[\HONEzr(\Omega)\big]^2\,:\,\restrict{\vvh}{\P}\in\Vvhk(\P)\quad\forall\P\in\Th\Big\}.
  \label{eq:VEM:global:space}
\end{align}
This functional space is defined by ``gluing together'' in a
conforming way the local virtual element spaces $\Vvhk(\P)$, defined
on the mesh elements $\P\in\Th$.
In particular, we denote the elemental space of the first formulation
by $\VvhkFO(\P)$ (\emph{formulation $\FO$}) and that of the second
formulation by $\VvhkFT(\P)$ (\emph{formulation $\FT$}), and we will
use the generic symbols $\Vvhk(\P)$ (local space) and $\Vvhk$ (global
space) when we discuss properties that hold regardless of the specific
space definition.
For both formulation, we also consider the modified definition of the
elemental spaces according to the so called \emph{enhancement
strategy}~\cite{Ahmad-Alsaedi-Brezzi-Marini-Russo:2013}.
This strategy allows us to compute the $\LTWO$-orthogonal projection
onto the local subspace of vector polynomials of degree $k$, i.e., the
subspace $\big[\PS{k}(\P)\big]^2\subset\Vvhk(\P)$.
This orthogonal projection is required in the formulation of the
right-hand side of Eq.~\eqref{eq:stokes:vem:A}.

\medskip
In the rest of this section, we first review the general construction
of the virtual element approximation.
Then, for each formulation
\begin{description}
\item[-] $(i)$ we explicitly define the local virtual element space
  and its degrees of freedom and discuss their unisolvence;
\item[-] $(ii)$ we prove that the following polynomial projections of
  $\nabla\vvh$, $\DIV\vvh$ and $\vvh$ are computable for every virtual
  element vector-valued field $\vvh$ using only the degrees of freedom
  of $\vvh$:
  $\PizP{k-1}\nabla\vvh\in\big[\PS{k-1}(\P)\big]^{2\times2}$;
  $\PizP{k-1}\DIV\vvh\in\PS{k-1}(\P)$;
  $\PinP{k}\vvh\in\big[\PS{k}(\P)\big]^2$;
  $\PizP{\kb}\vvh\in\big[\PS{\kb}(\P)\big]^2$ where $\kb=max(0,k-2)$
  for the regular space definition or $\kb=\ks$ for the enhanced space
  definition;
  (we recall that the formal definitions of these operators are given
  in~\eqref{eq:proj:H1:P:def}-\eqref{eq:tensor:proj:L2:P:def}).
\end{description}

\PGRAPH{Construction of the virtual element bilinear form $\ash$}
Using these projection operators, we define the virtual element
bilinear form $\ash(\cdot,\cdot)$ as the sum of local bilinear forms
$\asPh(\cdot,\cdot):\Vvhk(\P)\times\Vvhk(\P)\to\REAL$ as follows:
\begin{align}
  &\ash(\vvh,\wvh) = \sum_{\P\in\Th}\asPh(\vvh,\wvh)
  \label{eq:ash:def}
  \intertext{where~}
  &\asPh(\vvh,\wvh)
  = \int_{\P}\PizP{k-1}\nabla\vvh:\PizP{k-1}\nabla\wvh\dV
  + \SPh\big( (1-\Pi^{\P}_{k})\vvh, (1-\Pi^{\P}_{k})\wvh \big).
  \label{eq:asPh:def}
\end{align}
Here, $\SPh(\cdot,\cdot):\Vvhk(\P)\times\Vvhk(\P)\to\REAL$ is the
local bilinear form providing the stabilization term, and
$\Pi^{\P}_{k}$ denote either the $\LTWO$-orthogonal projection
$\PizP{k}$ (when computable) or the elliptic projection $\PinP{k}$.
The term $\SPh(\cdot,\cdot)$ can be any symmetric, positive definite
bilinear form for which there exist two real, positive constant
$\sigma_*$ and $\sigma^*$ independent of $\hh$ (and $\P$) such that
\begin{align*}
  \sigma_*\asP(\vvh,\vvh)\leq\SPh(\vvh,\vvh)\leq\sigma^*\asP(\vvh,\vvh)
  \qquad
  \forall\vvh\in\Vvhk(\P)\cap\textrm{ker}(\Pi^{\P}_{k}),
\end{align*}
where $\asP(\cdot,\cdot)$ is defined in~\eqref{eq:asP:def}.
Several possible stabilizations have been proposed over the last few
years and are available from the technical literature,
cf.~\cite{Mascotto:2018}.
The local bilinear form $\asPh(\cdot,\cdot)$ has two fundamental
properties that are used in the analysis:
\begin{itemize}
\item \textbf{Polynomial consistency}: for every vector field
  $\vvh\in\Vvhk(\P)$ and vector polynomial field
  $\qvh\in\big[\PS{k}(\P)\big]^2$ it holds:
  \begin{align}
    \asPh(\vvh,\qvh) = \asP(\vvh,\qvh);
    \label{eq:consistency}
  \end{align}

  \medskip
\item \textbf{Stability}: there exist two real, positive constants
  $\alpha_*$ and $\alpha^*$ independent of $\hh$ such that
  \begin{align}
    \alpha_*\asP(\vvh,\vvh)\leq\asPh(\vvh,\vvh)\leq\alpha^*\asP(\vvh,\vvh)
    \qquad\forall\vvh\in\Vvhk(\P).
    \label{eq:ash:stability}
  \end{align}
  Both constants $\alpha_*$ and $\alpha^*$ may depend on the
  polynomial degree $k$ and the mesh regularity constant $\rho$.
\end{itemize}
By adding all the elemental contributions, we find that
$\ash(\cdot,\cdot)$ is a coercive bilinear form on $\Vvhk\times\Vvhk$:
\begin{align}
  \ash(\vvh,\vvh) \geq \alpha_*\snorm{\vvh}{1,\Omega}^2.
  \label{eq:ash:coercivity}
\end{align}
A second straighforward consequence of~\eqref{eq:ash:stability} and
the symmetry of $\asPh(\cdot,\cdot)$ is that this bilinear form is an
inner product on $\Vvhk(\P)\setminus\REAL$.
Using the Cauchy-Schwarz inequality, it holds that:
\begin{align}
  \asPh(\vvh,\wvh)
  \leq \big[\asPh(\vvh,\vvh)\big]^{\frac{1}{2}}\,\big[\asPh(\wvh,\wvh)\big]^{\frac{1}{2}}
  \leq \alpha^*\,\big[\asP(\vvh,\vvh)\big]^{\frac{1}{2}}\,\big[\asP(\wvh,\wvh)\big]^{\frac{1}{2}}
  =    \alpha^*\,\snorm{\vvh}{1,\P}\,\snorm{\wvh}{1,\P},
  \label{eq:ash:local:continuity}
\end{align}
which implies that the local bilinear form $\asPh(\cdot,\cdot)$ is continuous on
$\Vvhk(\P)\times\Vvhk(\P)$.
The global continuity of $\ash(\cdot,\cdot)$ follows on summing all
the local terms and using again the Cauchy-Schwarz inequality:
\begin{align}
  \ash(\vvh,\wvh)
  &=   \sum_{\P\in\Th}\asPh(\vvh,\wvh)
  \leq \alpha^*\,\sum_{\P\in\Th}\,\snorm{\vvh}{1,\P}\,\snorm{\wvh}{1,\P}
  \leq \alpha^*\,\Bigg(\sum_{\P\in\Th}\,\snorm{\vvh}{1,\P}^2\Bigg)^{\frac12}\,\Bigg(\sum_{\P\in\Th}\,\snorm{\wvh}{1,\P}^2\Bigg)^{\frac12}
  \nonumber\\[0.5em]
  &=    \alpha^*\,\snorm{\vvh}{1,\Omega}\,\snorm{\wvh}{1,\Omega}.
  \label{eq:ash:global:continuity}
\end{align}

\PGRAPH{Construction of the virtual element bilinear forms $\bsh$}
Similarly, we define the virtual element bilinear form
$\bsh(\cdot,\cdot)$ as the sum of local bilinear forms
$\bsPh(\cdot,\cdot):\Vvhk(\P)\times\PS{k-1}(\P)\to\REAL$ as follows:
\begin{align}
  \bsh(\vvh,\qsh) = \sum_{\P\in\Th}\bsPh(\vvh,\qsh)
  \qquad\textrm{where}\qquad
  \bsPh(\vvh,\qsh)
  = \int_{\P}\qsh\PizP{k-1}\DIV\vvh\dV.
  \label{eq:bsh:def}
\end{align}
From the definition of the orthogonal projection operator
$\PizP{k-1}$, it immediately follows that
\begin{align}
  \bsPh(\vvh,\qsh)=\bsP(\vvh,\qsh)
  \qquad\forall\vvh\in\Vvhk(\P),\,\qsh\in\PS{k-1}(\P).
  \label{eq:bsPh=bsP}
\end{align}
If we add this relation over all the mesh elements, we find that
\begin{align}
  \bsh(\vvh,\qsh) = \bs(\vvh,\qsh)
  \qquad\forall\vvh\in\Vvhk,\,\qsh\in\PS{k-1}(\Th),
  \label{eq:bsh=bs}
\end{align}
which will be used in the analysis of the next section.

\medskip
\begin{remark}
  Since $\PizP{k-1}(\DIV\uvh)$ for all elements $\P$ is a polynomial
  of degree $k-1$, equation~\eqref{eq:stokes:vem:B} is equivalent to
  require that $\PizP{k-1}(\DIV\uvh)=0$ in $\P$.
  This condition is the discrete analog in $\PS{k-1}(\P)$ of the
  incompressibility condition $\DIV\uv=0$.
\end{remark}

\PGRAPH{Construction of the virtual element right-hand side}
In every polygonal element $\P$, we approximate the right-hand side
vector $\fv$ with its polynomial projection $\fvh:=\PizP{\kb}\fv$
onto the local polynomial space $\PS{\kb}(\P)$.
We consider two possible choices of $\kb$ given the integer $k\geq1$:
\begin{itemize}
\item[$\bullet$] $\kb=\max(k-2,0)$: this is the setting proposed in the original
  paper~\cite{BeiraodaVeiga-Brezzi-Cangiani-Manzini-Marini-Russo:2013};
\item[$\bullet$] $\kb=k$: this is the setting proposed
  in Ref.~\cite{Ahmad-Alsaedi-Brezzi-Marini-Russo:2013}, which requires the
  enhanced definition of the virtual element space.
  We discuss the enhanced definition of the virtual element space of
  both formulations in the next sections.
\end{itemize}

\medskip
Finally, the right hand-side of equation~\eqref{eq:stokes:vem:A} is
given by
\begin{align}
  \bil{\fvh}{\vvh}
  = \sum_{\P\in\Th} \int_{\P}\PizP{\kb}\fvh\cdot\vvh\dV
  = \sum_{\P\in\Th} \int_{\P}\fvh\cdot\PizP{\kb}\vvh\dV,
  \label{eq:fvh:def}
\end{align}
where the second equality follows on applying the definition of the
orthogonal projector $\PizP{\kb}$.

We recall the following results pertaining these two possible
approximations of the right-hand side, which follows on noting that
$\big(1-\PizP{\kb}\big)$ is orthogonal to $\PizP{0}$ in the
$\LTWO$-inner product.
Assuming $\fv\in\big[\HS{s}(\Omega)\big]^{\dims}$ with
$1\leq\ss\leq\kb$, we find that
\begin{align}
  \ABS{\bil{\fvh}{\vvh}-\scal{\fv}{\vvh}}
  &=   \ABS{ \sum_{\P\in\Th}\int_{\P}\big(\PizP{\kb}\fv-\fv\big)\vvh\dV }
  \leq \sum_{\P\in\Th}\ABS{ \int_{\P}\big(\PizP{\kb}\fv-\fv\big)\big(\vvh-\PizP{0}\vvh\big)\dV }
  \nonumber\\[0.5em]
  &\leq \sum_{\P\in\Th} \norm{\PizP{\kb}\fv-\fv}{0,\P}\,\norm{\vvh-\PizP{0}\vvh}{0,\P}
  \leq \Cs\hh^{s+1}\norm{\fv}{s,\P}\,\snorm{\vvh}{1,\P}.
  \label{eq:fv:bound:0}
\end{align}
For $\kb=0$ and assuming $\fv\in\big[\LTWO(\Omega)\big]^{\dims}$, we
find that
\begin{align}
  \ABS{\bil{\fvh}{\vvh}-\scal{\fv}{\vvh}}
  &=   \ABS{ \sum_{\P\in\Th}\int_{\P}\big(\PizP{0}\fv-\fv\big)\vvh\dV }
  \leq \sum_{\P\in\Th}\ABS{ \int_{\P}\big(\PizP{0}\fv-\fv\big)\big(\vvh-\PizP{0}\vvh\big)\dV }
  \nonumber\\[0.5em]
  &\leq \sum_{\P\in\Th} \norm{\PizP{0}\fv-\fv}{0,\P}\,\norm{\vvh-\PizP{0}\vvh}{0,\P}
  \leq \Cs\hh\norm{\fv}{0,\P}\,\snorm{\vvh}{1,\P}.
  \label{eq:fv:bound:1}
\end{align}

\subsection{Formulation $\FO$}
\label{subsec:first:formulation}

We set the virtual element space for the velocity vector-valued fields
of the first formulation as
\begin{align*}
  \Vvhkp(\P)=\Big[\Vshkp(\P)\Big]^2,
\end{align*}
where the corresponding scalar virtual element space is given by
\begin{align}
  \Vshkp(\P):=\Big\{\vsh\in\HONE(\P)\,:\,
  \restrict{\vsh}{\partial\P}\in\CS{0}(\partial\P),\,
  \restrict{\vsh}{\E}\in\PS{k+1}(\E)\,\forall\E\in\partial\P,\,
  \Delta\vsh\in\PS{k-2}(\P)
  \Big\}.
  \label{eq:FO:regular-space:def}
\end{align}
With a small abuse of notation, we denote the enhanced version of the
local space with the same symbol $\Vshkp(\P)$, and we consider the
following definition:
\begin{align}
  \Vshkp(\P):=\Big\{\vsh\in\HONE(\P)\,:\,
  &\restrict{\vsh}{\partial\P}\in\CS{0}(\partial\P),\,
  \restrict{\vsh}{\E}\in\PS{k+1}(\E)\,\forall\E\in\partial\P,\,
  \Delta\vsh\in\PS{k}(\P),\nonumber\\[0.25em]
  &\int_{\P}\big(\vsh-\PinP{k}\vsh\big)\qsh\dV=0\quad\forall\qsh\in\PS{k}(\P)\backslash{\PS{k-2}(\P)}
  \Big\},
  \label{eq:FO:enhanced-space:def}
\end{align}
where $\PS{k}(\P)\backslash{\PS{k-2}(\P)}$ is the space of polynomials
of degree exactly equal to $k$ or $k-1$.
This definition uses the elliptic projection operator $\PinP{k}$,
which is computable from the degrees of freedom defined below,
cf. Lemma~\ref{lemma:3}.

\medskip
\begin{remark}
  The virtual element space~\eqref{eq:FO:regular-space:def} and its
  modified version~\eqref{eq:FO:enhanced-space:def} differ from the
  spaces respectively defined in
  References~\cite{BeiraodaVeiga-Brezzi-Cangiani-Manzini-Marini-Russo:2013}
  and~\cite{Ahmad-Alsaedi-Brezzi-Marini-Russo:2013} because all the
  edge traces of a virtual element function are polynomials of degree
  $k+1$ instead of $k$.
  This definition is a special case of the generalized local virtual
  element space that is considered
  in~\cite[Section~3]{BeiraodaVeiga-Vacca:2020-arXiv} for the
  discretization of the Poisson equation.
  In fact, the local scalar space~\eqref{eq:FO:regular-space:def} can
  be obtained by setting $k_{\partial}=k+1$
  in~\cite[Eq.~(7)]{BeiraodaVeiga-Vacca:2020-arXiv} (with the same
  meaning for the parameter $k$).
\end{remark}

\medskip
\begin{remark}
  Assuming that the trace on the edges of the elemental boundary is a
  polynomial of degree $k+1$ instead of $k$ does not change the
  convergence rate of the method and implies that an additional degree
  of freedom is needed for each velocity components on every edge,
  thus increasing the complexity and the computational costs.
  However, it makes the proof of the inf-sup condition almost
  straightforward, which is crucial to prove the well-posedness and
  convergence of the method.
  So, this formulation allows us to build a stable numerical
  approximation to the Stokes problem that holds on any kind of
  polygonal meshes, including triangular and square meshes, for all
  orders of accuracy $k\geq1$.
\end{remark}


\begin{figure}[!t]
  \centering
  \begin{tabular}{ccc}
    \includegraphics[width=0.28\textwidth]{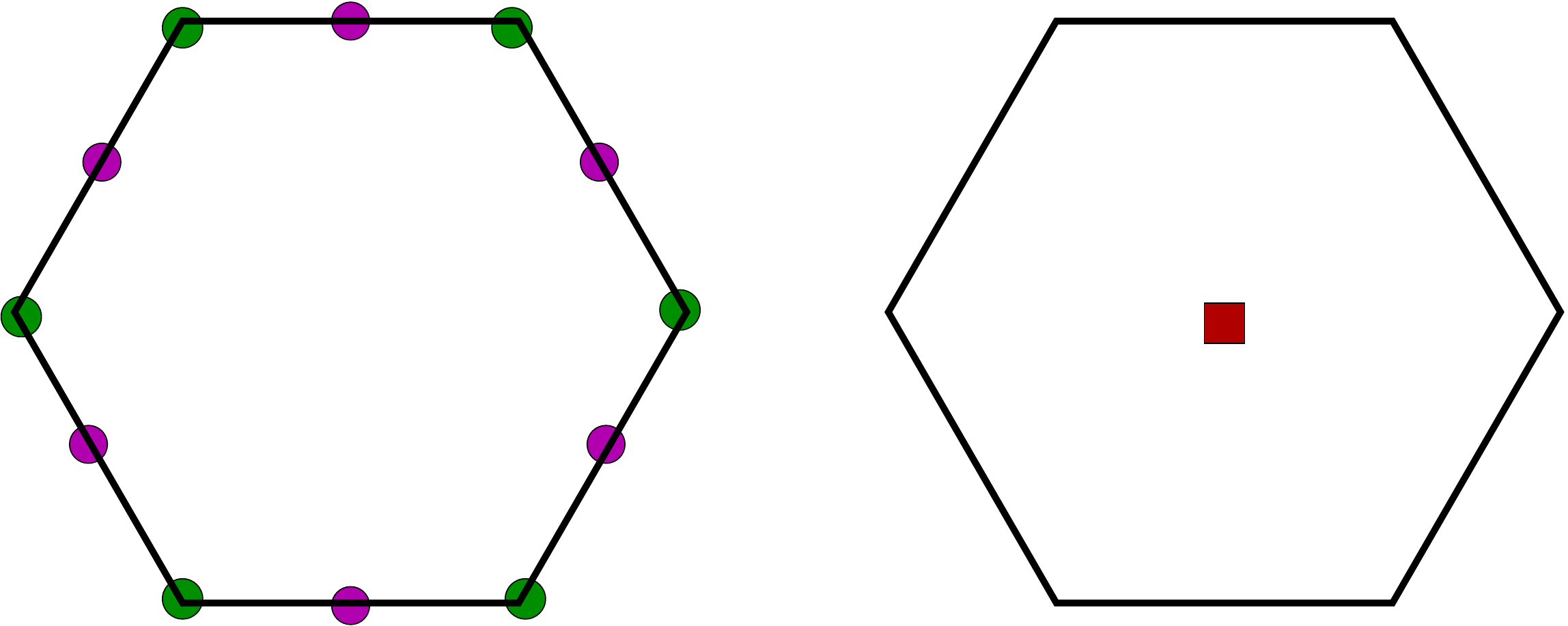} &\qquad   
    \includegraphics[width=0.28\textwidth]{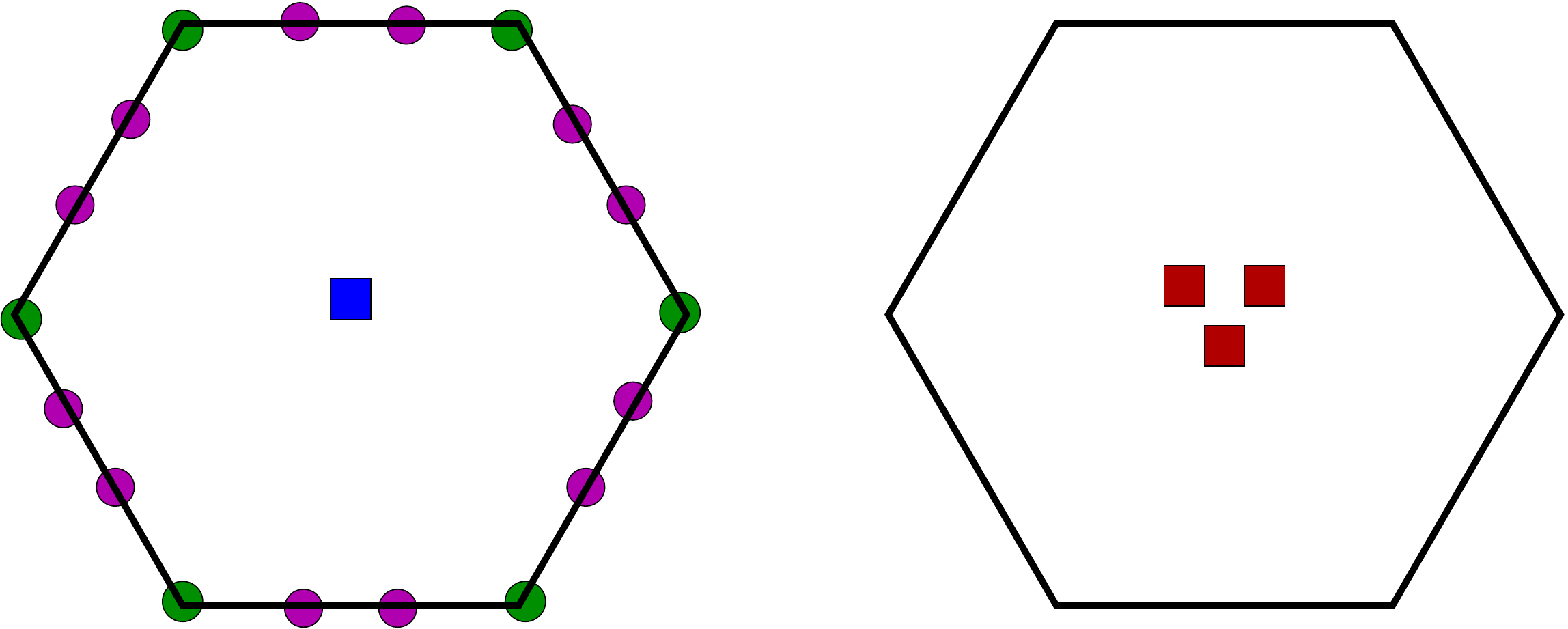} &\qquad   
    \includegraphics[width=0.28\textwidth]{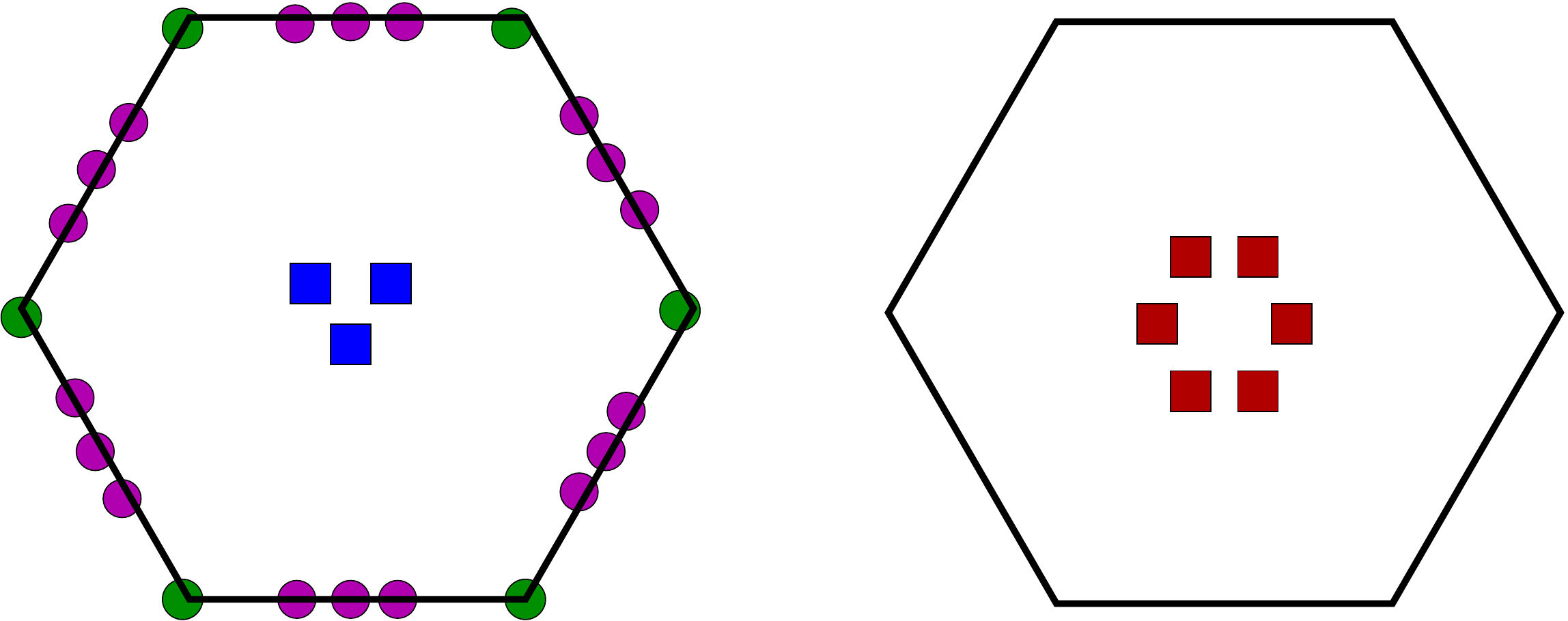} \\[0.5em] 
    \hspace{2mm}$\mathbf{k=1}$ & \hspace{9mm}$\mathbf{k=2}$ & \hspace{9mm}$\mathbf{k=3}$ 
  \end{tabular}
  \caption{First virtual element formulation: degrees of freedom of
    each component of the virtual element vector-valued fields (left)
    and the scalar polynomial fields (right) of an hexagonal element
    for the accuracy degrees $k=1,2,3$.
    Nodal values at the polygonal vertices and edge polynomial moments
    are marked by a circular bullet; cell polynomial moments are
    marked by a square bullet.}
  \medskip
  \label{fig:dofs:FO}
\end{figure}

\medskip
\noindent
The degrees of freedom of this formulation for the spaces defined
in~\eqref{eq:FO:regular-space:def}
and~\eqref{eq:FO:enhanced-space:def} are given by:

\medskip
\begin{description}
\item[-]\DOFS{F1}{a} for $k\geq1$, the vertex values $\vsh(\xV)$, $\V\in\partial\P$;
\item[-]\DOFS{F1}{b} for $k\geq1$, the polynomial edge moments of
  $\vsh$
  \begin{align}
    \frac{1}{\mE}\int_{\E}\vsh(\ss)\qsh(\ss)\ds\quad\forall\qsh\in\PS{k-1}(\E)
  \end{align}
  for every edge $\E\in\partial\P$;
\item[-]\DOFS{F1}{c} for $k\geq2$, the polynomial cell moments of
  $\vsh$
  \begin{align}
    \frac{1}{\mP}\int_{\P}\vsh(\xv)\qsh(\xv)\dV\quad\forall\qsh\in\PS{k-2}(\P).
  \end{align}
\end{description}
Figure~\ref{fig:dofs:FO} shows the degrees of freedom for each
component of the velocity vector and the pressure for $k=1,2,3$ on an
hexagonal element.

\medskip
\noindent
\begin{lemma}[Unisolvence of the degrees of freedom]
  \label{lemma:F1:unisolvence}
  The degrees of freedom \DOFS{F1}{a}, \DOFS{F1}{b}, and \DOFS{F1}{c}
  are unisolvent in the space $\Vshkp(\P)$ for both the definitions
  given in~\eqref{eq:FO:regular-space:def}
  and~\eqref{eq:FO:enhanced-space:def}.
\end{lemma}
\begin{proof}
  The proof of the unisolvence of the degrees of freedom
  \DOFS{F1}{a}-\DOFS{F1}{c} for $\Vshkp(\P)$ follows by adapting the
  arguments used in
  \cite[Proposition~1]{BeiraodaVeiga-Brezzi-Cangiani-Manzini-Marini-Russo:2013}
  for the space defined in~\eqref{eq:FO:regular-space:def}
  and~\cite[Proposition~2]{Ahmad-Alsaedi-Brezzi-Marini-Russo:2013} for
  the space defined in~\eqref{eq:FO:enhanced-space:def}.
  We briefly sketch the proof of the unisolvence for the space defined
  in~\eqref{eq:FO:regular-space:def}.
  For every virtual element function in $\Vshkp(\P)$, we consider the
  integration by parts:
  \begin{align}
    \int_{\P}\ABS{\nabla\vsh}^2\dV
    = -\int_{\P}\vsh\cdot\Delta\vsh + \sum_{\E\in\partial\P}\int_{\E}\vsh\,\norE\cdot\nabla\vsh\dS
    = \TERM{I}{} + \TERM{II}{}.
  \end{align}
  Now, assume that the degrees of freedom \DOFS{F1}{a}, \DOFS{F1}{b},
  and \DOFS{F1}{c} are all zero.
  Then,
  \begin{description}
  \item[-] for $k=1$, it holds that $\Delta\vsh=0$; for $k\geq2$, it
    holds that $\Delta\vsh$ is a polynomial of degree $k-2$ and
    $\TERM{I}{}$ is a degree of freedom, hence it is zero by hypothesis;
  \item[-] the trace of $\vsh$ along each edge $\E\in\partial\P$ is a
    polynomial of degree $k+1$ that can be recovered by the
    interpolation of the degrees of freedom \DOFS{F1}{a} and
    \DOFS{F1}{b}.
    Since these degrees of freedom are zero by hypothesis, their trace
    interpolation is zero.
  \end{description}
  Consequently, $\nabla\vsh=0$, which implies that $\vsh$ is constant
  on $\P$, and this constant is zero since it coincides with the value
  of all its degrees of freedom, which we assume to be zero.
  The proof of the unisolvence for the space defined
  in~\eqref{eq:FO:regular-space:def} is completed by noting that the
  number of the degrees of freedom equals the dimension of space
  $\Vshkp(\P)$.
  Similar modifications to the argument
  of~\cite[Proposition~2]{Ahmad-Alsaedi-Brezzi-Marini-Russo:2013} make
  it possible to prove the unisolvence for the enhanced virtual
  element space defined in~\eqref{eq:FO:enhanced-space:def}.
\end{proof}

\medskip
\noindent
\begin{lemma}
  \label{lemma:2}
  Let $\P$ be an element of mesh $\Th$.
  For every virtual element function $\vsh\in\Vshkp(\P)$, the
  polynomial projection $\PizP{k-1}\nabla\vsh$ is computable using the
  degrees of freedom \DOFS{F1}{a}, \DOFS{F1}{b}, and \DOFS{F1}{c} of
  $\vsh$.
\end{lemma}
\begin{proof}
  To prove that $\PizP{k-1}\big(\nabla\vsh\big)$ is computable, we
  explicitly prove that
  $\PizP{k-1}\big(\partial\vsh\slash{\partial\xs}\big)$ is computable.
  Then, the same argument can be applied to prove that
  $\PizP{k-1}(\partial\vsh\slash{\partial\ys})$ is also computable.
  To this end, we start from the definition of the orthogonal
  projection and integrate by parts:
  \begin{align}
    \int_{\P}\qsh\PizP{k-1}\frac{\partial\vsh}{\partial\xs}\dV
    = \int_{\P}\qsh\frac{\partial\vsh}{\partial\xs}\dV
    = -\int_{\P}\vsh\frac{\partial\qsh}{\partial\xs}\dV
    + \sum_{\E\in\partial\P}\ns_x\int_{\E}\vsh\qsh\dS
    = \TERM{I}{} + \TERM{II}{},
  \end{align}
  which holds for every $\qsh\in\PS{k-1}(\P)$.
  Term $\TERM{I}{}$ is computable since
  $\partial\qsh\slash{\partial\xs}\in\PS{k-2}(\P)$ and this integral is
  determined by the degrees of freedom of $\vsh$ in \DOFS{F1}{c}.
  Term $\TERM{II}{}$ is computable since the polynomial $\qsh$ is known
  and $\restrict{\vsh}{\E}\in\PS{k+1}(\E)$ can be interpolated from
  the degrees of freedom of $\vsh$ given by \DOFS{F1}{a} and
  \DOFS{F1}{b} on every edge $\E\in\partial\P$.
\end{proof}

\medskip
\noindent
\begin{remark}
  \label{remark:2}
  For all scalar virtual element functions $\vsh\in\Vshkp(\P)$, the
  polynomial projections
  $\PizP{k-1}\big(\partial\vsh\slash{\partial\xs}\big)$ and
  $\PizP{k-1}\big(\partial\vsh\slash{\partial\ys}\big)$ forming
  $\PizP{k-1}\nabla\vsh$ are computable by using the degrees of
  freedom of $\vsh$.
  Consequently, the polynomial projections
  $\PizP{k-1}\nabla\vvh\in\big[\PS{k-1}(\P)\big]^{2\times2}$ and
  $\PizP{k-1}\DIV\vvh\in\PS{k-1}(\P)$ are computable for all virtual
  vector-valued fields $\vvh\in\big[\Vshkp(\P)\big]^2$.
\end{remark}

\medskip
\noindent
\begin{lemma}
  \label{lemma:3}
  Let $\P$ be an element of mesh $\Th$.
  For all virtual element functions $\vsh\in\Vshkp(\P)$, the
  polynomial projection $\PinP{k}\vsh\in\PS{k}(\P)$ is computable
  from the degrees of freedom of $\vsh$.
\end{lemma}
\begin{proof}
  The same argument of Lemma~\ref{lemma:2} is used here.
  We start from the definition of the elliptic projection and we
  integrate by parts:
  \begin{align}
    \int_{\P}\nabla\PinP{k}\vsh\cdot\nabla\qsh\dV
    = \int_{\P}\nabla\vsh\cdot\nabla\qsh\dV
    = -\int_{\P}\vsh\Delta\qsh\dV + \sum_{\E\in\partial\P}\int_{\E}\vsh\norE\cdot\nabla\vsh\dS
    = \TERM{I}{} + \TERM{II}{}.
    \label{eq:lemma3:aux:00}
  \end{align}
  Since in~\eqref{eq:lemma3:aux:00} we take $\qsh\in\PS{k}(\P)$ and
  $\Delta\qsh\in\PS{k-2}(\P)$, term $\TERM{I}{}$ is computable using
  the degrees of freedom \DOFS{F1}{c} of $\vsh$.
  Similarly, since $\restrict{\vsh}{\E}\in\PS{k+1}(\E)$ is computable
  from an interpolation of the degrees of freedom \DOFS{F1}{a} and
  \DOFS{F1}{b}, term $\TERM{II}{}$ is computable.
\end{proof}

\medskip
\noindent
\begin{remark}
  $\PinP{k}\vvh$ is computable componentwisely for every vector-valued
  virtual element field $\vvh\in\Vvhkp(\P)$ and is used in the
  stabilization term of $\asPh(\cdot,\cdot)$, cf.~\eqref{eq:asPh:def}.
\end{remark}

\subsection{Formulation $\FT$}
\label{subsec:second:formulation}

We denote the tangential and normal components of $\vvh$ along the
edge $\E\in\partial\P$ by $\restrict{\vvh}{\E}\cdot\tngE$ and
$\restrict{\vvh}{\E}\cdot\norE$, where $\tngE$ and $\norE$ are the
unit tangential and orthogonal vector of $\E$.
The virtual element space of the second formulation is defined as:
\begin{align}
  \Vvhkpp(\P):=\Big\{
  \vvh\in\big[\HONE(\P)\big]^2:\,
  \restrict{\vvh}{\partial\P}\in\big[\CS{0}(\partial\P)\big]^2,
  \restrict{\vvh}{\E}\cdot\tngE\in\PS{k}(\E),
  \restrict{\vvh}{\E}\cdot\norE\in\PS{k+1}(\E),
  \Delta\vvh\in\big[\PS{k-2}(\P)\big]^2
  \Big\}.
  \label{eq:FT:regular-space:def}
\end{align}
With a small abuse of notation we denote the ``enhanced'' version of
this space with the same symbol ``$\Vvhkpp$'':
\begin{align}
  \Vvhkpp(\P):=\Big\{
  \vvh\in\big[\HONE(\P)\big]^2:\,
  &
  \restrict{\vvh}{\partial\P}\in\big[\CS{0}(\partial\P)\big]^2,
  \restrict{\vvh}{\E}\cdot\tngE\in\PS{k}(\E),
  \restrict{\vvh}{\E}\cdot\norE\in\PS{k+1}(\E),
  \Delta\vvh\in\big[\PS{k}(\P)\big]^2,\,
  \nonumber\\[0.25em]
  &
  \int_{\P}\big(\vvh-\PinP{k}\vvh\big)\cdot\qvh\dV=0\quad\forall\qvh\in\big[\PS{k}(\P)\backslash{\PS{k-2}(\P)}\big]^2
  \Big\},
  \label{eq:FT:enhanced-space:def}
\end{align}
where $\PS{k}(\P)\backslash{\PS{k-2}(\P)}$ is the space of polynomials
of degree exactly equal to $k$ and $k-1$.
This definition uses the elliptic projection operator $\PinP{k}$,
which is computable from the degrees of freedom defined below,
cf.~Lemma~\ref{lemma:6}.

Note that the normal component of $\vvh$ is a polynomial of degree
$k+1$ while the tangential component is a polynomial of degree $k$.
These conditions are reflected by the following degrees of freedom,
which are the same for the virtual element functions defined in
both~\eqref{eq:FT:regular-space:def}
and~\eqref{eq:FT:enhanced-space:def}:


\begin{figure}[!t]
  \centering
  \begin{tabular}{ccc}
    \includegraphics[width=0.28\textwidth]{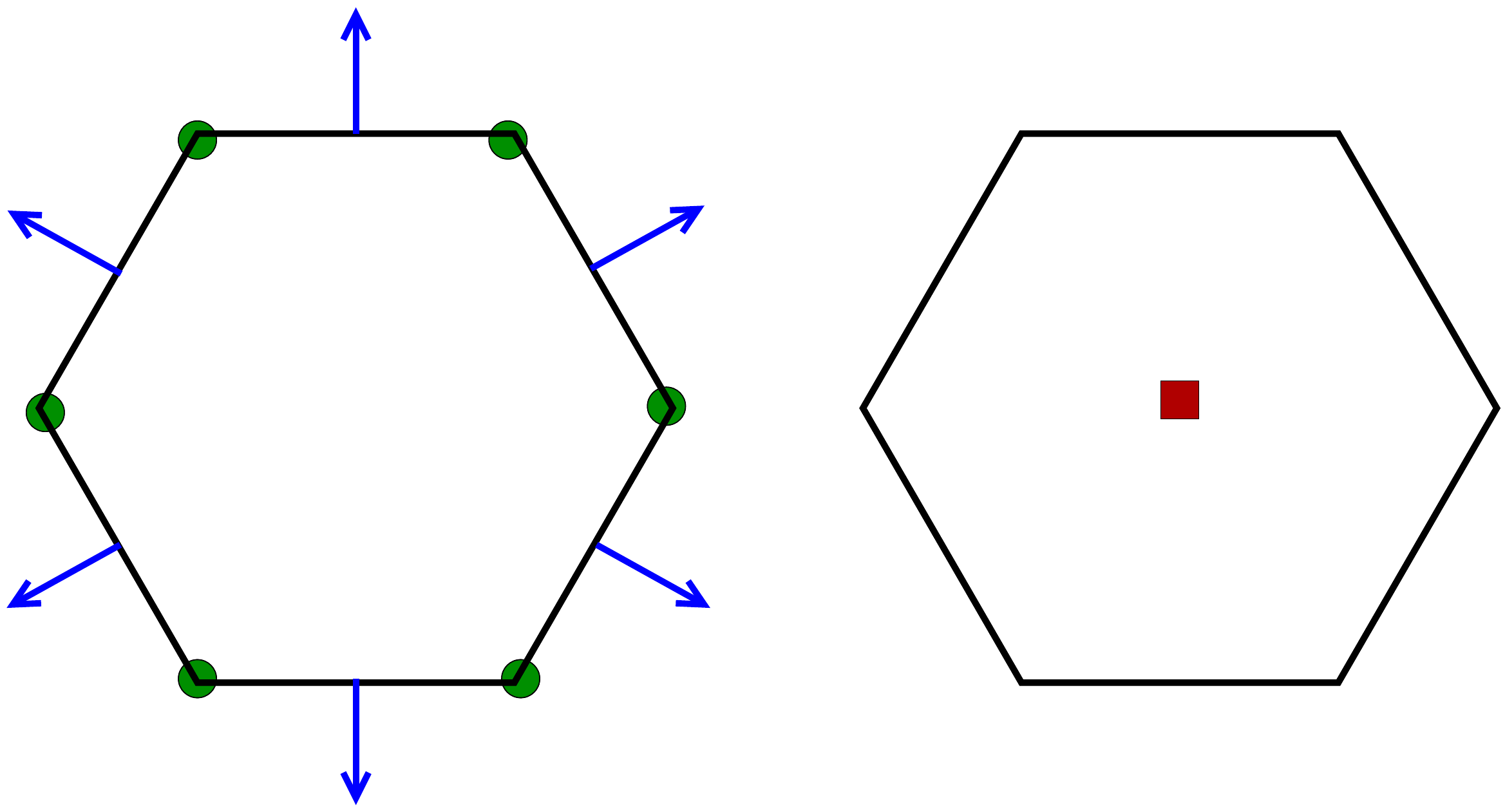} &\qquad   
    \includegraphics[width=0.28\textwidth]{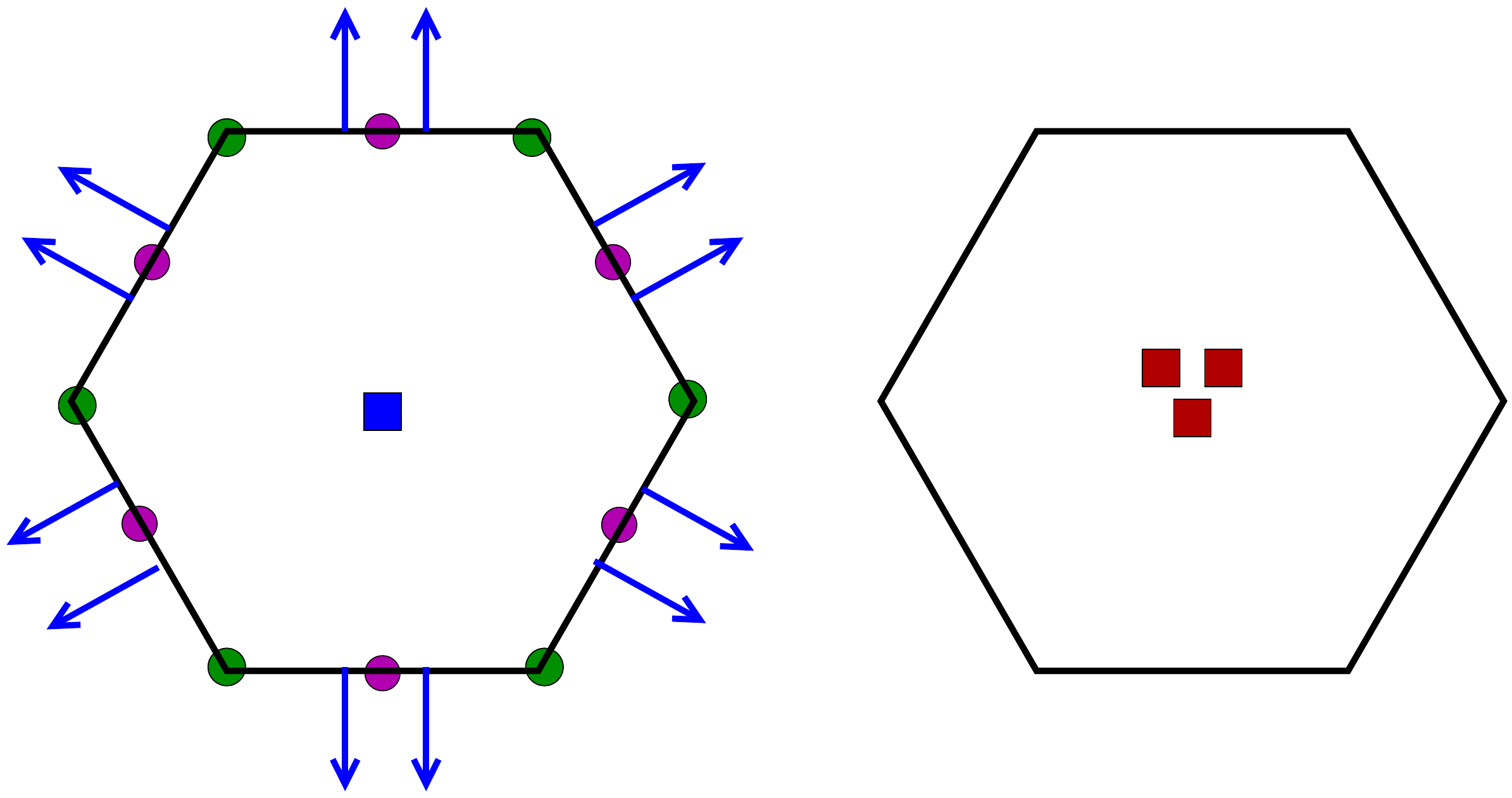} &\qquad   
    \includegraphics[width=0.28\textwidth]{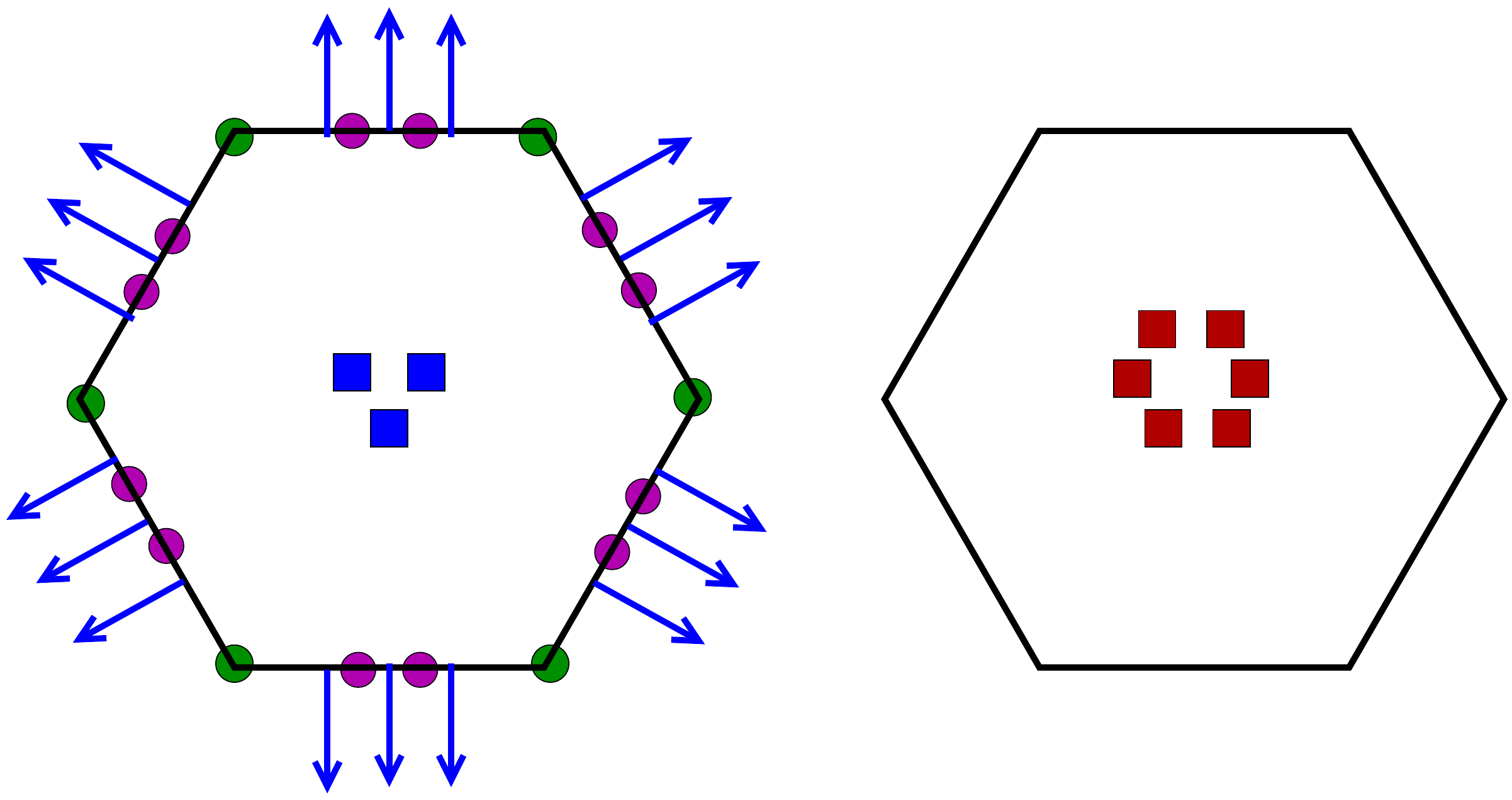} \\[0.5em] 
    \hspace{2mm}$\mathbf{k=1}$ & \hspace{9mm}$\mathbf{k=2}$ & \hspace{9mm}$\mathbf{k=3}$ 
  \end{tabular}
  \caption{Second virtual element formulation: degrees of freedom of
    the virtual element vector-valued fields (left) and the scalar
    polynomial fields (right) of an hexagonal element for the accuracy
    degrees $k=1,2,3$.
    Nodal values are marked by a circular bullet at the vertices; the
    edge moments of the tangential and normal components of the
    vector-valued fields are respectively marked by circular bullets
    and arrows in the interior of the edges.
    Cell polynomial moments for both the vector and scalar fields are
    marked by a square bullet.}
  \medskip
  \label{fig:dofs:FT}
\end{figure}

\medskip
\begin{description}
\item[-]\DOFS{F2}{a} for $k\geq1$, the vertex values $\vvh(\xV)$;
\item[-]\DOFS{F2}{b} for $k\geq1$, the polynomial edge moments of
  $\vvh\cdot\norE$:
  \begin{align}
    \frac{1}{\mE}\int_{\E}\vvh\cdot\norE\qsh\dS
    \qquad \forall \qsh\in\PS{k-1}(\E)
  \end{align}
  for every edge $\E\in\partial\P$;
\item[-]\DOFS{F2}{c} for $k\geq2$, the polynomial edge moments of
  $\vvh\cdot\tngE$:
  \begin{align}
    \frac{1}{\mE}\int_{\E}\vvh\cdot\tngE\qsh\dS
    \qquad \forall \qsh\in\PS{k-2}(\E)
  \end{align}
  for every edge $\E\in\partial\P$;
\item[-]\DOFS{F2}{d} for $k\geq2$, the polynomial cell moments of
  $\vvh$:
  \begin{align}
    \frac{1}{\mP}\int_{\P}\vvh\cdot\qvh\dV
    \qquad \forall\qvh\in\big[\PS{k-2}(\P)\big]^2.
  \end{align}
\end{description}
Figure~\ref{fig:dofs:FT} shows the degrees of freedom of the velocity
vector and the pressure for $k=1,2,3$ on an hexagonal element.

\medskip
\noindent
\begin{remark}
  \label{remark:4}
  In this virtual element space, the normal component of $\vvh$ has an
  increased polynomial degree.
  For example, for $k=1$ the vector field $\vvh\in\Vvhpp{1}(\P)$ is
  such that $\vvh\cdot\norE\in\PS{2}(\E)$ and
  $\vvh\cdot\tngE\in\PS{1}(\E)$ for every edge $\E\in\partial\P$.
  These degrees of freedom are the same used in the low-order MFD
  method of Reference~\cite{BeiraodaVeiga-Gyrya-Lipnikov-Manzini:2009}
  and our VEM is actually a reformulation of this mimetic scheme in
  the variational setting and a generalization to orders of accuracy
  that are higher than one.
  The analysis of the mimetic method and its extension to the
  three-dimensional case is presented
  in~\cite{BeiraodaVeiga-Lipnikov-Manzini:2010} and considers the
  additional edge degrees of freedom as associated with edge bubble
  functions.
\end{remark}

\medskip
\noindent
\begin{remark}
  \label{remark:5}
  Using the degrees of freedom \DOFS{F2}{a} and \DOFS{F2}{b} the
  edge traces $\vvh\cdot\norE\in\PS{k+1}(\E)$ and
  $\vvh\cdot\tngE\in\PS{k}(\E)$ are computable by solving a suitable
  interpolation problem.
  Consider the edge $\E=(\xVp,\xVpp)$ defined by the vertices $\xVp$
  and $\xVpp$.
  Then, 
  \begin{description}
  \item[-] to interpolate $\vvh\cdot\norE\in\PS{k+1}(\E)$ we need
    $k+2$ independent pieces of information, which are provided by
    $\vvh(\xVp)\cdot\norE$, $\vvh(\xVpp)\cdot\norE$ from the degrees
    of freedom \DOFS{F2}{a} and by the $k$ moments of $\vvh\cdot\norE$
    from the degrees of freedom \DOFS{F2}{b};
  \item[-] to interpolate $\vvh\cdot\tngE\in\PS{k}(\E)$ we need $k+1$
    independent pieces of information, which are provided by
    $\vvh(\xVp)\cdot\tngE$, $\vvh(\xVpp)\cdot\tngE$ from the degrees
    of freedom \DOFS{F2}{a} and by the $k-1$ moments of
    $\vvh\cdot\tngE$ from the degrees of freedom \DOFS{F2}{c}.
  \end{description}
\end{remark}

\medskip
\noindent
\begin{lemma}[Unisolvence of the degrees of freedom]
  \label{lemma:BF:unisolvence}
  The degrees of freedom \DOFS{F2}{a}-\DOFS{F2}{d} are unisolvent
  for both the regular and enhanced definition of $\Vvhkpp(\P)$,
  respectively given in~\eqref{eq:FT:regular-space:def}
  and~\eqref{eq:FT:enhanced-space:def}.
  \end{lemma}
\begin{proof}
  The argument that we use to prove the assertion of the lemma is
  similar to the one used to prove the unisolvency of the degrees of
  freedom of the first formulation.
  First, consider a vector field in the virtual element space defined
  in~\eqref{eq:FT:regular-space:def}.
  An integration by parts yields:
  \begin{align}
    \int_{\P}\abs{\nabla\vvh}^2\dV
    = -\int_{\P}\vvh\cdot\Delta\vvh\dV + \sum_{\E\in\partial\P}\int_{\E}\vvh\cdot\nabla\vvh\cdot\norE\dS
    = \TERM{I}{} + \TERM{II}{}.
  \end{align}
  Next, we assume that all the degrees of freedom in \DOFS{F2}{a},
  \DOFS{F2}{b}, \DOFS{F2}{c}, and \DOFS{F2}{d} are zero.
  Then,
  \begin{description}
  \item[-] $\TERM{I}{}$ is zero because
    $\Delta\vvh\in\big[\PS{k-2}(\P)\big]^2$, and, hence, it is a
    degree of freedom of type \DOFS{F2}{d} for $k\geq2$ or zero for
    $k=1$;
  \item[-] to see that $\TERM{II}{}$ is also zero, we use the
    orthogonal decomposition
    $\vvh=(\vvh\cdot\norE)\norE+(\vvh\cdot\tngE)\tngE$ and note that
    $\vvh\cdot\norE=0$ and $\vvh\cdot\tngE=0$ since these traces are
    computed by the interpolation of the degrees of freedom
    \DOFS{F2}{a}, \DOFS{F2}{b}, and \DOFS{F2}{c}, and these data are
    zero by hypothesis.
    Therefore, $\restrict{\vvh}{\E}=0$ on every edge $\E\in\partial\P$
    and all the edge integrals of $\TERM{II}{}$ must be zero.
  \end{description}
  It follows that $\nabla\vvh=0$, i.e., all the spatial derivatives of
  the components of $\vvh$ are zero.
  Therefore, the vector-valued field $\vvh$ is constant on $\P$ and
  since all its degrees of freedom are zero the constant must be zero.
  The assertion of the lemma is finally proved by noting that the
  number of the degrees of freedom is equal to the dimension of
  $\Vvhkpp(\P)$.
  The unisolvence of the degrees of freedom \DOFS{F2}{a}-\DOFS{F2}{d}
  for the enhanced space defined in~\eqref{eq:FT:enhanced-space:def}
  follows by similarly adjusting the argument that is used in the
  proof
  of~\cite[Proposition~2]{Ahmad-Alsaedi-Brezzi-Marini-Russo:2013}.
\end{proof}

\medskip
\noindent
\begin{lemma}
  \label{lemma:5}
  Let $\P$ be an element of mesh $\Th$.
  For every virtual element function $\vvh\in\Vvhkpp(\P)$, the
  polynomial projection $\PizP{k-1}\nabla\vvh$ is computable from the
  degrees of freedom \DOFS{F2}{a}, \DOFS{F2}{b}, and \DOFS{F2}{c} of
  $\vv$.
\end{lemma}
\begin{proof}
  We start from the definition of the orthogonal projection:
  \begin{align}
    \int_{\P}\PizP{k-1}\nabla\vvh:\btauh\dV
    = \int_{\P}\nabla\vvh:\btauh\dV
    \qquad\forall\btauh\in[\PS{k-1}(\P)]^{2\times2}.
  \end{align}
  To prove that the right-hand side is computable from the degrees of
  freedom of $\vvh$, we integrate by parts:
  \begin{align}
    \int_{\P}\nabla\vvh:\btauh\dV
    =
    -\int_{\P}\vvh\cdot\DIV\btauh\dV
    +\sum_{\E\in\partial\P}\int_{\E}\vvh\cdot\btauh\cdot\norE\dS
    = \TERM{I}{} + \TERM{II}{}.
  \end{align}
  Since $\DIV\btauh\in\big[\PS{k-2}(\P)\big]^2$, term $\TERM{I}{}$ is
  computable using the values \DOFS{F2}{d} of $\vvh$.
  Then, we observe that the traces
  $\restrict{\vvh}{\E}\cdot\norE\in\PS{k+1}(\E)$ and
  $\restrict{\vvh}{\E}\cdot\tngE\in\PS{k}(\E)$ are computable from the
  degrees of freedom \DOFS{F2}{a}-\DOFS{F2}{c}.
  On using the decomposition
  $\vvh=(\vvh\cdot\norE)\norE+(\vvh\cdot\tngE)\tngE$, we conclude that
  the trace $\restrict{\vvh}{\E}$ is computable.
  Therefore, all edge integrals and ultimately $\TERM{II}{}$ are
  computable.
\end{proof}

\medskip
\noindent
\begin{lemma}
  \label{lemma:6}
  Let $\P$ be an element of mesh $\Th$.
  For every virtual element function $\vvh\in\Vvhkpp(\P)$, the
  polynomial projection $\PinP{k}\vvh\in\big[\PS{k-1}(\P)\big]^{2}$ is
  computable from the degrees of freedom of $\vvh$.
\end{lemma}
\begin{proof}
  Consider the definition of the elliptic projection operator:
  \begin{align}
    \int_{\P}\nabla\PinP{k}\vvh:\nabla\qvh\dV =
    \int_{\P}\nabla\vvh:\nabla\qvh\dV
    \qquad\qvh\in\big[\PS{k}(\P)\big]^2.
  \end{align}
  We integrate the right-hand side by parts:
  \begin{align}
    \int_{\P}\nabla\vvh:\nabla\qvh\dV
    =
    - \int_{\P}\vvh\cdot\Delta\qvh\dV
    + \sum_{\E\in\partial\P}\int_{\E}\vvh\cdot\nabla\qvh\cdot\norE\dS
    = \TERM{I}{} + \TERM{II}{}.
  \end{align}
  Since we take $\qvh\in\big[\PS{k}(\P)\big]^2$ and
  $\Delta\qvh\in\big[\PS{k-2}(\P)\big]^2$, the first integral in term
  $\TERM{I}{}$ is the moment of $\vvh$ against a vector polynomial of
  degree $k-2$ and is, thus, computable using the degrees of freedom
  of $\vvh$ provided by \DOFS{F2}{d}.
  Then, we observe that the traces
  $\restrict{\vvh}{\E}\cdot\norE\in\PS{k+1}(\E)$ and
  $\restrict{\vvh}{\E}\cdot\tngE\in\PS{k}(\E)$ are computable from the
  degrees of freedom \DOFS{F2}{a}-\DOFS{F2}{c}.
  On using the decomposition
  $\vvh=(\vvh\cdot\norE)\norE+(\vvh\cdot\tngE)\tngE$, also the trace
  $\restrict{\vvh}{\E}$ is computable, cf. Remark~\ref{remark:5}.
  Therefore, all edge integrals and ultimately $\TERM{II}{}$ are
  computable.
\end{proof}

\section{Wellposedness and convergence analysis}
\label{sec:convergence}

In this section, we first prove the wellposedness of the two virtual
element formulations of Section~\ref{sec:VEM}.
Then, we prove that these two formulations are convergent and we
derive error estimates in the energy norm and the $\LTWO$ norm for the
velocity field and the $\LTWO$ norm for the pressure field.
The analysis is the same for both formulations $\FO$ and $\FT$,
regardless of using the non-enhanced or the enhanced definition of the
virtual element space.
For this reason, we use the generic symbol $\Vvhk(\P)$ to refer to the
two virtual element spaces introduced in Section~\ref{sec:VEM}, i.e.,
$\Vvhkp(\P)$ and $\Vvhkpp(\P)$.

Hereafter, we use the capitol letter ``$\Cs$'' to denote a generic
constant that is independent of $\hh$ but may depend on the other
parameters of the discretization, e.g., the polynomial degree $k$, the
mesh regularity constant $\rho$, the stability constants $\alpha_*$
and $\alpha^*$, etc.
The constant $\Cs$ may take a different value at any occurrence.

In some mathematical proofs, we may find it convenient to write
``$A\STACKON{=}{(X)}B$'' to mean that ``$A=B$ follows from equation
(X)'', i.e., to stack the equation reference number on the symbols
``$=$'', ``$\leq$'', ``$\geq$``etc.

\subsection{Wellposedness of the virtual element approximation}
\label{subsec:well-posedness}

To prove the wellposedness of our formulations, we must verify that
the virtual element space $\Vvhk$
and the discontinuous polynomial space $\Qshkk$
are such that: $(i)$ the bilinear form $\ash(\cdot,\cdot)$ is bounded
and coercive; $(ii)$ the bilinear form $\bsh(\cdot,\cdot)$ is bounded
and satisfies the inf-sup condition.
Properties $(i)$ are the immediate consequence of the stability
property~\eqref{eq:ash:stability} and the Cauchy-Schwarz inequality,
which imply~\eqref{eq:ash:coercivity}
and~\eqref{eq:ash:global:continuity}.
We rewrite these two inequalities here for the reader's convenience:
\begin{align}
  \ABS{\ash(\vvh,\wvh)}
  &\leq
  \alpha^*\snorm{\vvh}{1,\Omega}\,\snorm{\wvh}{1,\Omega}
  \phantom{\ash(\vvh,\vvh)}\hspace{-0.5cm}
  \forall\vvh,\,\wvh\in\Vvhk,
  \label{eq:continuity}
  \\[0.5em]
  \alpha_*\snorm{\vvh}{1,\Omega}^2
  &\leq
  \ash(\vvh,\vvh)
  \phantom{\Cs\snorm{\vvh}{1,\Omega}\,\snorm{\wvh}{1,\Omega}}\hspace{-0.5cm}
  \forall\vvh\in\Vvhk.
  \label{eq:coercivity}
\end{align}
Similarly, we can readily prove the boundedness of the bilinear form
$\bsh(\cdot,\cdot)$ by using the Cauchy-Schwarz inequality, so that
\begin{align*}
  \ABS{\bsh(\vvh,\qsh)}\leq\sqrt{2}\snorm{\vvh}{1,\Omega}\,\norm{\qsh}{0,\Omega}
  \qquad\forall\vvh\in\Vvhk,\,\qsh\in\PS{k-1}(\Th).
\end{align*}
Instead, the discrete inf-sup condition is proved in the following
lemma, which relies on the construction of a suitable Fortin operator,
see~\cite{Boffi-Brezzi-Fortin:2013}.
The construction of this operator is the same for both the regular and
the enhanced versions of formulations $\FO$ and $\FT$.

\medskip
\begin{lemma}[Inf-sup condition]
  \label{lemma:inf-sup:condition}
  The bilinear form $\bsh(\cdot,\cdot)$ is \emph{inf-sup stable} on
  $\Vvhk\times\Qshkk$ for the formulations \FO{} and \FT{} and for any
  given polynomial degree $k\geq1$.
\end{lemma}
\begin{proof}
  The proof is essentially based on the construction of a Fortin
  operator $\PiF:\big[\HONE(\Omega)\big]^2\to\Vvhk$ such that 
  \begin{align}
    \bs(\vv,\qsh) &= \bsh(\PiF\vv,\qsh)\qquad\forall\qsh\in\PS{k-1}(\Th),\label{eq:Fortin:bs=bsh}\\[0.5em]
    \norm{\PiF\vv}{1,\Omega} &\leq \norm{\vv}{1,\Omega},\label{eq:Fortin:boundedness}
  \end{align}
  for all $\vv\in\big[\HONE(\Omega)\big]^2$, cf.,
  e.g.,~\cite{Boffi-Brezzi-Fortin:2013}.
  As the proof is based on rather standard arguments, see e.g.,
  \cite[Proposition~3.1]{BeiraodaVeiga-Lovadina-Vacca:2017}, we only
  briefly mention its three main steps.

  \medskip
  In the first step, reasoning as
  in~\cite[Proposition~4.2]{Mora-Rivera-Rodriguez:2015} for the
  non-enhanced virtual element space and~\cite[Theorem~5 (case
    $d=2$)]{Cangiani-Georgoulis-Pryer-Sutton:2016} for the enhanced
  virtual element space, we can prove the existence of a
  quasi-interpolation operator
  $\PioP:\big[\HS{s+1}(\P)\big]^2\to\Vvhk(\P)$, $0\leq\ss\leq\ks$ for
  all elements $\P\in\Th$ such that
  \begin{align*}
    \norm{\vv-\PioP\vv}{0,\P} + \hP\snorm{\vv-\PioP\vv}{1,\P} \leq
    \Cs\hP^{s+1}\snorm{\vv}{s+1,\P}.
  \end{align*}
  Adding all elemental contributions, it is easy to see that
  \begin{align*}
    \norm{\vv-\Pio\vv}{1,\Omega} \leq \Cs \norm{\vv}{1,\Omega},
  \end{align*}
  where $\Pio:\big[\HS{s+1}(\Omega)\big]^2\to\Vvhk$ is the global
  quasi-interpolation operator such that
  $\restrict{\big(\Pio\vv\big)}{\P}=\PioP(\restrict{\vv}{\P})$ for all
  $\P\in\Th$.

  \medskip
  In the second step, for any $\vv\in\big[\HONE(\Omega)\big]^2$ we
  consider a vector-valued virtual element function $\vvh$ such that
  
  \medskip
  \begin{description}
  \item[$(i)$] for $k\geq1$, for all mesh edges $\E$, it holds that
    \begin{align}
      \int_{\E}\qsh\vvh\cdot\norE\dS = \int_{\E}\qsh\vv\cdot\norE\dS
      \qquad \forall\qsh\in\PS{k-1}(\E),
    \end{align}
    where we recall that $\norE$ is the unit normal vector to the edge
    $\E$, whose orientation is fixed once and for all;

    \medskip
  \item[$(ii)$] for $k\geq2$ and for all $\P\in\Th$, it holds that 
    \begin{align}
      \int_{\P}\vvh\cdot\qvh\dV = \int_{\P}\vv\cdot\qvh\dV
      \qquad \forall\qvh\in\big[\PS{k-2}(\P)\big]^2.
    \end{align}   
  \end{description}
  The vector-valued field $\vvh$ is easily determined in $\Vvhk$ by
  properly setting the degrees of freedom of the formulations $\FO$
  and $\FT$.
  In particular, if $\vvh\cdot\norE=\vshx\norEx+\vshy\norEy$ for
  $\norE=(\norEx,\norEy)^T$ and $\vvh=(\vshx,\vshy)^T$, then it holds
  that
  \begin{itemize}
  \item condition $(i)$ is verified by setting accordingly the degrees
    of freedom \DOFS{F1}{b} of formulation $\FO$ and \DOFS{F2}{b} of
    formulation $\FT$;
  \item condition $(ii)$ is verified by setting accordingly the
    degrees of freedom \DOFS{F1}{c} of formulation $\FO$ and
    \DOFS{F2}{d} of formulation $\FT$.
  \end{itemize}
  All the remaining degrees of freedom are set to zero.
  The unisolvency property ensures that such $\vvh$ exists and is
  unique in $\Vvhk$.
  We denote the correspondance between $\vv$ and $\vvh$ by introducing
  the elemental operator $\PitP:\big[\HONE(\P)\big]^2\to\Vvhk(\P)$,
  which is such that $\PitP\vv=\vvh$, and the global operator
  $\restrict{\big(\Pit\vv\big)}{\P}=\PitP(\restrict{\vv}{\P})$ for all
  $\P\in\Th$.

  \medskip
  In the third and last step, we define the Fortin operator as
  $\PiF\vv = \Pio\vv + \Pit(1-\Pio)\vv$.
  This operator satisfies~\eqref{eq:Fortin:bs=bsh}
  and~\eqref{eq:Fortin:boundedness}.
  The discrete inf-sup condition then follows immediately from the
  Fortin argument by using these relations and the continuous inf-sup
  condition~\eqref{eq:exact:inf-sup}.
\end{proof}

The properties of coercivity and boundedness of $\ash(\cdot,\cdot)$
and inf-sup stability (cf. Lemma~\ref{lemma:inf-sup:condition}) and
boundedness of $\bsh(\cdot,\cdot)$ implies the wellposedness of the
two virtual element formulations considered in this work.
We formally state this result in the next theorem.

\medskip
\begin{theorem}[Well-posedness]
  The virtual element formulations \FO{} and \FT{} for any given
  polynomial degree $k\geq1$ have one and only one solution pair
  $(\uvh,\psh)\in\Vvhk\times\Qshkk$, which is such that
  \begin{align}
    \norm{\uvh}{1,\Omega} + \norm{\psh}{0,\Omega} \leq \Cs\norm{\fv}{0,\Omega}.
  \end{align}
\end{theorem}

\medskip
The proof is omitted as this is a standard result in the numerical
approximation of saddle-point problems,
cf.~\cite{Boffi-Brezzi-Fortin:2013}.

\subsection{Preliminary results}
\label{subsec:preliminary}

To derive the error estimates in the energy norm and the $\LTWO$ norm,
we need three technical lemmas that are preliminarly reported here.
The first two lemmas are reported without the proof as they are
well-known results from the approximation theory, see
\cite{Brenner-Scott:1994,Dupont-Scott:1980}.
In particular, the first lemma provides an estimate of the projection
error and is the vector version of the analogous result reported
in~\cite{BeiraodaVeiga-Brezzi-Cangiani-Manzini-Marini-Russo:2013} for
the scalar case.

\medskip
\begin{lemma}[Projection error]
  \label{lemma:projection:error}
  Under Assumptions~\textbf{(M1)}-\textbf{(M2)}, for every
  vector-valued field $\vv\in\big[\HS{s+1}(\P)\big]^2$
  with $1\leq\ss\leq\ell$ for some given integer number $\ell$, there
  exists a vector polynomial $\vv_{\pi}\in\big[\PS{\ell}(\P)\big]^2$
  such that
  \begin{align}
    &\norm{\vv-\vv_{\pi}}{0,\P} + \hP\snorm{\vv-\vv_{\pi}}{1,\P}\leq\Cs\hP^{s+1}\snorm{\vv}{s+1,\P},
  \end{align}
  where $\Cs$ is some positive constant that is independent of $\hP$
  but may depend on the polynomial degree $\ell$ and the mesh
  regularity constant $\varrho$.
\end{lemma}

\medskip
\noindent
The second lemma reports an estimate of the approximation errors for
the interpolants $\vvI$ and $\qsI$.
According
to~\cite{BeiraodaVeiga-Brezzi-Cangiani-Manzini-Marini-Russo:2013}, we
define the local interpolation $\vvI\in\Vvhk(\P)$ of a (smooth enough)
field $\vv$ as the virtual element field that has the same degrees of
freedom.
Similarly, we define the local interpolation $\qsI\in\Qshkk$ of a
(smooth enough) scalar function $\qs$ as the polynomial function that
has the same degrees of freedom.
Therefore, $\restrict{(\qsI)}{\P}\in\PS{k-1}(\P)$ for all elements
$\P\in\Th$, and
\begin{align}
  \int_{\Omega}\qsI(\xv)\dV = 0,
  \label{eq:interp:zero-average}
\end{align}
since according to~\eqref{eq:SV:scalar:space:def} it also holds that
$\qsI\in\LTWOzr(\Omega)$.

\medskip
\begin{lemma}[Interpolation error]
  \label{lemma:interpolation:error}
  Under Assumptions~\textbf{(M1)}-\textbf{(M2)}, for every
  vector-valued field $\vv\in\big[\HS{s+1}(\P)\big]^2$ and scalar
  function $\qs\in\HS{s}(\P)$ with $1\leq\ss\leq\ell$, for some given
  integer number $\ell$, there exist a vector-valued field
  $\vvI\in\Vvh{\ell}(\P)$ and a scalar field
  $\qsI\in\PS{\ell-1}(\P)$ such that
  \begin{align}
    \norm{\vv-\vvI}{0,\P} + \hP\snorm{\vv-\vvI}{1,\P}\leq\Cs\hP^{s+1}\snorm{\vv}{s+1,\P},\\[0.5em]
    \norm{\qs-\qsI}{0,\P} + \hP\snorm{\qs-\qsI}{1,\P}\leq\Cs\hP^{s} \snorm{\qs}{s,\P},
  \end{align}
  for some positive constant $\Cs$ that is independent of $\hP$ but
  may depend on the polynomial degree $\ell$ and the mesh regularity
  constant $\varrho$.
\end{lemma}

\medskip
\noindent
In the last lemma of this section we prove a relation between $\uvh$,
$\uvI$, $\psh$, and $\psI$ that will be used in the convergence
analysis of the next sections.

\medskip
\begin{lemma}
  Let $(\uv,\ps)\in\big[\HS{s+1}(\Omega)\big]^2\times\LTWOzr(\Omega)$,
  $s\geq1$, be the exact solution of the variational formulation of
  the Stokes problem given
  in~\eqref{eq:stokes:var:A}-\eqref{eq:stokes:var:B} and
  $(\uvI,\psI)\in\Vvhk\times\Qshkk$ the corresponding virtual element
  interpolation.
  Let $(\uvh,\psh)\in\Vvhk\times\Qshkk$ be the virtual element
  approximation to $(\uv,\ps)$
  solving~\eqref{eq:stokes:vem:A}-\eqref{eq:stokes:vem:B}.
  Then, it holds that
  \begin{align}
    \bs(\uvh-\uvI,\psh-\psI) = 0.
    \label{eq:aux:20}
  \end{align}
\end{lemma}
\begin{proof}
  Let $\P$ be an element of mesh $\Th$ and $k\geq1$ an integer number.
  Consider the function $\vv\in\big[\HS{s+1}(\P)\big]^2$, $s\geq1$,
  and its virtual element interpolant $\vvI\in\Vvhk(\P)$.
  Integrating by parts twice and using the definition of the
  interpolant $\vvI$, we find that:
  \begin{align}
    -\bsP(\vv,\qsh)
    = \int_{\P}\qsh\DIV\vv\dV
    &=
    -\int_{\P}\nabla\qsh\cdot\vv\dV
    +\sum_{\E\in\partial\P}\int_{\E}\qsh\norE\cdot\vv\ds
    \nonumber\\[0.5em]
    &=
    -\int_{\P}\nabla\qsh\cdot\vvI\dV
    +\sum_{\E\in\partial\P}\int_{\E}\qsh\norE\cdot\vvI\ds
    = \int_{\P}\qsh\DIV\vvI\dV
    = -\bsP(\vvI,\qsh),
    \label{eq:bsP=bsP}
  \end{align}
  which holds for all $\qsh\in\PS{k-1}(\P)$.
  The identity chain~\eqref{eq:bsP=bsP} implies that
  $\bsP(\vv,\qsh)=\bsP(\vvI,\qsh)$, and, adding this relation over all
  elements $\P$ yields $\bs(\vv,\qsh)=\bs(\vvI,\qsh)$.
  By taking $\vv=\uv$, equation \eqref{eq:stokes:var:B} implies that
  $\bs(\uvI,\qsh)=\bs(\uv,\qsh)=0$.
  Likewise, by taking $\vvh=\uvh$, equations~\eqref{eq:bsh=bs}
  and~\eqref{eq:stokes:vem:B} imply that
  $\bs(\uvh,\qsh)=\bsh(\uvh,\qsh)=0$.
  Taking the difference of the left-most left-hand side of the two
  previous identities yields $\bs(\uvh-\uvI,\qsh)=0$, which holds for
  all $\qsh\in\Qshkk$.
  The assertion of the lemma readily follows by taking $\qsh=\psh-\psI$.
\end{proof}

\subsection{Error estimate in the energy norm}
\label{subsec:error:estimate:H1}

\begin{theorem}
  \label{theorem:H1:estimate}
  Let $\uv\in\big[\HS{s+1}(\Omega)\cap\HONEzr(\Omega)\big]^2$ and
  $\ps\in\HS{s}(\Omega)\cap\LTWOzr(\Omega)$, $1\leq\ss\leq\ks$,
  be the solution of the variational formulation of the Stokes problem
  given in~\eqref{eq:stokes:var:A}-\eqref{eq:stokes:var:B}.
  Let $(\uvh,\psh)\in\Vvhk\times\Qshkk$ be the solution of the virtual
  element variational formulation
  \eqref{eq:stokes:vem:A}-\eqref{eq:stokes:vem:B} under the mesh
  regularity assumptions $\textbf{(M1)}-\textbf{(M2)}$ and for any
  polynomial degree $k\geq1$.
  Then, there exists a real, strictly positive constant $\Cs$
  independent of $\hh$ such that the following abstract estimate
  holds:
  \begin{align}
    \snorm{\uv-\uvh}{1,\Omega} + \norm{\ps-\psh}{0,\Omega}
    \leq\Cs\Bigg(
    \snorm{\uv-\uvI}{1,\Omega}
    + \snorm{\uv-\uv_{\pi}}{1,\hh}
    + \norm{\ps-\psI}{0,\Omega}
    + \sup_{\vvh\in\Vvhk\setminus\{\mathbf{0}\}}\frac{\ABS{\bil{\fvh}{\vvh}-\scal{\fv}{\vvh}}}{\snorm{\vvh}{1,\Omega}}
    \Bigg)
    \label{eq:theo:abstract}
  \end{align}
  where $\uvI\in\Vvhk$ and $\psI\in\Qshkk$ are the interpolants of
  $\uv$ and $\ps$ from Lemma~\ref{lemma:interpolation:error}, and
  $\uv_{\pi}\in\big[\PS{k}(\Th)\big]^2$ is any polynomial
  approximation of $\uv$ that is defined in accordance with
  Lemma~\ref{lemma:projection:error}.
  Moreover, if
  $\fv\in\big[\HS{t}(\Omega)\big]^2$, $t\geq0$, it holds that
  \begin{align}
    \snorm{\uv-\uvh}{1,\Omega} + \norm{\ps-\psh}{0,\Omega}
    \leq \Cs\Big(
    \hh^{s}\big(\norm{\uv}{s+1,\Omega} + \norm{\ps}{s,\Omega}\big) + \hh^{\min(t,\kb)+1}\norm{\fv}{t,\Omega} 
    \Big),
    \label{eq:the:H1:estimates}
  \end{align}
  where $\kb$ is defined as in~\eqref{eq:fvh:def}.
\end{theorem}
\begin{proof}
  We add and subtract $\uvI$ and $\psI$ in the two terms of the
  left-hand side of~\eqref{eq:theo:abstract} and use the triangle
  inequality:
  \begin{align}
    \snorm{\uv-\uvh}{1,\Omega} \leq \snorm{\uv-\uvI}{1,\Omega} + \snorm{\uvI-\uvh}{1,\Omega}, \label{eq:H1:proof:00} \\[0.2em]
    \norm {\ps-\psh}{0,\Omega} \leq \norm {\ps-\psI}{0,\Omega} + \snorm{\psI-\psh}{1,\Omega}. \label{eq:H1:proof:10}
  \end{align}
  The two terms $\snorm{\uv-\uvh}{1,\Omega}$ and
  $\norm{\ps-\psh}{0,\Omega}$ are in the right-hand side
  of~\eqref{eq:theo:abstract}.
  We can estimate them by applying
  Lemma~\ref{lemma:interpolation:error} to
  obtain~\eqref{eq:the:H1:estimates}.
  Instead, to estimate the second term of the right-hand side
  of~\eqref{eq:H1:proof:00} and~\eqref{eq:H1:proof:10}, we proceed as
  follows.
  Let $\dvh=\uvh-\uvI\in\Vvhk$.
  Starting from the coercivity inequality~\eqref{eq:coercivity}, we
  find that:
  \begin{align}
    \begin{array}{lll}
      &\alpha_*\snorm{\dvh}{1,\Omega}^2
      \leq \ash(\dvh,\dvh)
      &\hspace{-3cm}\mbox{\big[split $\dvh=\uvh-\uvI$\big]}\nonumber\\[0.4em]
      &\qquad= \ash(\uvh,\dvh) - \ash(\uvI,\dvh)
      &\hspace{-3cm}\mbox{\big[use~\eqref{eq:stokes:vem:A} and add $\pm\uv_{\pi}$\big]}\nonumber\\[0.2em]
      &\qquad= \bil{\fvh}{\dvh} - \bsh(\dvh,\psh) - \sum_{\P\in\Th}\Big(\asPh(\uvI-\uv_{\pi},\dvh) + \asPh(\uv_{\pi},\dvh) \Big)
      &\hspace{-3cm}\mbox{\big[use~\eqref{eq:bsh=bs} and~\eqref{eq:consistency}\big]}\nonumber\\[0.2em]
      &\qquad= \bil{\fvh}{\dvh} - \bs(\dvh,\psh) - \sum_{\P\in\Th}\Big(\asPh(\uvI-\uv_{\pi},\dvh) + \asP (\uv_{\pi},\dvh) \Big)
      &\hspace{-3cm}\mbox{\big[use~\eqref{eq:aux:20} and add $\pm\uv$\big]}\nonumber\\[0.2em]
      &\qquad= \bil{\fvh}{\dvh} - \bs(\dvh,\psI) - \sum_{\P\in\Th}\asPh(\uvI-\uv_{\pi},\dvh) - \sum_{\P\in\Th}\Big( \asP (\uv_{\pi}-\uv,\dvh) + \asP(\uv,\dvh) \Big)
      &\hspace{-1.5cm}\mbox{\big[use~\eqref{eq:asP:def}\big]}\nonumber\\[0.2em]
      &\qquad= \bil{\fvh}{\dvh} - \bs(\dvh,\psI) - \sum_{\P\in\Th}\asPh(\uvI-\uv_{\pi},\dvh) - \sum_{\P\in\Th}\asP (\uv_{\pi}-\uv,\dvh) - \as(\uv,\dvh)
      &\hspace{-1.5cm}\mbox{\big[use~\eqref{eq:stokes:var:A}\big]}\nonumber\\[0.2em]
      &\qquad= \bil{\fvh}{\dvh} - \bs(\dvh,\psI) - \sum_{\P\in\Th}\asPh(\uvI-\uv_{\pi},\dvh) - \sum_{\P\in\Th}\asP (\uv_{\pi}-\uv,\dvh) - \Big( \scal{\fv}{\dvh} - \bs(\dvh,\ps) \Big)
      &
      \nonumber\\[0.5em]
      &\qquad
      = \Big[ \bil{\fvh}{\dvh} - \scal{\fv}{\dvh} \Big]
      + \Big[ \bs(\dvh,\ps)   - \bs(\dvh,\psI)   \Big]
      + \bigg[ - \sum_{\P\in\Th}\asPh(\uvI-\uv_{\pi},\dvh) - \sum_{\P\in\Th}\asP (\uv_{\pi}-\uv,\dvh) \bigg]
      \nonumber\\[0.2em]
      &\qquad
      = \big[\TERM{R}{1}\big] + \big[\TERM{R}{2}\big] + \big[\TERM{R}{3}\big].
    \end{array}
  \end{align}
  We derive an upper bound of term $\TERM{R}{1}$ as follows:
  \begin{align*}
    \ABS{\TERM{R}{1}}
    = \ABS{\bil{\fvh}{\dvh} - \scal{\fv}{\dvh}}
    \leq \left[\sup_{\vvh\in\Vvhk\setminus\{\mathbf{0}\}}\frac{ \ABS{\bil{\fvh}{\vvh} - \scal{\fv}{\vvh}} }{ \snorm{\vvh}{1,\Omega} }\right]\,\snorm{\dvh}{1,\Omega}.
  \end{align*}
  We derive an upper bound of term $\TERM{R}{2}$ by
  using
  the Cauchy-Schwarz inequality:
  \begin{align*}
    \ABS{\TERM{R}{2}}
    = \ABS{\bs(\dvh,\ps-\psI)}
    \leq \norm{\DIV\dvh}{0,\Omega}\,\norm{\ps-\psI}{0,\Omega}
    \leq \Cs\snorm{\dvh}{1,\Omega}\,\norm{\ps-\psI}{0,\Omega}.
  \end{align*}
  To derive an upper bound of term $\TERM{R}{3}$, we use the
  continuity of $\ash(\cdot,\cdot)$,
  cf.~\eqref{eq:continuity}, and $\as(\cdot,\cdot)$, we add
  and subtract $\uv$ in the first summation argument, and, in the last
  step, we use definition~\eqref{eq:broken:seminorm} of the broken
  seminorm $\snorm{\cdot}{1,\hh}$ to find that
  \begin{align*}
    &\ABS{\TERM{R}{3}}
    = \ABS{ \sum_{\P\in\Th}\Big( \asPh(\uvI-\uv_{\pi},\dvh) + \asP (\uv_{\pi}-\uv,\dvh) \Big) }
    \leq \sum_{\P\in\Th}\Big(
    \alpha^*\snorm{\uvI-\uv_{\pi}}{1,\P} +
    \snorm{\uv_{\pi}-\uv}{1,\P}
    \Big)\,\snorm{\dvh}{1,\P}
    \nonumber\\[0.5em]
    &\quad\leq \sum_{\P\in\Th}\Big(
    \alpha^*\snorm{\uvI-\uv}{1,\P} +
    (1+\alpha^*)\snorm{\uv-\uv_{\pi}}{1,\P}
    \Big)\,\snorm{\dvh}{1,\P}
    \leq\Big(\,
    \alpha^*\snorm{\uv-\uv_{I}}{1,\Omega} + (1+\alpha^*)\snorm{\uv-\uv_{\pi}}{1,\hh}
    \,\Big)
    \,\snorm{\dvh}{1,\Omega}.
  \end{align*}
  Let $\ssh=\psh-\psI\in\Qshkk$.
  In view of the discrete inf-sup condition,
  cf. Lemma~\ref{lemma:inf-sup:condition}, there exists a real,
  strictly positive constant $\tilde{\beta}$ and a virtual element
  vector-valued field $\vvh$ such that
  \begin{align}
    \begin{array}{lll}
      &\tilde{\beta}\norm{\ssh}{0,\Omega}\snorm{\vvh}{1,\Omega}
      \leq \bsh(\vvh,\ssh)
      &\hspace{-1.75cm}\mbox{\big[split $\ssh=\psh-\psI$\big]}\nonumber\\[0.5em]
      &\qquad= \bsh(\vvh,\psh) - \bsh(\vvh,\psI)
      &\hspace{-1.75cm}\mbox{\big[use~\eqref{eq:stokes:vem:A}\big]}\nonumber\\[0.5em]
      &\qquad= -\ash(\uvh,\vvh) + \bil{\fvh}{\vvh} - \bsh(\vvh,\psI)
      &\hspace{-1.75cm}\mbox{\big[add~\eqref{eq:stokes:var:A}\big]}\nonumber\\[0.5em]
      &\qquad= -\ash(\uvh,\vvh) + \big[ \as(\uv,\vvh) + \bs(\vvh,\ps) - \scal{\fv}{\vvh} \big] + \bil{\fvh}{\vvh} - \bsh(\vvh,\psI)
      &\hspace{-1.75cm}\mbox{\big[use~\eqref{eq:asP:def} and~\eqref{eq:ash:def}\big]}\nonumber\\[0.5em]
      &\qquad=
      \bil{\fvh}{\vvh} - \scal{\fv}{\vvh}
      + \bs(\vvh,\ps) - \bsh(\vvh,\psI)
      + \sum_{\P\in\Th}\Big( \asP(\uv,\vvh) - \asPh(\uvh,\vvh) \Big)
      &\hspace{-1.75cm}\mbox{\big[use~\eqref{eq:consistency} with $\qvh=\uv_{\pi}$\big]}\nonumber\\[1.em]
      &\qquad=
      \Big[ \bil{\fvh}{\vvh} - \scal{\fv}{\vvh} \Big]
      + \Big[ \bs(\vvh,\ps) - \bsh(\vvh,\psI)      \Big]
      + \sum_{\P\in\Th}\Big( \asP(\uv-\uv_{\pi},\vvh) - \asPh(\uvh-\uv_{\pi},\vvh) \Big)
      \nonumber\\
      &
      \qquad= \big[\TERM{R}{4}\big] + \big[\TERM{R}{5}\big] + \big[\TERM{R}{6}\big].
    \end{array}
  \end{align}
  We derive an upper bound of term $\TERM{R}{4}$ using the same steps
  as for the bound of term $\TERM{R}{1}$ with $\vvh$ instead of
  $\dvh$:
  \begin{align*}
    \ABS{\TERM{R}{4}}
    = \ABS{\bil{\fvh}{\vvh} - \scal{\fv}{\vvh}}
    \leq \left[\sup_{\vvh\in\Vvhk\setminus\{\mathbf{0}\}}\frac{ \ABS{\bil{\fvh}{\vvh} - \scal{\fv}{\vvh}} }{ \snorm{\vvh}{1,\Omega} }\right]\,\snorm{\vvh}{1,\Omega}.
  \end{align*}
  We derive an upper bound of term $\TERM{R}{5}$ using the same steps
  as for the bound of term $\TERM{R}{2}$ with $\vvh$ instead of
  $\dvh$:
  \begin{align*}
    \ABS{\TERM{R}{5}}
    \leq \snorm{\vvh}{1,\Omega}\,\norm{\psI-\ps}{0,\Omega}.
  \end{align*}
  We derive an upper bound of term $\TERM{R}{6}$ using the same steps
  as for the bound of term $\TERM{R}{3}$ with $\vvh$ instead of
  $\dvh$ and $\uvh$ instead of $\uvI$
  \begin{align*}
    \ABS{\TERM{R}{6}}
    &\leq
    \Big(
    \alpha^*\snorm{\uv-\uvh}{1,\Omega} +
    (1+\alpha^*)\snorm{\uv-\uv_{\pi}}{1,\hh}
    \Big)
    \,\snorm{\vvh}{1,\Omega}.
  \end{align*}

  \medskip
  Finally, we use the bound of terms $\TERM{R}{1}-\TERM{R}{3}$ to
  control $\snorm{\uvI-\uvh}{1,\Omega}$
  in~\eqref{eq:H1:proof:00}.
  Then, we use the bound of terms $\TERM{R}{4}-\TERM{R}{4}$ and
  $\snorm{\uv-\uvh}{1,\Omega}$ to control
  $\snorm{\psI-\psh}{1,\Omega}$
  in \eqref{eq:H1:proof:10}.
  The first assertion of the theorem follows on using the resulting
  inequalities to control the left-hand side
  of~\eqref{eq:theo:abstract}.
  The estimate~\eqref{eq:the:H1:estimates} follows from a
  straightforward application of Lemmas~\ref{lemma:projection:error}
  and~\ref{lemma:interpolation:error}, and
  estimates~\eqref{eq:fv:bound:0}-\eqref{eq:fv:bound:1} in the
  right-hand side of~\eqref{eq:theo:abstract}.
\end{proof}

\subsection{Error estimate in the $\LTWO$ norm for the velocity field}
\label{subsec:error:estimate:L2}

\begin{theorem}
  \label{theorem:L2:estimate}
  Let $\uv\in\big[\HS{s+1}(\Omega)\cap\HONEzr(\Omega)\big]^2$ and
  $\ps\in\HS{s}(\Omega)\cap\LTWOzr(\Omega)$, $1\leq\ss\leq\ks$,
  be the exact solution of the variational formulation of the Stokes
  problem given in~\eqref{eq:stokes:var:A}-\eqref{eq:stokes:var:B}
  with $\fv\in\big[\HS{t}(\Omega)\big]^2$, $0\leq\ts$.
  Let $(\uvh,\psh)\in\Vvhk\times\Qshkk$
  be the solution of the virtual element variational formulation
  \eqref{eq:stokes:vem:A}-\eqref{eq:stokes:vem:B} under the mesh
  regularity assumptions $\textbf{(M1)}-\textbf{(M2)}$.
  Then, it holds:
  \begin{align}
    \label{eq:theo:L2:estimates}
    \norm{\uv-\uvh}{0,\Omega}
    \leq \Cs\bigg( \hh^{s+1}\Big( \norm{\uv}{s+1,\Omega} + \norm{\ps}{s,\Omega} \Big)
    + \hh^{\min(t,\kb)+1}\norm{\fv}{t,\Omega} \bigg)
  \end{align}
  for some real, strictly positive constant $\Cs$ independent of $\hh$
  and where $\kb$ is defined as in~\eqref{eq:fvh:def}.
\end{theorem}
\begin{proof}
  In the derivation of the $\LTWO$ error for the virtual element
  approximation of the velocity vector $\uv$, we make use of the
  solution
  $(\psiv,\phis)\in\big[\HTWO(\Omega)\cap\HONEzr(\Omega)\big]^2\times\big[\HONE(\Omega)\cap\LTWOzr(\Omega)\big]$
  of the dual problem:
  \begin{align}
    -\Delta\psiv - \nabla\phis = \uv -\uvh & \qquad\textrm{in~}\Omega,\label{eq:dual:A}\\[0.2em]
    \DIV\psiv = 0                          & \qquad\textrm{in~}\Omega.\label{eq:dual:B}
  \end{align}
  Since $\psiv\in\big[\HTWO(\Omega)\big]^2$ and
  $\phis\in\HONE(\Omega)$, the application of
  Lemmas~\ref{lemma:projection:error}
  and~\ref{lemma:interpolation:error} yields
  \begin{align}
    \snorm{\psiv-\psivI}{1,\Omega} + \snorm{\psiv-\psiv_{\pi}}{1,\hh}
    &\leq \Cs\hh\snorm{\psiv}{2,\Omega},
    \label{eq:bound:psiv}\\[0.5em]
    \norm{\phis-\phisI}{0,\Omega}
    &\leq \Cs\hh\snorm{\phis}{1,\Omega},
    \label{eq:bound:phis}
  \end{align}
  where $\psivI$ and $\phisI$ are the virtual element interpolant of
  $\psiv$ and $\phis$ in $\Vvhk$ and $\Qshkk$, respectively,
  $\psiv_{\pi}$ is the polynomial approximation of $\psiv$ according
  to Lemma~\ref{lemma:projection:error}, and $\norm{\,\cdot\,}{1,\hh}$
  in~\eqref{eq:bound:psiv} is the ``broken'' norm defined in
  Eq.~\eqref{eq:broken:seminorm}.
  Under the assumption that the domain $\Omega$ is convex, the
  solution pair $(\psiv,\phis)$ has the following regularity property:
  \begin{align}
    \norm{\psiv}{2,\Omega} + \norm{\phis}{1,\Omega} \leq \Cs\norm{\uv-\uvh}{0,\Omega}.
    \label{eq:regularity:bound:psiv}
  \end{align}

  Then, we use the definition of the $\LTWO$ norm, and note that the
  boundary integral on $\partial\Omega$ of $\nv\cdot(\uv-\uvh)$, which
  is originated by an integration by parts, is zero since $\uv=\uvh=0$
  on $\partial\Omega$, and we find that
  \begin{align}
    \begin{array}{lll}
      &\norm{\uv-\uvh}{0,\Omega}^2 = \int_{\Omega}(\uv-\uvh)\cdot(\uv-\uvh)\dV
      &\hspace{-5cm}\mbox{\big[use~\eqref{eq:dual:A}\big]}\nonumber\\[0.5em]
      &\qquad= \int_{\Omega}\big(-\Delta\psiv-\nabla\phis\big)\cdot(\uv-\uvh)\dV
      &\hspace{-5cm}\mbox{\big[integrate by parts both terms\big]}\nonumber\\[1.em]
      &\qquad= \int_{\Omega}\nabla\psiv\cdot\nabla(\uv-\uvh)\dV + \int_{\Omega}\phis\,\DIV(\uv-\uvh)\dV
      &\hspace{-5cm}\mbox{\big[use~\eqref{eq:as:def}-\eqref{eq:bs:def}\big]}\nonumber\\[1.25em]
      &\qquad= \as(\psiv,\uv-\uvh) - \bs(\uv-\uvh,\phis)
      &\hspace{-5cm}\mbox{\big[add $\pm\psivI$ and $\pm\phisI$ \big]}\nonumber\\[1.em]
      &\qquad= \big[ \as(\psiv-\psivI,\uv-\uvh)  \big]
      +  \big[ \as(\psivI,\uv-\uvh)        \big]
      +  \big[ -\bs(\uv-\uvh,\phis-\phisI) \big]
      +  \big[ -\bs(\uv-\uvh,\phisI)       \big]
      &\nonumber\\[1.em]
      &\qquad= \big[\TERM{R}{1}\big] + \big[\TERM{R}{2}\big] + \big[\TERM{R}{3}\big] + \big[\TERM{R}{4}\big].
      \label{eq:L2:error-estimate}  
    \end{array}
  \end{align}
  We estimate separately each term $\TERM{R}{i}$, $i=1,\ldots,4$.
  
  \medskip
  We derive an upper bound for term $\TERM{R}{1}$ by using the
  continuity of the bilinear form $\as(\cdot,\cdot)$ and
  inequalities~\eqref{eq:bound:psiv}
  and~\eqref{eq:regularity:bound:psiv}:
  \begin{align}
    \ABS{\TERM{R}{1}}
    &
    = \ABS{\as(\psiv-\psivI,\uv-\uvh)}
    \leq \snorm{\psiv-\psivI}{1,\Omega}\,\snorm{\uv-\uvh}{1,\Omega}
    \STACKON{\leq}{\eqref{eq:bound:psiv}} \Cs\hs\snorm{\psiv}{2,\Omega}\,\snorm{\uv-\uvh}{1,\Omega}
    \nonumber\\[0.5em]
    &
    \STACKON{\leq}{\eqref{eq:regularity:bound:psiv}} \Cs\hs\norm{\uv-\uvh}{0,\Omega}\,\snorm{\uv-\uvh}{1,\Omega}.
  \end{align}
  
  \medskip
  We split term $\TERM{R}{2}$ into three subterms by
  using~\eqref{eq:stokes:var:A}, adding~\eqref{eq:stokes:vem:A} and
  rearranging the terms:
  \begin{align}
    \TERM{R}{2}
    &
    = \as(\psivI,\uv-\uvh)
    = \as(\uv,\psivI) - \as(\uvh,\psivI)
    \nonumber\\[0.5em]
    &= \scal{\fv}{\psivI} - \bs(\psivI,\ps) - \as(\uvh,\psivI)
    + \Big( \ash(\uvh,\psivI) + \bsh(\psivI,\psh) - \bil{\fvh}{\psivI} \Big)
    \nonumber\\[0.5em]
    &
    = \big[ \scal{\fv}{\psivI} - \bil{\fvh}{\psivI} \big]
    + \big[ \bsh(\psivI,\psh)  - \bs(\psivI,\ps)    \big]
    + \big[ \ash(\uvh,\psivI)  - \as(\uvh,\psivI)   \big]
    \nonumber\\[0.5em]
    &= \TERM{R}{21} + \TERM{R}{22} + \TERM{R}{23}.
  \end{align}
  To bound term $\TERM{R}{21}$, we use
  inequalities~\eqref{eq:fv:bound:0} and~\eqref{eq:fv:bound:1},
  the boundedness of the interpolation operator, and
  inequality~\eqref{eq:regularity:bound:psiv}, and we find that
  \begin{align}
    \ABS{\TERM{R}{21}}
    \leq \Cs\hh^{\min(s,\kb)+1}\norm{\fv}{s,\Omega} \snorm{\psivI}{1,\Omega}
    \leq \Cs\hh^{\min(s,\kb)+1}\norm{\fv}{s,\Omega} \norm {\uv-\uvh}{0,\Omega}.
  \end{align}
  To derive an upper bound for term $\TERM{R}{22}$, we first note that
  $\bsh(\psivI,\psh)=\bs(\psivI,\psh)$ from~\eqref{eq:bsh=bs} and
  that we can subtract $\bs(\psiv,\psh-\ps)=0$, which is zero since
  $\DIV\psiv=0$, cf.~\eqref{eq:dual:B}.
  Then, we use the Cauchy-Schwarz inequality,
  inequalities~\eqref{eq:bound:psiv}
  and~\eqref{eq:regularity:bound:psiv}, and we find that

  \noindent
  \begin{align}
    \ABS{\TERM{R}{22}}
    &
    = \ABS{\bs(\psivI,\psh-\ps)}
    = \ABS{\bs(\psivI-\psiv,\psh-\ps)}
    \leq \norm{\DIV(\psivI-\psiv)}{0,\Omega}\,\norm{\psh-\ps}{0,\Omega}
    \nonumber\\[0.5em]
    &\leq \Cs\snorm{\psivI-\psiv}{1,\Omega}  \,\norm{\psh-\ps}{0,\Omega}
    \STACKON{\leq}{\eqref{eq:bound:psiv}}            \Cs\hh\snorm{\psiv}{2,\Omega}       \,\norm{\psh-\ps}{0,\Omega}
    \STACKON{\leq}{\eqref{eq:regularity:bound:psiv}} \Cs\hh\norm{\uv-\uvh}{0,\Omega}    \,\norm{\psh-\ps}{0,\Omega}.
  \end{align}

  \noindent
  To estimate $\TERM{R}{23}$, we first note that the local consistency
  property of the bilinear form $\ash(\cdot,\cdot)$ implies that
  \begin{align}
    \asPh(\uvh,\psivI) - \asP(\uvh,\psivI)
    &= \asPh(\uvh-\uv_{\pi},\psivI-\psiv_{\pi}) - \asP(\uvh-\uv_{\pi},\psivI-\psiv_{\pi}),
    \label{eq:R23:aux}
  \end{align}
  where $\uv_{\pi}$ and $\psiv_{\pi}$ are suitable polynomial
  approximations of $\uv$ and $\psiv$ satisfying the assumptions of
  Lemma~\ref{lemma:projection:error}.
  Then, we use this identity, Lemmas~\ref{lemma:projection:error}
  and~\ref{lemma:interpolation:error} and
  inequality~\eqref{eq:regularity:bound:psiv} to obtain the bound on
  $\TERM{R}{23}$ as follows:
  \begin{align}
    \ABS{\TERM{R}{23}}
    &= \ABS{\ash(\uvh,\psivI) - \as(\uvh,\psivI)}
    =  \ABS{\sum_{\P\in\Th}\Big( \asPh(\uvh-\uv_{\pi},\psivI-\psiv_{\pi}) - \asP(\uvh-\uv_{\pi},\psivI-\psiv_{\pi}) \Big)}
    \nonumber\\[0.125em]
    &
    \leq (1+\alpha^*)\sum_{\P\in\Th} \snorm{\uvh-\uv_{\pi}}{1,\P}\,\snorm{\psivI-\psiv_{\pi}}{1,\P}
    \leq (1+\alpha^*)
    \left(\sum_{\P\in\Th} \snorm{\uvh-\uv_{\pi}}{1,\P}^2     \right)^{\frac12}
    \left(\sum_{\P\in\Th} \snorm{\psivI-\psiv_{\pi}}{1,\P}^2 \right)^{\frac12}.
    \label{eq:L2:proof:R23:00}
  \end{align}
  We add and subtract $\uv$ and $\psiv$, and use the triangular inequality to find that
  \begin{align}
    \snorm{\uvh-\uv_{\pi}}{1,\P}^2
    &= \left( \snorm{\uvh-\uv}{1,\P} + \snorm{\uv-\uv_{\pi}}{1,\P} \right)^2
    \leq 2\snorm{\uvh-\uv}{1,\P}^2 + 2\snorm{\uv-\uv_{\pi}}{1,\P}^2,
    \label{eq:R23:aux:1}\\[0.5em]
    \snorm{\psivI-\psiv_{\pi}}{1,\P}^2
    &= \left( \snorm{\psivI-\psiv}{1,\P} + \snorm{\psiv-\psiv_{\pi}}{1,\P} \right)^2
    \leq 2\snorm{\psivI-\psiv}{1,\P}^2 + 2\snorm{\psiv-\psiv_{\pi}}{1,\P}^2.
    \label{eq:R23:aux:2}
  \end{align}
  Using
  inequalities~\eqref{eq:R23:aux:1},~\eqref{eq:R23:aux:2},~\eqref{eq:bound:psiv},
  and~\eqref{eq:regularity:bound:psiv}, we find that
  \begin{align}
    \ABS{\TERM{R}{23}}
    &
    \STACKON{\leq}{\eqref{eq:R23:aux:1},\eqref{eq:R23:aux:2}}
    \Cs
    \Big(\snorm{\uvh-\uv}{1,\Omega}+\snorm{\uv-\uv_{\pi}}{1,\hh}\Big)
    \Big(\snorm{\psivI-\psiv}{1,\Omega}+\snorm{\psiv-\psiv_{\pi}}{1,\hh}\Big)
    \nonumber\\[0.25em]
    &\STACKON{\leq}{\eqref{eq:bound:psiv}}
    \Cs\Big(\snorm{\uvh-\uv}{1,\Omega}+\snorm{\uv-\uv_{\pi}}{1,\hh}\Big)\,\,
    \hh\snorm{\psiv}{2,\Omega}
    \STACKON{\leq}{\eqref{eq:regularity:bound:psiv}}
    \Cs\hh\Big(\snorm{\uvh-\uv}{1,\Omega} + \snorm{\uv-\uv_{\pi}}{1,\hh}\Big)\,\norm{\uv-\uvh}{0,\Omega}.
  \end{align}
  
  \medskip
  We derive an upper bound for term $\TERM{R}{3}$ by using the
  Cauchy-Schwarz inequality, and the
  inequalities~\eqref{eq:bound:phis}
  and~\eqref{eq:regularity:bound:psiv}:
  \begin{align}
    \ABS{\TERM{R}{3}}
    &= \ABS{\bs(\uvh-\uv,\phis-\phisI)}
    \leq \Cs\norm{\DIV(\uvh-\uv)}{0,\Omega}\,\norm{\phis-\phisI}{0,\Omega}
    \leq \Cs\snorm{\uvh-\uv}{1,\Omega}\,\norm{\phis-\phisI}{0,\Omega}
    \nonumber\\[0.5em]
    &
    \STACKON{\leq}{\eqref{eq:bound:phis}}            \Cs\snorm{\uvh-\uv}{1,\Omega}\,\hh\snorm{\phis}{1,\Omega}
    \STACKON{\leq}{\eqref{eq:regularity:bound:psiv}} \Cs\hh\snorm{\uvh-\uv}{1,\Omega}\,\norm{\uvh-\uv}{0,\Omega}.
  \end{align}
  
  \medskip
  Finally, we note that term $\TERM{R}{4}$ is zero by
  using~\eqref{eq:stokes:var:B} and~\eqref{eq:stokes:vem:B} (set
  $\qs=\qsh=\phisI$):
  \begin{align}
    \TERM{R}{4} =
    \bs(\uv-\uvh,\phisI) =
    \bs(\uv,\phisI) - \bsh(\uvh,\phisI) = 0.
  \end{align}
  
  The assertion of the theorem follows by using the bounds of terms
  $\TERM{R}{i}$, for $i=1,2,3$ and $\TERM{R}{4}=0$ to estimate the
  left-hand side of~\eqref{eq:theo:L2:estimates},
  Theorem~\ref{theorem:H1:estimate} to bound the resulting term
  $\snorm{\uvh-\uv}{1,\Omega}+\norm{\ps-\psh}{0,\Omega}$ and
  Lemma~\ref{lemma:projection:error} to bound
  $\snorm{\uv-\uv_{\pi}}{1,\hh}$.
\end{proof}


\section{Numerical experiments}
\label{sec:numerical}

\begin{figure}
  \centering
  \begin{tabular}{ccc}
    \hspace{-0.42cm}\includegraphics[scale=0.35]{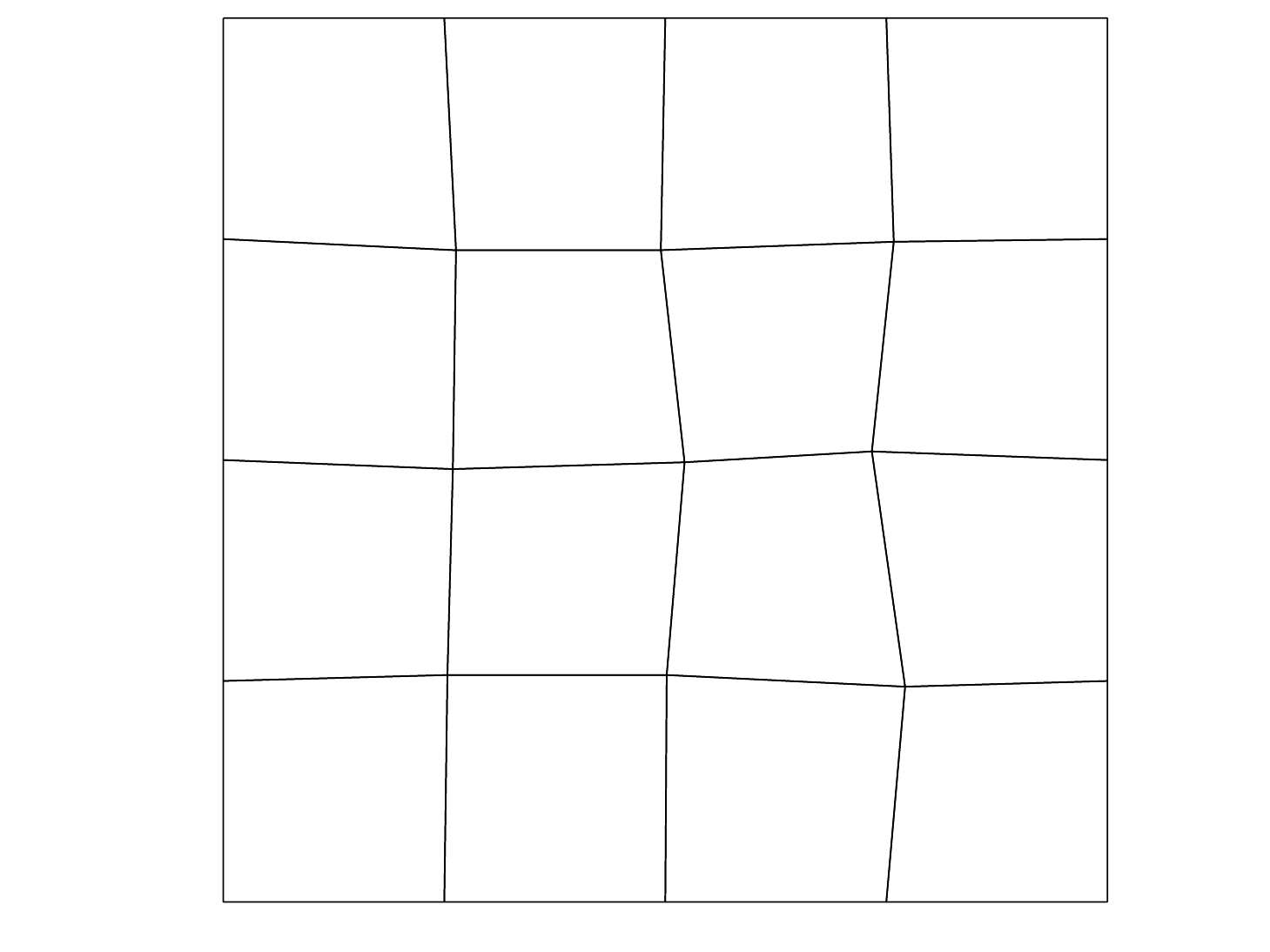} & 
    \hspace{-0.42cm}\includegraphics[scale=0.35]{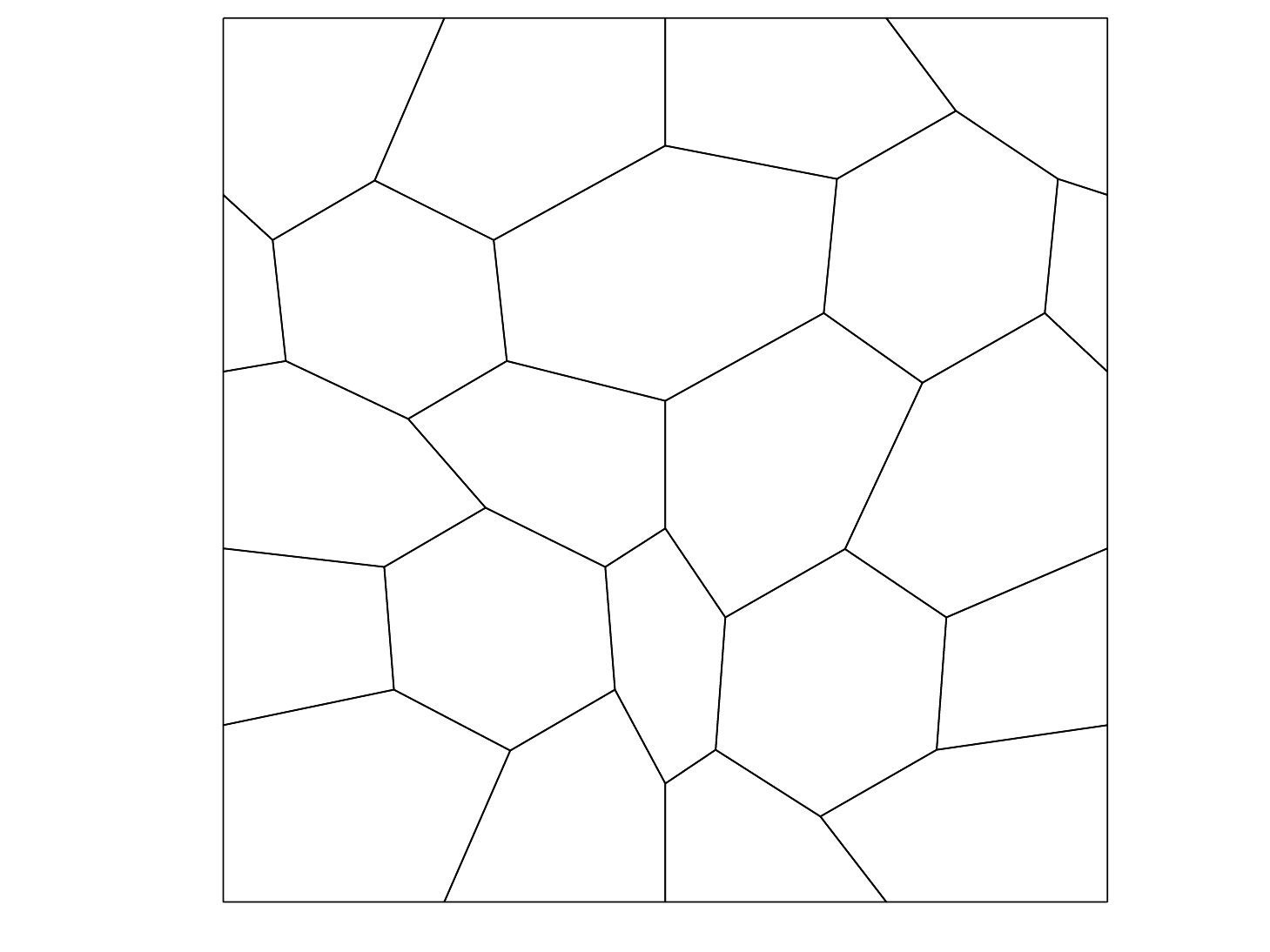} & 
    \hspace{-0.42cm}\includegraphics[scale=0.35]{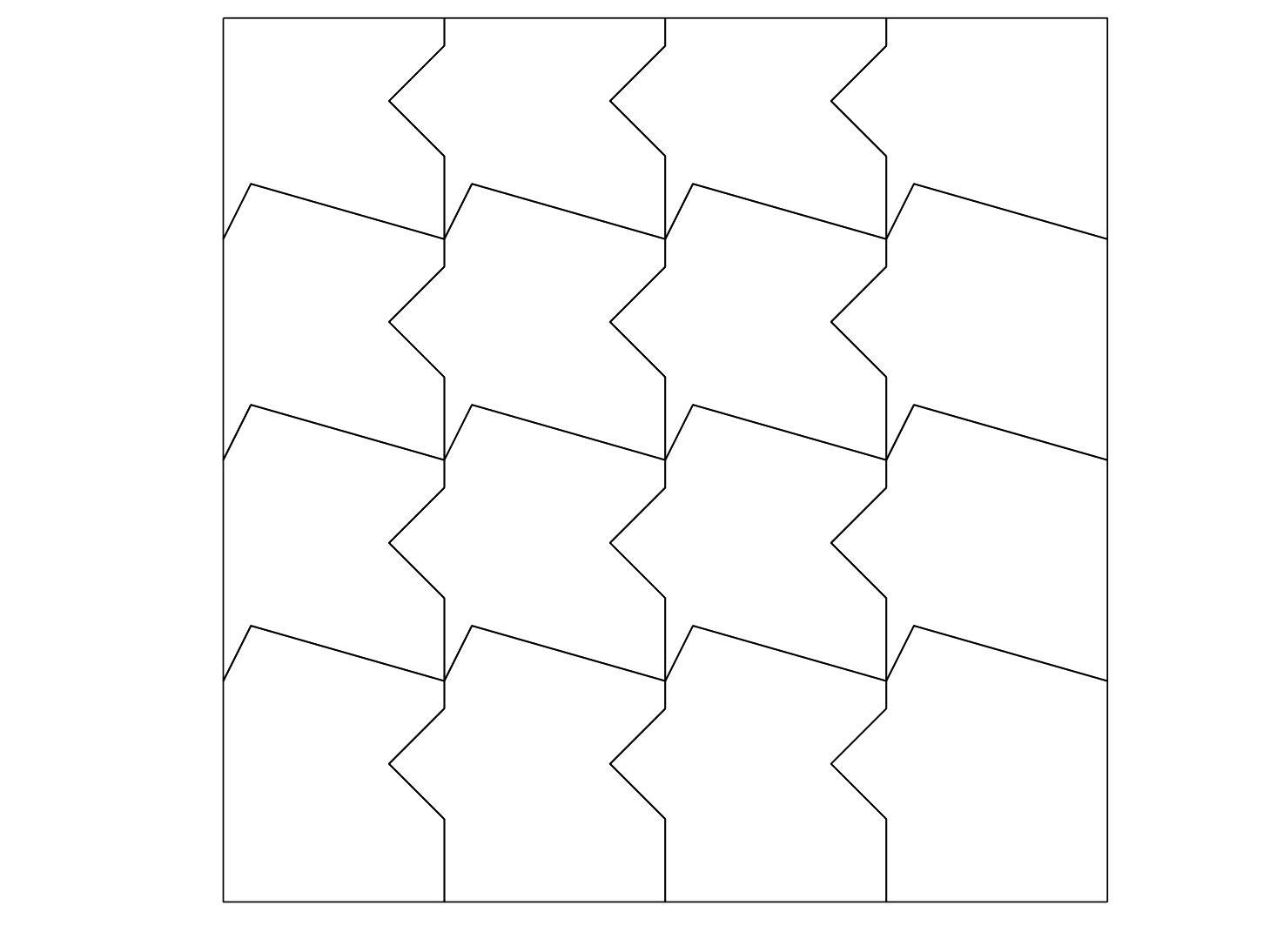}   
    \\[1em]
    \hspace{-0.42cm}\includegraphics[scale=0.35]{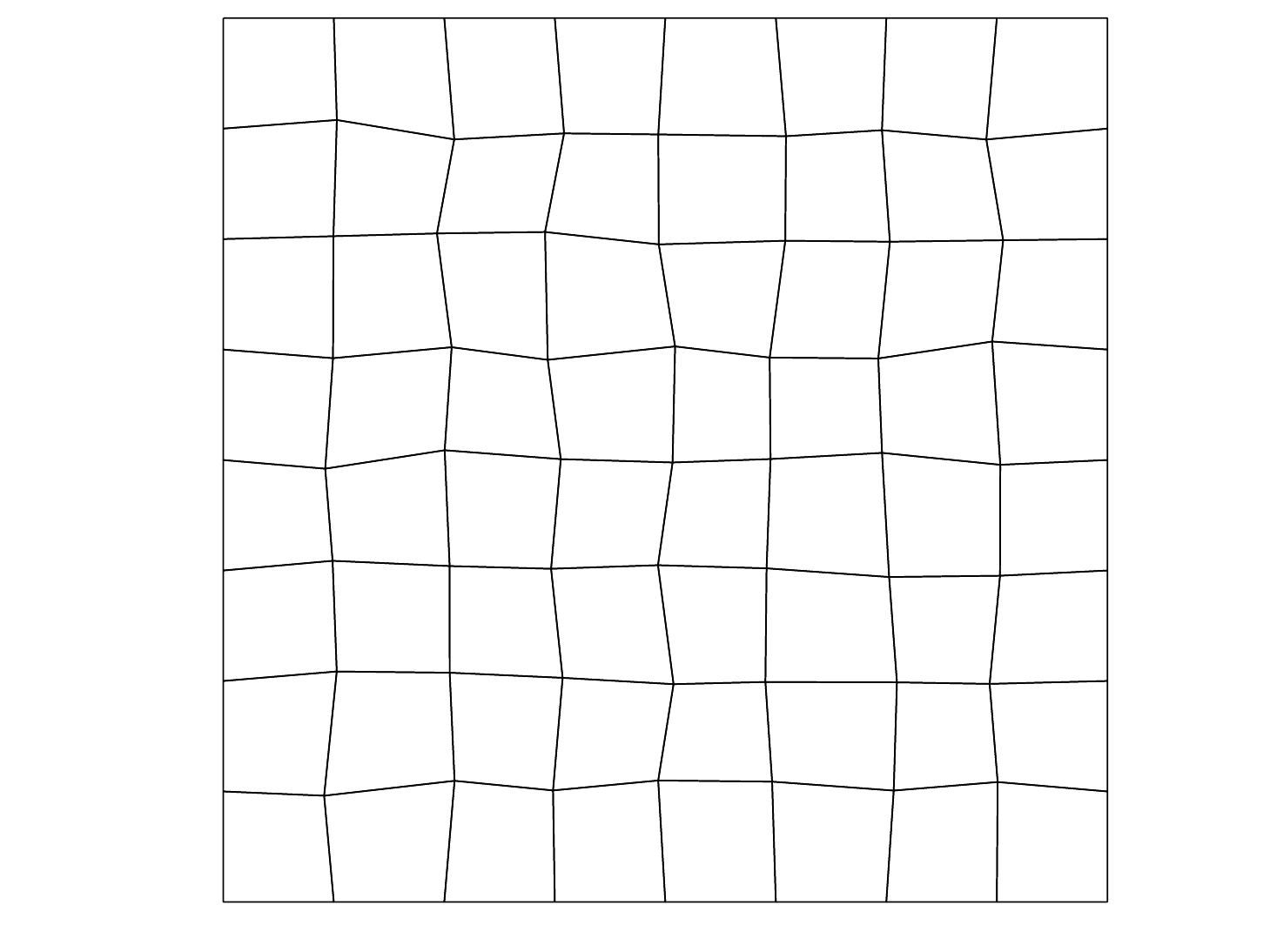} & 
    \hspace{-0.42cm}\includegraphics[scale=0.35]{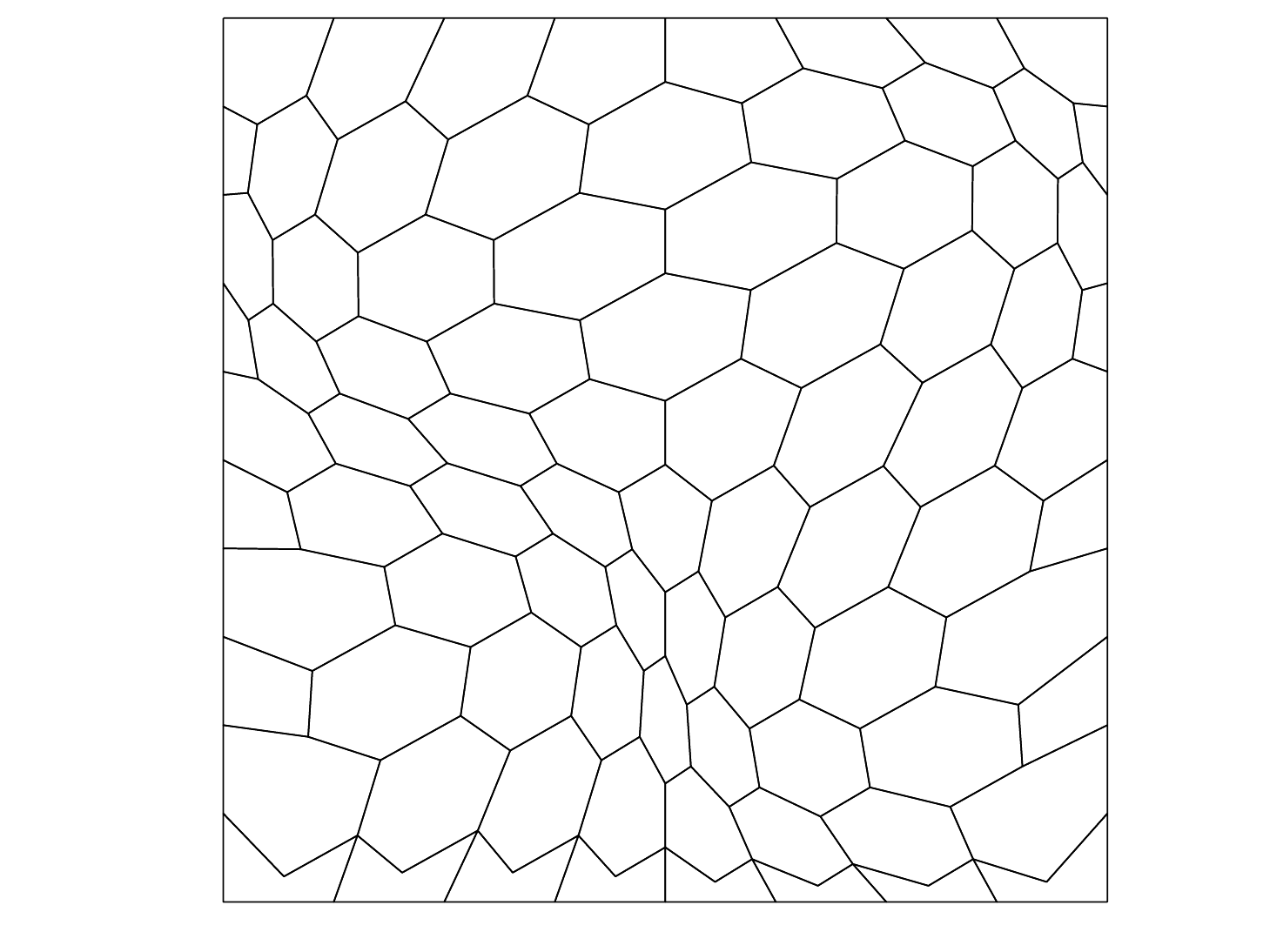} & 
    \hspace{-0.42cm}\includegraphics[scale=0.35]{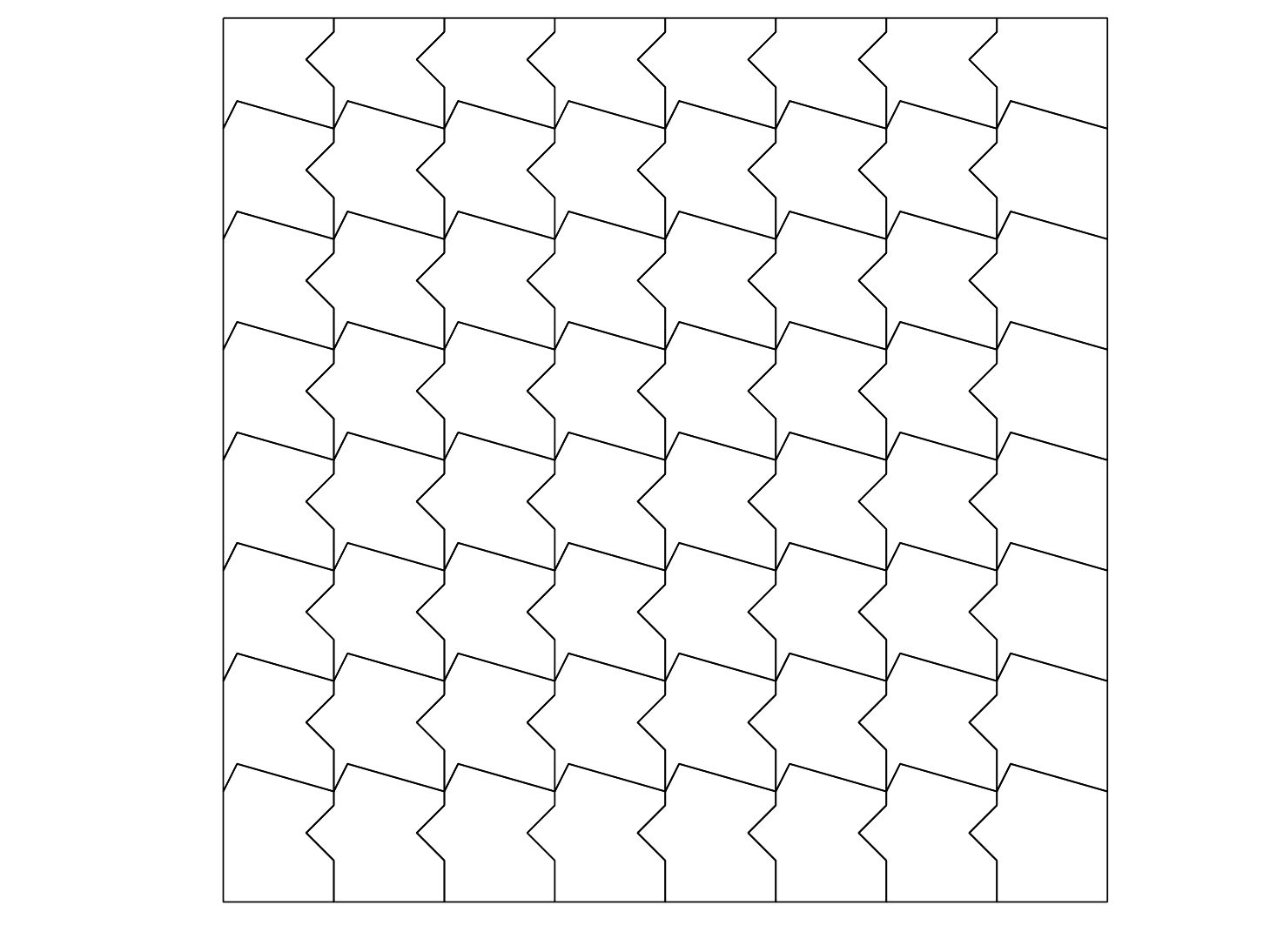}   
    \\[0.5em]        
    \hspace{-2mm} $\MESH{1}$ & \hspace{-2mm}$\MESH{2}$ & \hspace{-2mm}$\MESH{3}$
  \end{tabular}
  \caption{Base meshes (top row) and first refinement meshes (bottom
    row) of the three mesh families used in this section:
    $(\MESH{1})$ random quadrilateral meshes;
    $(\MESH{2})$ general polygonal meshes; 
    $(\MESH{3})$ concave element meshes;
    }
  \label{fig:Meshes}
\end{figure}

We assess the convergence property of the two virtual element
formulations considered in this paper by numerically solving problem
\eqref{eq:stokes:var:A}-\eqref{eq:stokes:var:B} on the computational
domain $\Omega=[0,1]\times[0,1]$.
The Dirichlet boundary conditions and the source term are set
accordingly to the manufactured solution $\uv=(\us_x,\us_y)^T$ and
$\ps$ given by
\begin{align*}
  \us_x(x,y) &=  \cos{(2\pi\xs)}\sin{(2\pi\ys)},\\
  \us_y(x,y) &= -\sin{(2\pi\xs)}\cos{(2\pi\ys)},\\
  \ps  (x,y) &= e^{\xs+\ys}-(e-1)^2.
\end{align*}
Our implementation of the virtual element method uses the basis of
orthogonal polynomials in all mesh elements, which is well-known to
control the ill-conditioning of the final linear system very
efficiently.

We run our virtual element solver on three mesh families respectively
composed by random quadrilateral meshes ($\MESH{1}$), general
polygonal meshes ($\MESH{2}$), and concave element meshes
($\MESH{3}$).
The construction of these mesh families is rather standard in the
literature of the VEM and its description can easily be found, for
example, in~\cite{Berrone-Borio-Manzini:2018:CMAME:journal}.
For every mesh family, we consider five refinements.
The base mesh and the first refined mesh of each family are shown in
Figure~\ref{fig:Meshes}; mesh data are reported in
Tables~\ref{tab:mesh-diameter} and~\ref{tab:mesh-elem-vrtx}.

\newcommand{\TABROW}[4]{ #1 & #2 & #3 & #4 \\}
\begin{table}[t!]
  \begin{center}
    \begin{tabular}{c|c|c|c}
      \TABROW{Level }{          $\MESH{1}$  }{          $\MESH{2}$ }{         $\MESH{3}$  }
      \hline
      \TABROW{    1 }{  $3.72 \cdot 10^{-1}$ }{ $4.26 \cdot 10^{-1}$ }{ $3.81 \cdot 10^{-1}$ }
      \TABROW{    2 }{  $1.99 \cdot 10^{-1}$ }{ $2.50 \cdot 10^{-1}$ }{ $1.91 \cdot 10^{-1}$ }
      \TABROW{    3 }{  $1.01 \cdot 10^{-1}$ }{ $1.25 \cdot 10^{-1}$ }{ $9.54 \cdot 10^{-2}$ }
      \TABROW{    4 }{  $5.17 \cdot 10^{-2}$ }{ $6.21 \cdot 10^{-2}$ }{ $4.77 \cdot 10^{-2}$ }
      \TABROW{    5 }{  $2.61 \cdot 10^{-2}$ }{ $3.41 \cdot 10^{-2}$ }{ $2.38 \cdot 10^{-2}$ } 
    \end{tabular}
    \caption{Diameter $h$ of meshes $\MESH{1}$, $\MESH{2}$, and
      $\MESH{3}$. }
      \label{tab:mesh-diameter}
  \end{center}
\end{table}

\renewcommand{\TABROW}[7]{ #1 & #2 & #3 & #4 & #5 & #6 & #7 \\}
\begin{table}[t!]
  \begin{center}
    \begin{tabular}{c|cc|cc|cc}
      Level & \multicolumn{2}{c|}{ $\MESH{1}$ } & \multicolumn{2}{c|}{ $\MESH{2}$ } & \multicolumn{2}{c}{ $\MESH{3}$ }\\ \hline
      \TABROW{     }{   $N_{el}$ }{  $N$ }{ $N_{el}$ }{   $N$ }{  $N_{el}$ }{   $N$ }
      \TABROW{  1  }{        16 }{   25 }{      22 }{    46 }{       16 }{    73 }
      \TABROW{  2  }{        64 }{   81 }{      84 }{   171 }{       64 }{   305 } 
      \TABROW{  3  }{       256 }{  289 }{     312 }{   628 }{      256 }{  1249 }
      \TABROW{  4  }{      1024 }{ 1089 }{    1202 }{  2406 }{     1024 }{  5057 }
      \TABROW{  5  }{      4096 }{ 4225 }{    4772 }{  9547 }{     4096 }{ 20353 }
    \end{tabular}
    \caption{Number of elements $N_{el}$ and vertices $N$ of meshes
      $\MESH{1}$, $\MESH{2}$, and $\MESH{3}$. }
      \label{tab:mesh-elem-vrtx}         
  \end{center}
\end{table}

On any set of refined meshes, we measure the $\HONE$ relative error
for the velocity vector field by applying the formula
\begin{align}
  \text{error}_{\HONE(\Omega)}(\uv) = \frac{\snorm{\uv-\Piz{k}\uvh}{1,\hh}}{\snorm{\uv}{1,\Omega}}
  \approx \dfrac{\snorm{\uv-\uvh}{1,\Omega}}{\snorm{\uv}{1,\Omega}},
  \label{eq:error:H1:velocity}
\end{align}
and the $\LTWO$ relative error by applying the formula
\begin{align}
  \text{error}_{\LTWO(\Omega)}(\uv)
  = \dfrac{\norm{ \uv-\Piz{k}\uvh}{0,\Omega}}{\norm{\uv}{0,\Omega}}
  \approx \dfrac{ \norm{\uv-\uvh}{0,\Omega} }{\norm{\uv}{0,\Omega}}.
  \label{eq:error:L2:velocity}
\end{align}
For the pressure scalar field we measure the $\LTWO(\Omega)$ relative
error by applying the formula
\begin{align}
  \text{error}_{\LTWO(\Omega)}(\ps)
  = \dfrac{\norm{\ps-\psh}{0,\Omega}}{\norm{\ps}{0,\Omega}}.
  \label{eq:error:L2:pressure}
\end{align}
In our implementations, the use of the enhancement spaces only changes
the calculation of the right-hand side of Eq.~\eqref{eq:stokes:vem:A}.
In fact, in the implementations of $\FO$ and $\FT$ using the
non-enhanced space definitions, we approximate the right-hand side
through the projection operator $\Piz{\kb}$ with $\kb=max(0,k-2)$,
while in the ones using the enhanced space definitions, we approximate
the right-hand side through the projection operator $\Piz{\ks}$.
However, since the nonenhanced and the enhanced versions have the same
degrees of freedom, we can always compute the projection operator
$\Piz{\ks}$, and use it to evaluate the approximation error as
in~\eqref{eq:error:H1:velocity} and~\eqref{eq:error:L2:velocity}
above.
In the non-enhanced case, this is equivalent to a sort of
post-processing of $\uvh$, which is known only through its degrees of
freedom, to derive a polynomial approximation of $\uv$ that is defined
on the whole computational domain.

\begin{figure}
  \includegraphics[width=\textwidth,clip=]{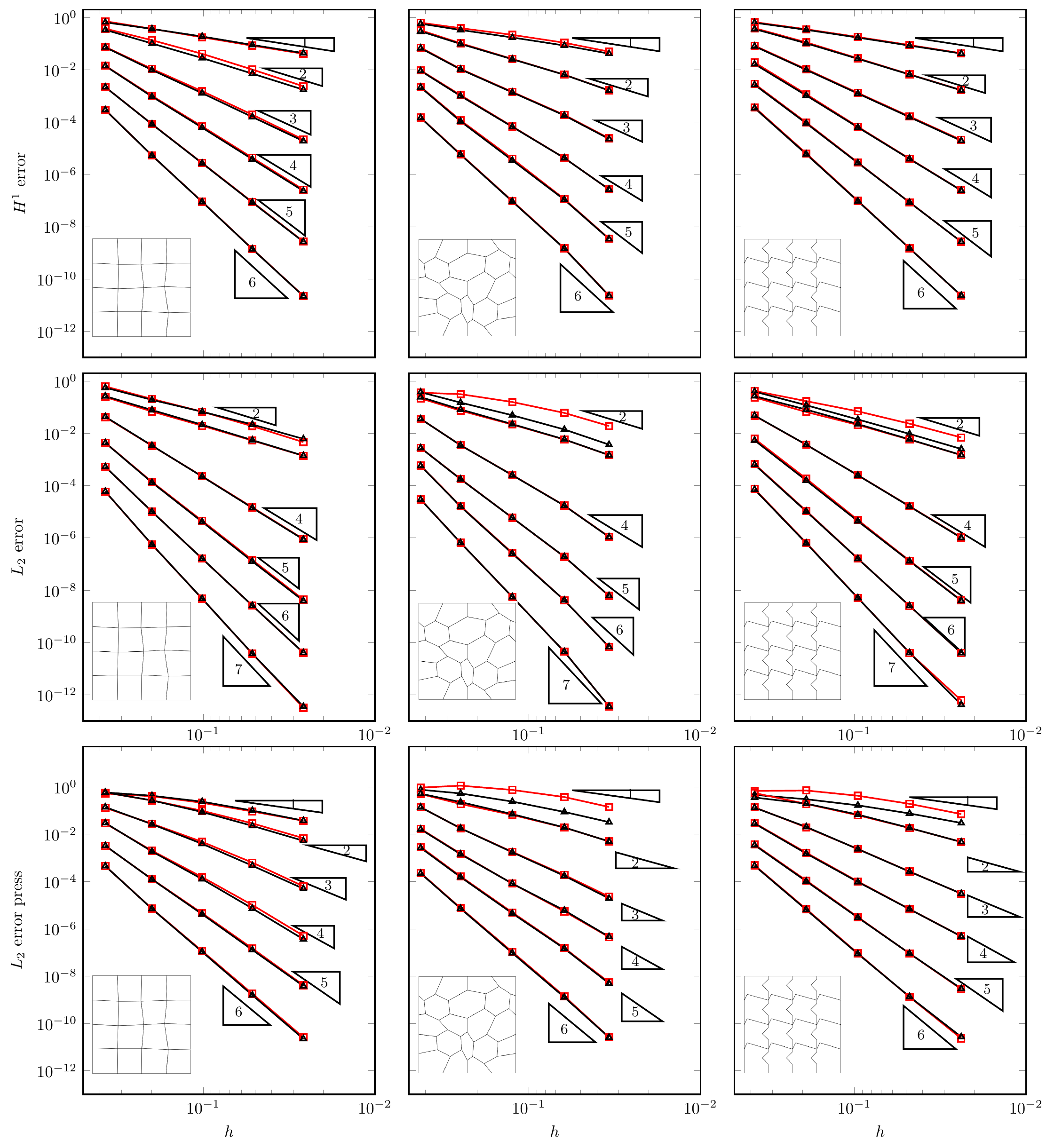} 
  \caption{ Error curves versus $\hh$ for the velocity approximation
    using the energy norm~\eqref{eq:error:H1:velocity} (top panels)
    and the $\LTWO$-norm~\eqref{eq:error:L2:velocity} (mid panels),
    and for the pressure approximation using the
    $\LTWO$-norm~\eqref{eq:error:L2:pressure} (bottom panels).
    Solid (red) lines with square markers show the errors for the
    first formulation using space~\eqref{eq:FO:regular-space:def};
    solid (black) lines with triangular markers show the errors for
    the second formulation using
    space~\eqref{eq:FT:regular-space:def}.
    The right-hand side is approximated by using the projection
    operator $\Piz{\kb}$ with $\kb=max(0,k-2)$. 
    The mesh families used in each calculations are shown in the left
    corner of each panel and the expected convergence rates are
    reflected by the slopes of the triangles and corresponding numeric
    labels.}
  \label{fig:h_errorPi0km2}
\end{figure}

\begin{figure}
  \includegraphics[width=\textwidth,clip=]{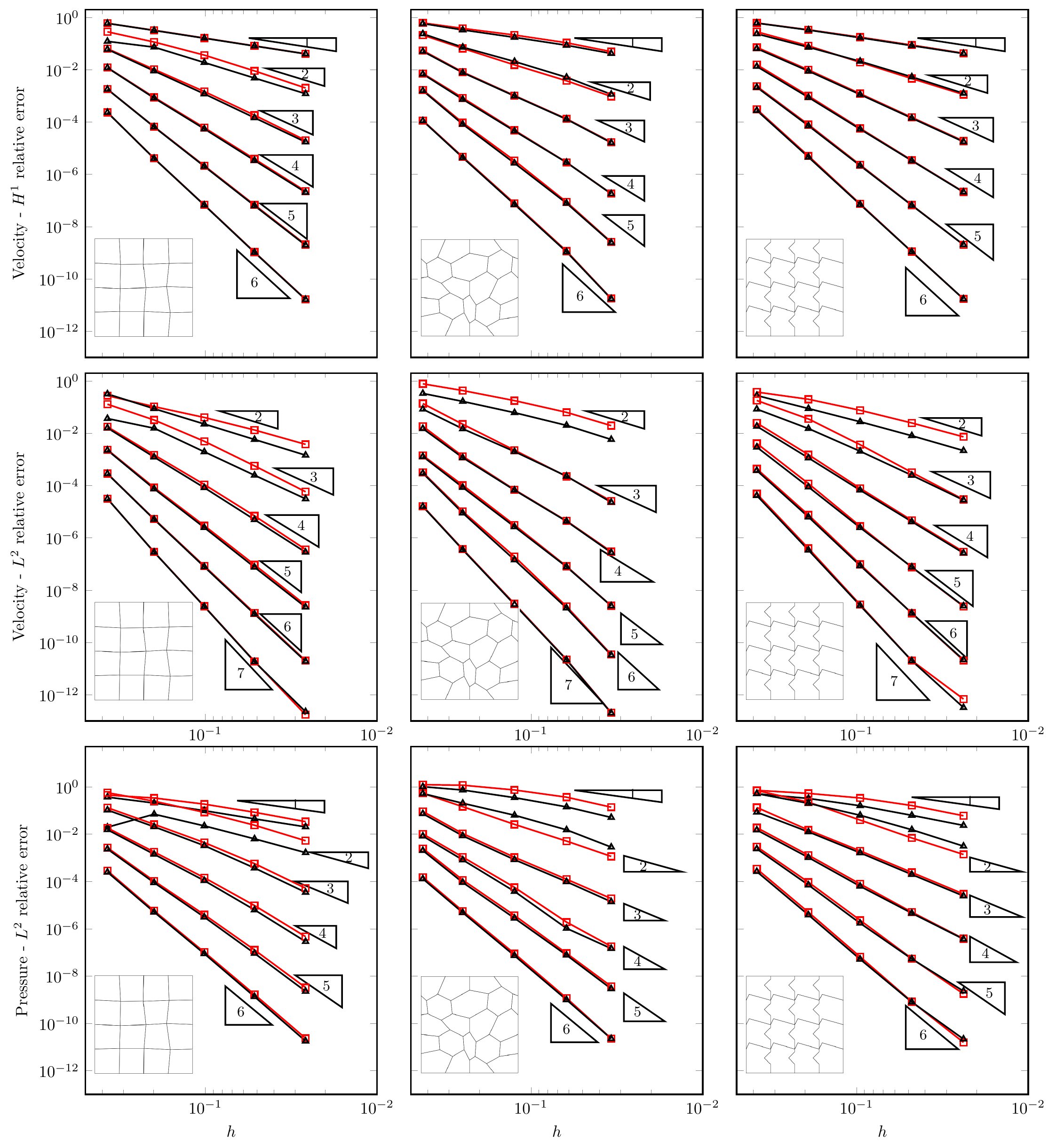} 
  \caption{ Error curves versus $\hh$ for the velocity
    approximation using the energy norm~\eqref{eq:error:H1:velocity}
    (top panels) and the $\LTWO$-norm~\eqref{eq:error:L2:velocity}
    (mid panels), and for the pressure approximation using the
    $\LTWO$-norm~\eqref{eq:error:L2:pressure} (bottom panels).
    Solid (red) lines with square markers show the errors for the first
    formulation using space~\eqref{eq:FO:regular-space:def};
    solid (black) lines with triangular markers show the errors for
    the second formulation using
    space~\eqref{eq:FT:regular-space:def}.
    The right-hand side is approximated by using the projection
    operator $\Piz{k}$.
    The mesh families used in each calculations are shown in the left
    corner of each panel and the expected convergence rates are
    reflected by the slopes of the triangles and corresponding numeric
    labels.}
  \label{fig:h_errorPi0k}
\end{figure}

\ifARXIV

\begin{figure}
  \includegraphics[width=\textwidth,clip=]{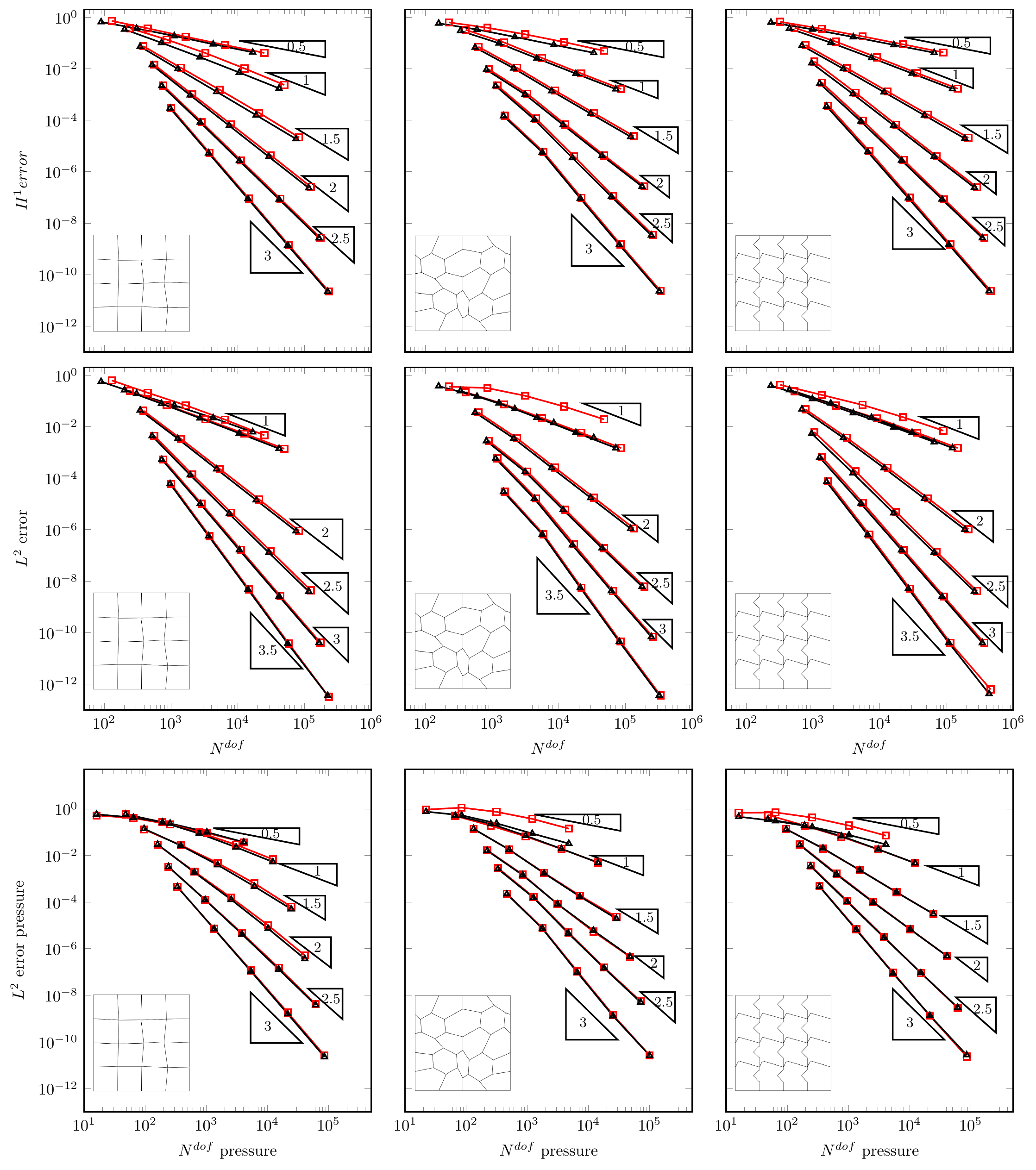} 
  \caption{ Error curves versus $N^{dof}$ for the velocity
    approximation using the energy norm~\eqref{eq:error:H1:velocity}
    (top panels) and the $\LTWO$-norm~\eqref{eq:error:L2:velocity}
    (mid panels), and for the pressure approximation using the
    $\LTWO$-norm~\eqref{eq:error:L2:pressure} (bottom panels).
    Solid (red) lines with square markers shows the errors for the
    first formulation using space~\eqref{eq:FO:regular-space:def};
    solid (black) lines with triangular markers shows the errors for
    the second formulation using
    space~\eqref{eq:FT:regular-space:def}.
    The right-hand side is approximated by using the projection
    operator $\Piz{\kb}$ with $\kb=max(0,k-2)$. 
    \RED{\textbf{????}}
    The mesh families used in each calculations are shown in the left
    corner of each panel and the expected convergence rates are
    reflected by the slopes of the triangles and corresponding numeric
    labels.}
  \label{fig:node_errorPi0km2}
\end{figure}

\begin{figure}
  \includegraphics[width=\textwidth,clip=]{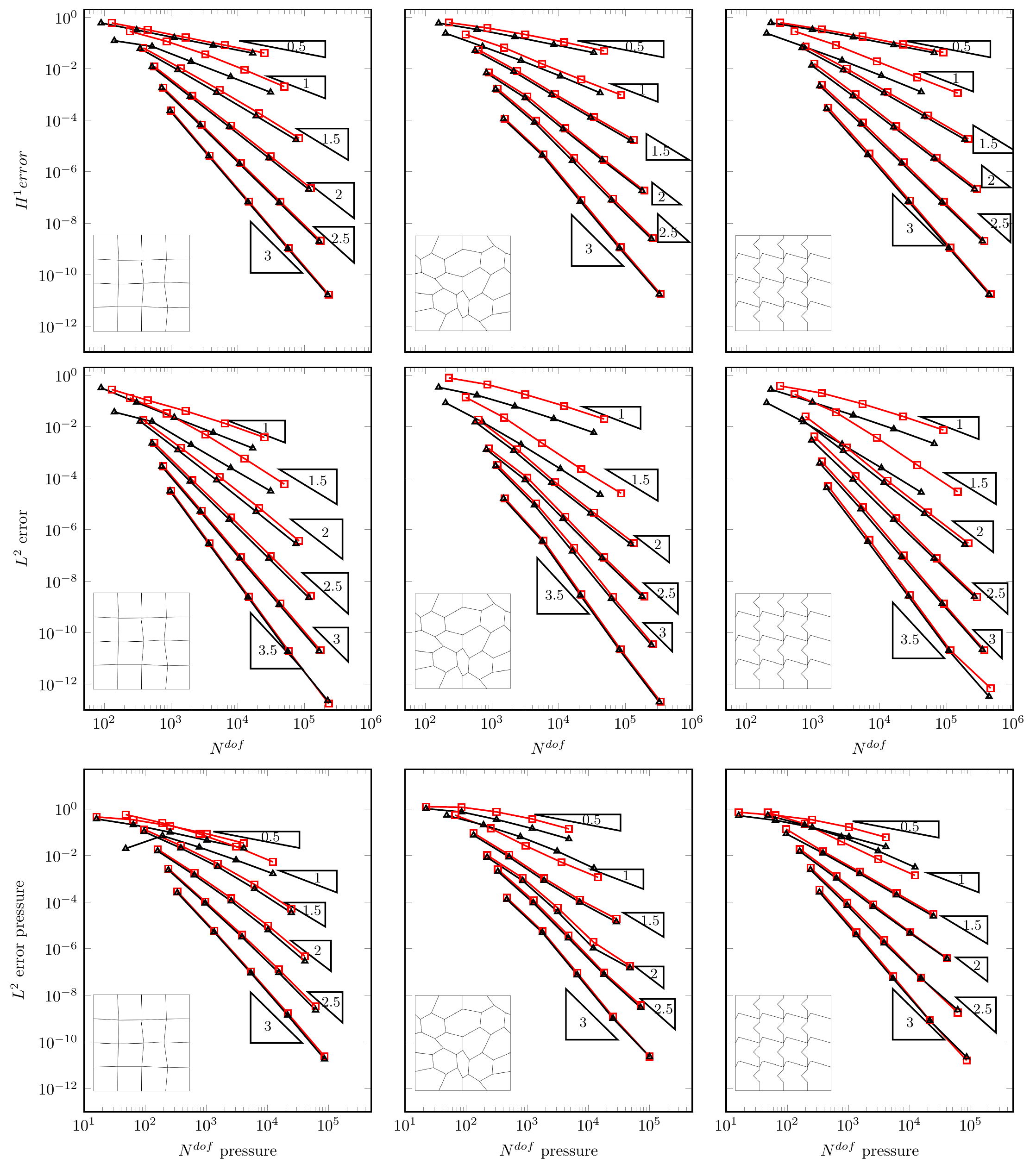} 
  \caption{ Error curves versus $N^{dof}$ for the velocity
    approximation using the energy norm~\eqref{eq:error:H1:velocity}
    (top panels) and the $\LTWO$-norm~\eqref{eq:error:L2:velocity}
    (mid panels), and for the pressure approximation using the
    $\LTWO$-norm~\eqref{eq:error:L2:pressure} (bottom panels).
    Solid (red) lines with square markers shows the errors for the
    first formulation using space~\eqref{eq:FO:regular-space:def};
    solid (black) lines with triangular markers shows the errors for
    the second formulation using
    space~\eqref{eq:FT:regular-space:def}.
    The right-hand side is approximated by using the projection
    operator $\Piz{k}$.
    The mesh families used in each calculations are shown in the left
    corner of each panel and the expected convergence rates are
    reflected by the slopes of the triangles and corresponding numeric
    labels.}
  \label{fig:node_errorPi0k}
\end{figure}

\fi

\begin{figure}
  \includegraphics[width=\textwidth,clip=]{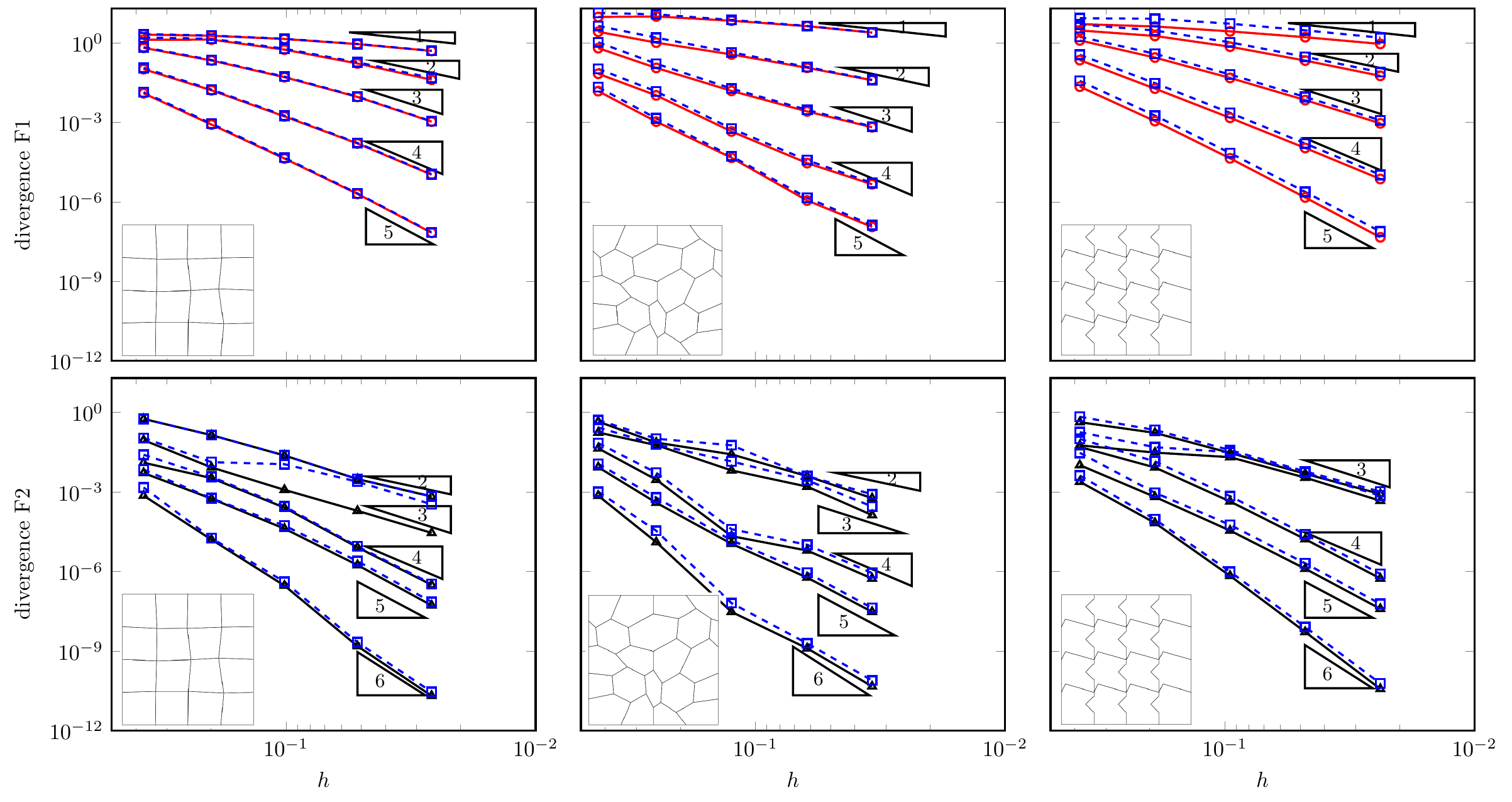} 
 \caption{ $\LTWO$-norm of the divergence of the velocity field using
   the non-enhanced virtual element
   space~\eqref{eq:FO:regular-space:def} (top panels) and the enhanced
   virtual element space space~\eqref{eq:FT:regular-space:def} (bottom
   panels).
   The right-hand side~\eqref{eq:fvh:def} is approximated by using the
   projection operator $\Piz{k}$.
   Solid (red and black) lines with square markers refer to
   $\PizP{k}(\DIV\uvh)$; dotted (blue) lines with circle markers refer
   to $\PizP{k+1}(\DIV\uvh)$.
   The mesh families used in each calculations are shown in the left
   corner of each panel and the expected convergence rates are
   reflected by the slopes of the triangles and corresponding numeric
   labels.}
  \label{fig:divergence}
\end{figure}
\begin{figure}
  \includegraphics[width=\textwidth,clip=]{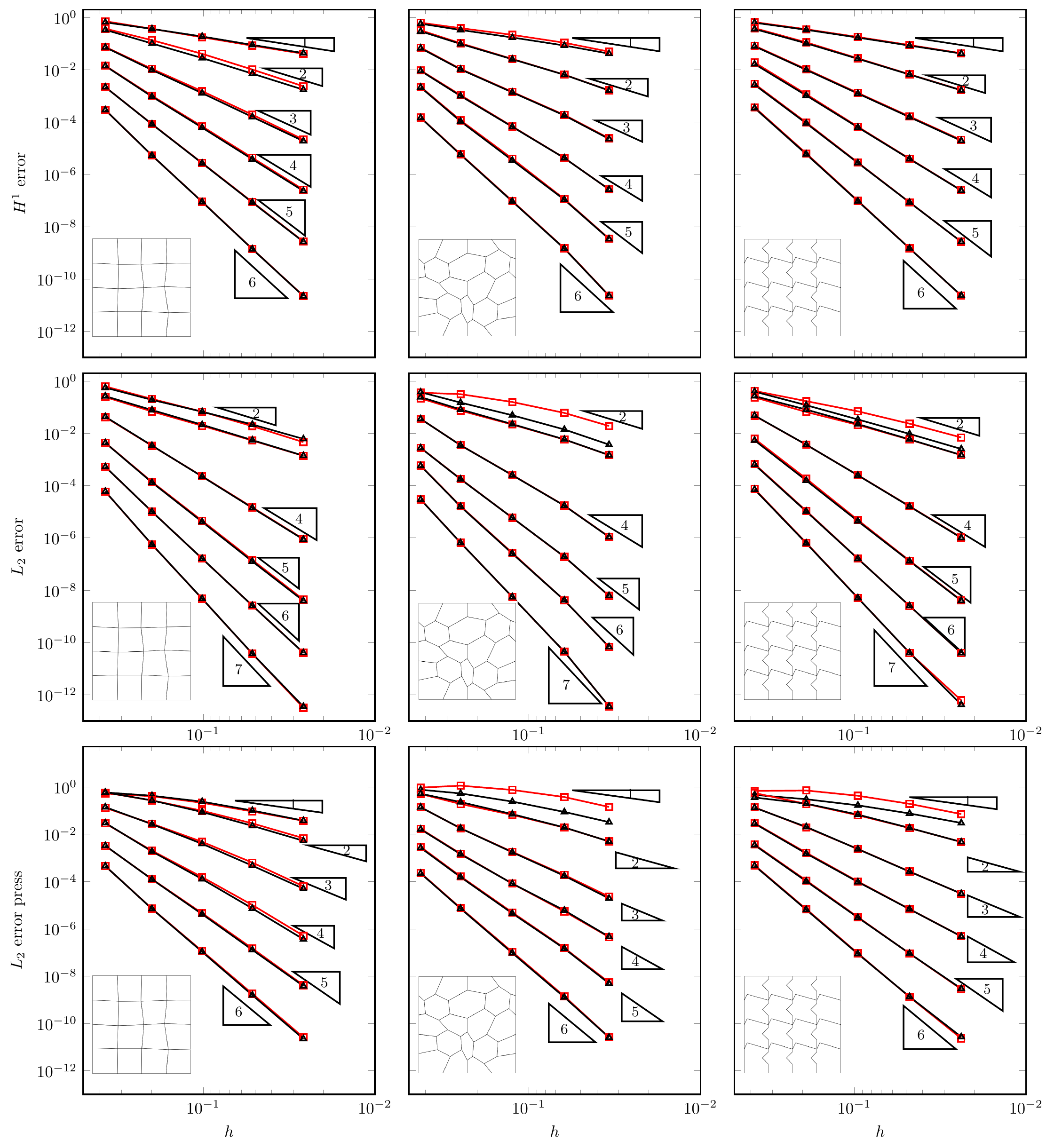} 
  \caption{ Error curves versus $\hh$ for the velocity approximation
    using the energy norm~\eqref{eq:error:H1:velocity} (top panels)
    and the $\LTWO$-norm~\eqref{eq:error:L2:velocity} (mid panels),
    and for the pressure approximation using the
    $\LTWO$-norm~\eqref{eq:error:L2:pressure} (bottom panels).
    Solid (red) lines with square markers show the errors for the
    first formulation using space~\eqref{eq:FO:regular-space:def};
    solid (black) lines with triangular markers show the errors for
    the second formulation using
    space~\eqref{eq:FT:regular-space:def}.
    The right-hand side is approximated by using the projection
    operator $\Piz{\kb}$ with $\kb=max(0,k-2)$. 
    The mesh families used in each calculations are shown in the left
    corner of each panel and the expected convergence rates are
    reflected by the slopes of the triangles and corresponding numeric
    labels.}
  \label{fig:h_errorPi0km2}
\end{figure}

\begin{figure}
  \includegraphics[width=\textwidth,clip=]{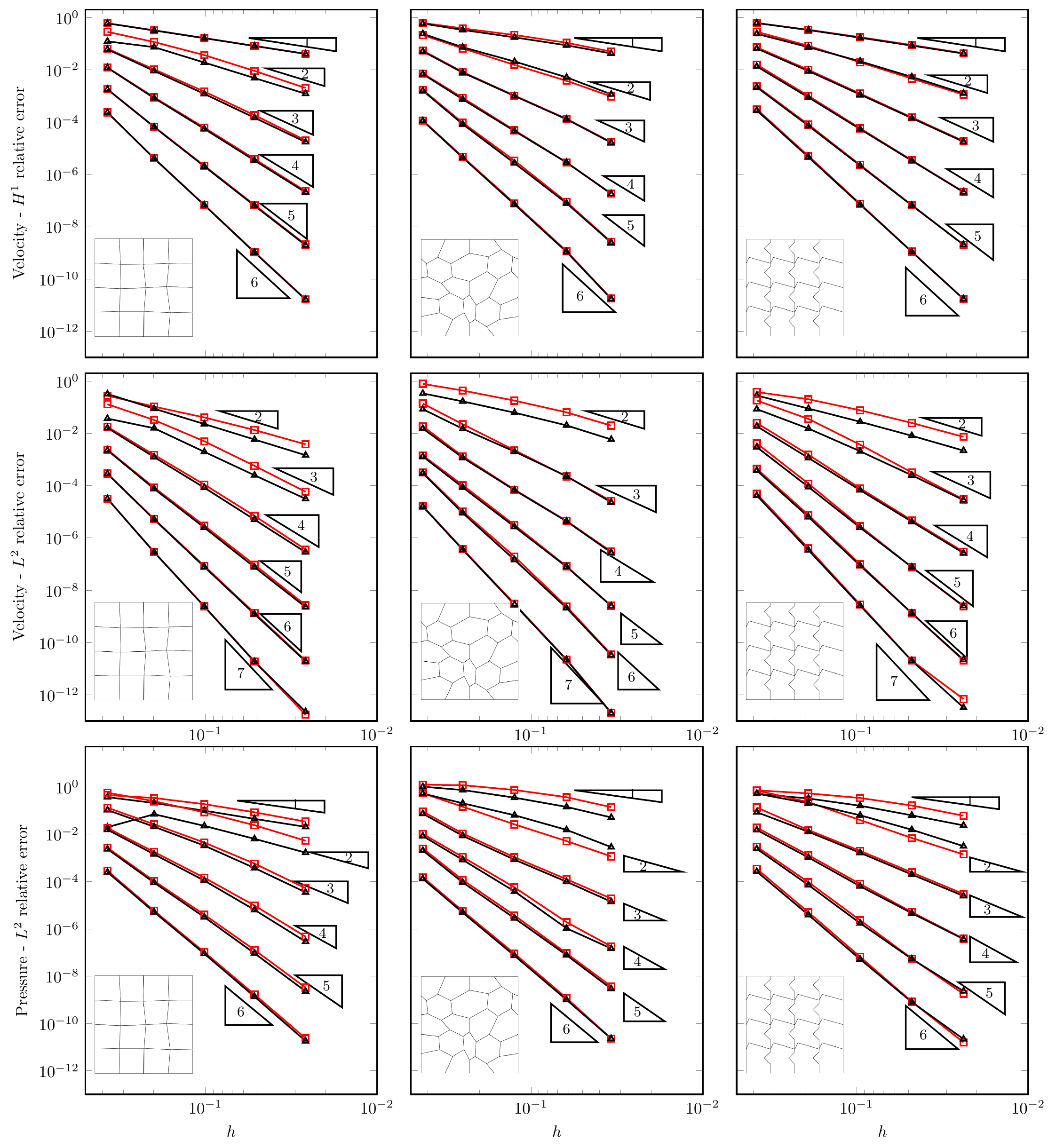} 
  \caption{ Error curves versus $\hh$ for the velocity
    approximation using the energy norm~\eqref{eq:error:H1:velocity}
    (top panels) and the $\LTWO$-norm~\eqref{eq:error:L2:velocity}
    (mid panels), and for the pressure approximation using the
    $\LTWO$-norm~\eqref{eq:error:L2:pressure} (bottom panels).
    Solid (red) lines with square markers show the errors for the first
    formulation using space~\eqref{eq:FO:regular-space:def};
    solid (black) lines with triangular markers show the errors for
    the second formulation using
    space~\eqref{eq:FT:regular-space:def}.
    The right-hand side is approximated by using the projection
    operator $\Piz{k}$.
    The mesh families used in each calculations are shown in the left
    corner of each panel and the expected convergence rates are
    reflected by the slopes of the triangles and corresponding numeric
    labels.}
  \label{fig:h_errorPi0k}
\end{figure}

\ifARXIV

\begin{figure}
  \includegraphics[width=\textwidth,clip=]{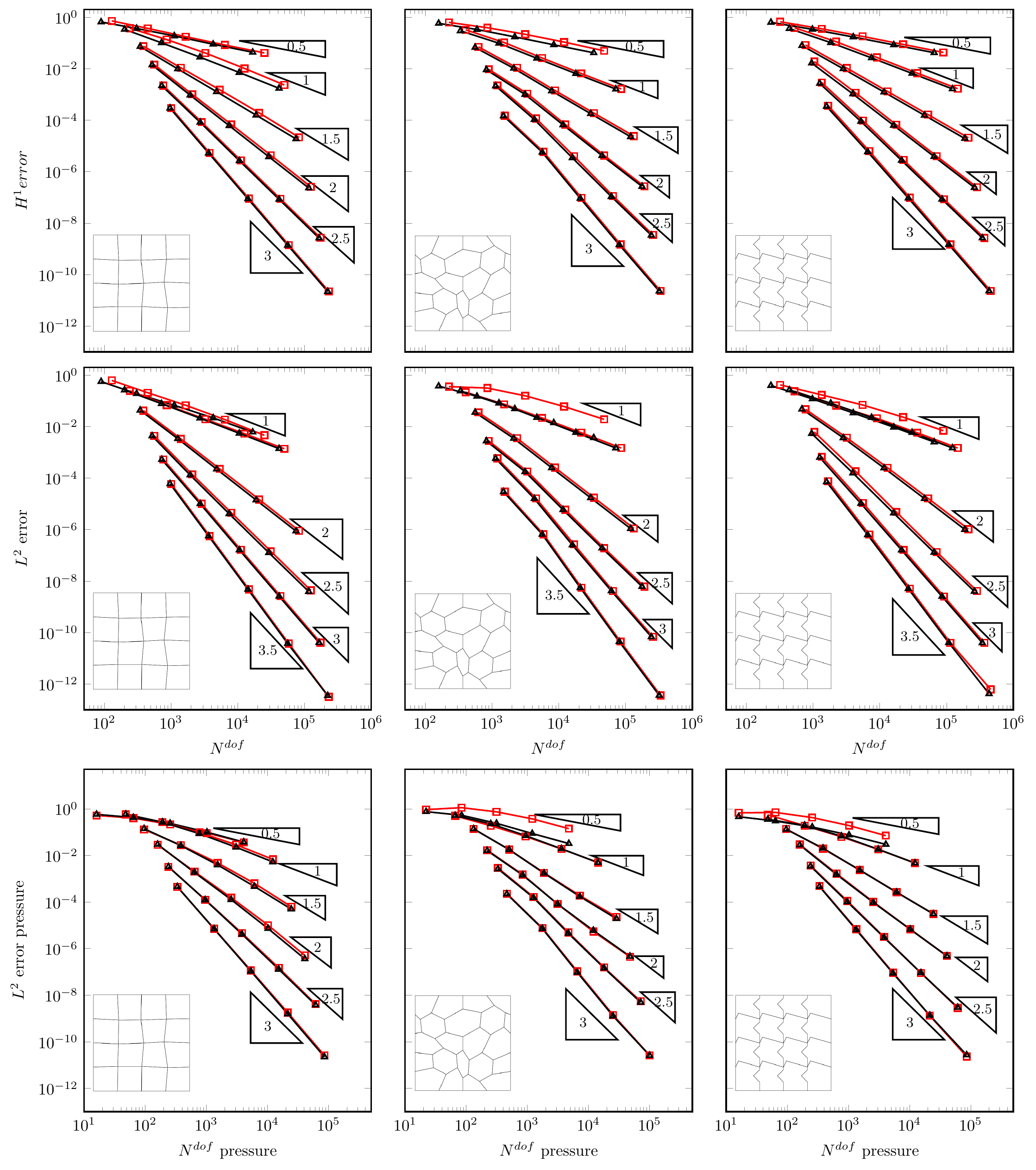} 
  \caption{ Error curves versus $N^{dof}$ for the velocity
    approximation using the energy norm~\eqref{eq:error:H1:velocity}
    (top panels) and the $\LTWO$-norm~\eqref{eq:error:L2:velocity}
    (mid panels), and for the pressure approximation using the
    $\LTWO$-norm~\eqref{eq:error:L2:pressure} (bottom panels).
    Solid (red) lines with square markers shows the errors for the
    first formulation using space~\eqref{eq:FO:regular-space:def};
    solid (black) lines with triangular markers shows the errors for
    the second formulation using
    space~\eqref{eq:FT:regular-space:def}.
    The right-hand side is approximated by using the projection
    operator $\Piz{\kb}$ with $\kb=max(0,k-2)$. 
    \RED{\textbf{????}}
    The mesh families used in each calculations are shown in the left
    corner of each panel and the expected convergence rates are
    reflected by the slopes of the triangles and corresponding numeric
    labels.}
  \label{fig:node_errorPi0km2}
\end{figure}

\begin{figure}
  \includegraphics[width=\textwidth,clip=]{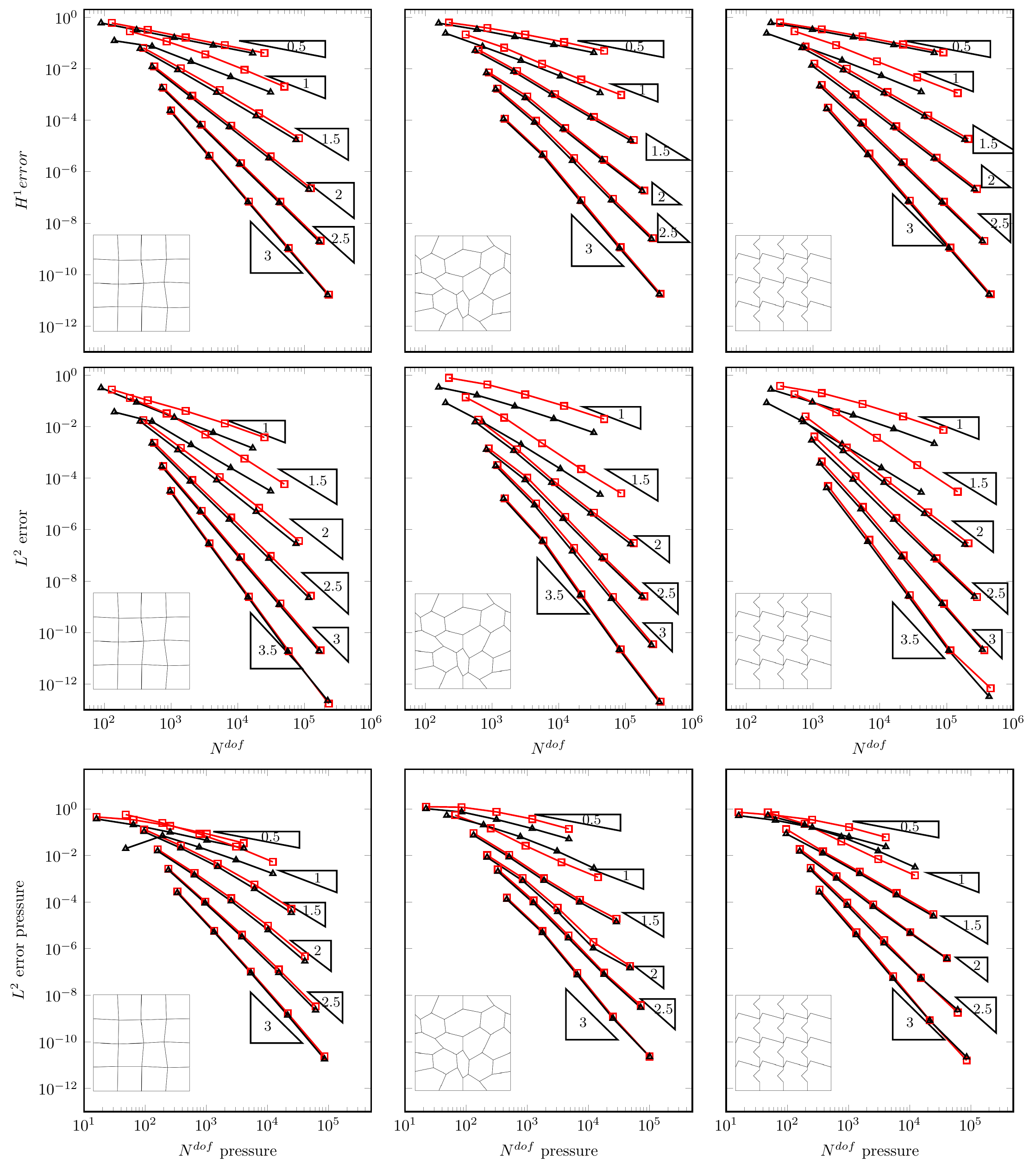} 
  \caption{ Error curves versus $N^{dof}$ for the velocity
    approximation using the energy norm~\eqref{eq:error:H1:velocity}
    (top panels) and the $\LTWO$-norm~\eqref{eq:error:L2:velocity}
    (mid panels), and for the pressure approximation using the
    $\LTWO$-norm~\eqref{eq:error:L2:pressure} (bottom panels).
    Solid (red) lines with square markers shows the errors for the
    first formulation using space~\eqref{eq:FO:regular-space:def};
    solid (black) lines with triangular markers shows the errors for
    the second formulation using
    space~\eqref{eq:FT:regular-space:def}.
    The right-hand side is approximated by using the projection
    operator $\Piz{k}$.
    The mesh families used in each calculations are shown in the left
    corner of each panel and the expected convergence rates are
    reflected by the slopes of the triangles and corresponding numeric
    labels.}
  \label{fig:node_errorPi0k}
\end{figure}

\fi

\begin{figure}
  \includegraphics[width=\textwidth,clip=]{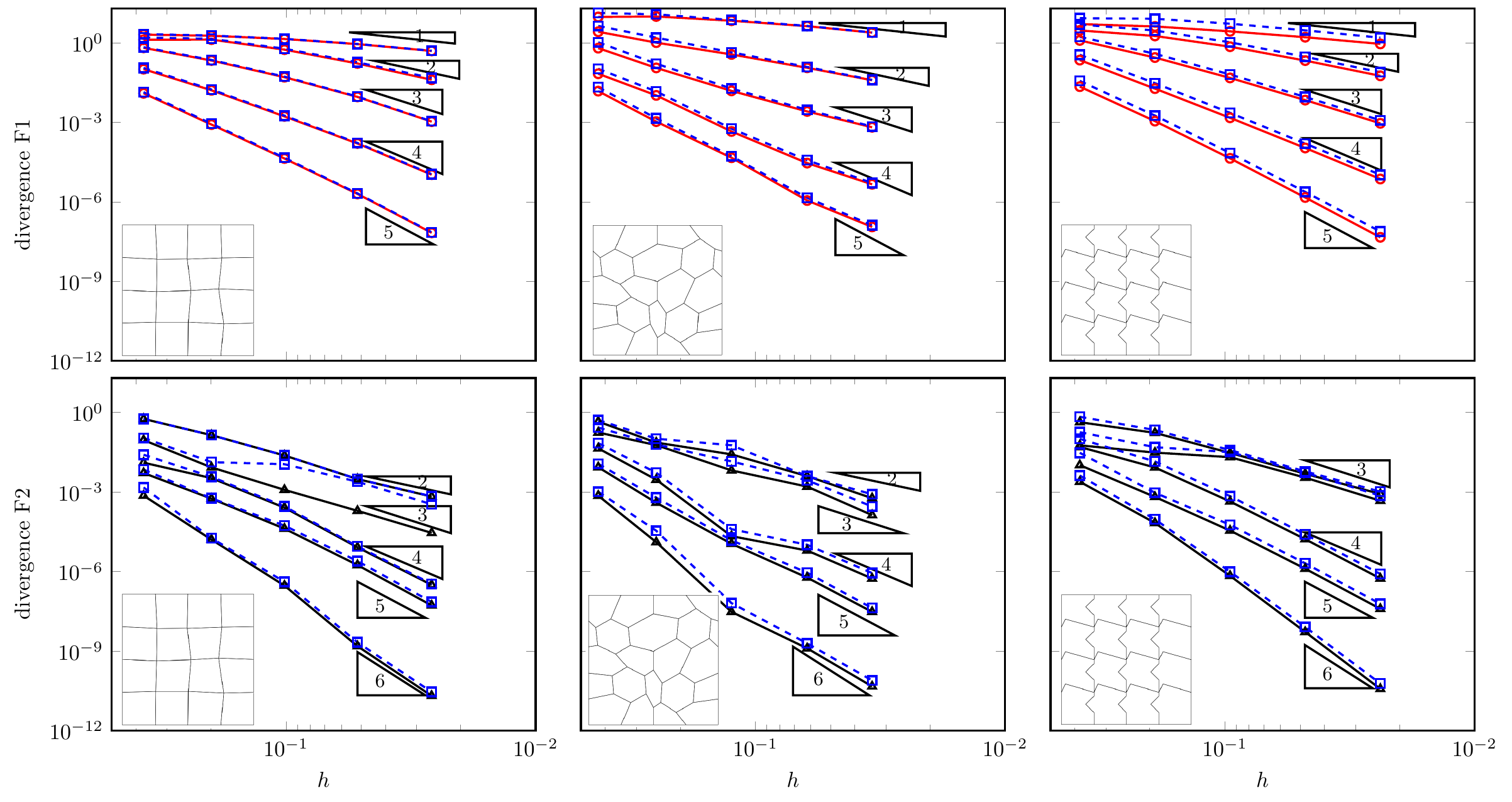} 
 \caption{ $\LTWO$-norm of the divergence of the velocity field using
   the non-enhanced virtual element
   space~\eqref{eq:FO:regular-space:def} (top panels) and the enhanced
   virtual element space space~\eqref{eq:FT:regular-space:def} (bottom
   panels).
   The right-hand side~\eqref{eq:fvh:def} is approximated by using the
   projection operator $\Piz{k}$.
   Solid (red and black) lines with square markers refer to
   $\PizP{k}(\DIV\uvh)$; dotted (blue) lines with circle markers refer
   to $\PizP{k+1}(\DIV\uvh)$.
   The mesh families used in each calculations are shown in the left
   corner of each panel and the expected convergence rates are
   reflected by the slopes of the triangles and corresponding numeric
   labels.}
  \label{fig:divergence}
\end{figure}

\PGRAPH{Convergence results}
\ifARXIV
In Figures~\ref{fig:h_errorPi0km2}, \ref{fig:h_errorPi0k},
\ref{fig:node_errorPi0km2}, and~\ref{fig:node_errorPi0k},
\else
In Figures~\ref{fig:h_errorPi0km2} and~\ref{fig:h_errorPi0k},
\fi
we compare the approximation errors \eqref{eq:error:H1:velocity},
\eqref{eq:error:L2:velocity}, and \eqref{eq:error:L2:pressure} that
are obtained when using the \emph{non-enhanced} and the
\emph{enhanced} definitions of the virtual element space for the
velocity approximation.
In particular, we recall that formulation~$\FO$ uses the space
definitions~\eqref{eq:FO:regular-space:def} (non-enhanced)
and~\eqref{eq:FO:enhanced-space:def} (enhanced); formulation~$\FT$
uses the space definitions~\eqref{eq:FT:regular-space:def}
(non-enhanced) and~\eqref{eq:FT:enhanced-space:def} (enhanced).
All error curves in Figures~\ref{fig:h_errorPi0km2}
and~\ref{fig:h_errorPi0k}, for $k=1,\ldots,6$ are shown in a log-log
plot versus the mesh size parameter $\hh$.
\ifARXIV
All error curves in Figures~\ref{fig:node_errorPi0km2}
and~\ref{fig:node_errorPi0k}, for $k=1,\ldots,6$ are shown in a
log-log plot versus the total number of degrees of freedom $N^{\footnotesize{\mbox{dof}}}$.
\fi
Solid (red) lines with square markers show the errors for the
formulation $\FO$; solid (black) lines with triangular markers show
the errors for the formulation $\FT$.
The mesh family is shown in the bottom-left corner and the slopes of
the error curves reflect the numerical order of convergence of each
scheme.

When the error on the velocity approximation is measured using the
energy norm, both formulations $\FO$ and $\FT$ provide the optimal
convergence rate, which scales as $\calO(\hh^{k})$ as expected from
Theorem~\ref{theorem:H1:estimate}, regardless of using the
non-enhanced or the enhanced versions of the method.
An optimal convergence rate, this time scaling like
$\calO(\hh^{k+1})$, is also visible for all the error curves of both
formulations in the $\LTWO$-norm as expected from
Theorem~\ref{theorem:L2:estimate} when using the enhanced definition
of the virtual element spaces and the projection operator $\Piz{\ks}$
in the right-hand side of the VEM.
Optimal convergence rates are also visible for both formulations $\FO$
and $\FT$ if $k\neq2$ when using the non-enhanced versions of the
virtual element spaces and the projection operator $\Piz{\kb}$ with
$\kb=max(0,k-2)$.
Instead, when $k=2$ the non-enhanced formulations $\FO$ and $\FT$
loose one order of convergence.
This fact is in agreement with the behavior previously noted
in~\cite{BeiraodaVeiga-Brezzi-Marini:2013}, where the optimal
convergence rate for $k=2$ was obtained by changing (in some sense,
``enhancing'') the construction of the right-hand side.
\ifARXIV
We also note that there is not a significant difference when we
compare the accuracy of the two formulation with respect to the number
of degrees of freedom, although we expect that formulation $\FT$ can
be more convenient than formulation $\FO$ as it has a smaller number of
degrees of freedom.
\fi

\PGRAPH{Free-divergence condition}
Regarding the approximation of the zero-divergence constraint,
the polynomial projection
$\PizP{k-1}(\DIV\uvh)$
is close to the machine precision in all elements $\P\in\Th$ for all
the formulations and meshes here considered.
Although we do not have a direct control on the divergence of the
virtual element approximation, a straightforward calculation using the
free-divergence condition for the ground truth, i.e., $\DIV\uv=0$, and
an application of Theorem~\ref{theorem:H1:estimate} yield
\begin{align*}
  \norm{\DIV\uvh}{0,\Omega}
  = \norm{\DIV(\uvh-\uv)}{0,\Omega}
  \leq \Cs\snorm{\uvh-\uv}{1,\Omega}
  \approx \mathcal{O}(\hh^{k}).
\end{align*}
So, we expect that the ``true'' divergence of the numerical
approximation $\uvh$ scales like $\mathcal{O}(\hh^{k})$ for
$\hh\to\mathbf{0}$.

Furthermore, we note that for both formulations $\FO$ and $\FT$ the
projections $\Piz{\ell}(\DIV\uvh)$, $\ell=k,k+1$, are computable from
the degrees of freedom of $\uvh$ when using the enhanced version of
the two spaces.
This fact allows us to post-process $\DIV\uvh$ and obtain the
polynomial projections $\PizP{k}(\DIV\uvh)$ and $\PizP{k+1}(\DIV\uvh)$
in every element $\P\in\Th$, which, in principle, could be better
approximations than $\PizP{k-1}(\DIV\uvh)$.
However, it is worth noting that $\PizP{k-1}(\DIV\uvh)$ is expected to
be zero (not considering rounding effects and the ill-conditioning of
the discretization) and a straightforward calculation using the
boundedness of $\Piz{\ell}$ and again the result of
Theorem~\ref{theorem:H1:estimate} shows that
\begin{align}
  \norm{\Piz{\ell}\DIV\uvh}{0,\Omega}
  =    \norm{\Piz{\ell}\DIV(\uvh-\uv)}{0,\Omega}
  \leq \Cs\norm{\DIV(\uvh-\uv)}{0,\Omega}
  \leq \Cs\snorm{\uvh-\uv}{1,\Omega}
  \approx \mathcal{O}(\hh^{k}),
  \label{eq:divg:k:rate}
\end{align}
where $\Cs\approx\norm{\Piz{\ell}}{}$.
So, we cannot expect a real gain by pursuing this route although this
estimate concerns with the worst case scenario and a convergence rate
to zero faster than $\mathcal{O}(\hh^{k})$ is still possible.
This effect is illustrated by the different error curves that are
obtained using the three mesh families $\MESH{1}$, $\MESH{2}$, and
$\MESH{3}$ and are shown in the log-log plots of
Figure~\ref{fig:divergence}.
In this figure, the three top panels are related to formulation $\FO$;
the solid (red) curves show the behavior of the $\LTWO$-norm of
$\Piz{k}(\DIV\uvh)$; the dotted (blue) curves show the behavior of the
$\LTWO$-norm of $\Piz{k+1}(\DIV\uvh)$.
Here, the deviation from zero looks decreasing like
$\mathcal{O}(\hh^k)$ in agreement with~\eqref{eq:divg:k:rate}.
The three bottom panels are related to formulation $\FT$;
the solid (black) curves show the behavior of the $\LTWO$-norm of
$\Piz{k}(\DIV\uvh)$; the dotted (blue) curves show the behavior of the
$\LTWO$-norm of $\Piz{k+1}(\DIV\uvh)$.
Here, the deviation from zero looks decreasing at a rate that is
closer to $\mathcal{O}(\hh^{k+1})$ for $k\neq2$ especially on mesh
families $\MESH{1}$ and $\MESH{3}$, and intermediate between $\hh^2$
and $\hh^3$ for $k=2$ when using mesh family $\MESH{2}$.


\section{Conclusions}
\label{sec:conclusions}

We studied two conforming virtual element formulations
for the numerical approximation of the Stokes
problem to unstructured meshes that work at any order of accuracy.
The components of the vector-valued unknown are approximated by using
variants of the conforming regular or enhanced virtual element spaces
that were originally introduced for the discretization of the Poisson
equation.
The scalar unknown is approximated by using discontinuous polynomials.
The stiffness bilinear form is approximated by using the orthogonal
polynomial projection of the gradients onto vector polynomials of
degree $k-1$ and adding a suitable stabilization term.
The zero divergence constraint is taken into account by projecting the
divergence equation onto the space of polynomials of degree $k-1$.
Our convergence analysis proves that the method is well-posed and
convergent and optimal convergence rates are obtained through error
estimates in the energy norm and in the $\LTWO$-norm.
Such optimal convergence rates are confirmed by numerical results on a
set of three different representative families of meshes.
These methods work well also in the lowest-order case (e.g., for the
polynomial order $k=1$) on triangular and square meshes,
which are well-known to be potentially unstable.
Moreover, our numerical experiments show that the divergence
constraint is satisfied at the machine precision level by the
orthogonal polynomial projection of the divergence of the approximate
velocity vector.


\section*{Acknowledgments}
GM was partially supported by the ERC Project CHANGE, which has
received funding from the European Research Council (ERC) under the
European Unions Horizon 2020 research and innovation programme (grant
agreement No 694515).






\end{document}